%% file: HFLandKhovanov.tex
\documentclass[11pt,twoside]{amsart}
\usepackage{etex}

\usepackage[margin=1.45in]{geometry}
\input{declarations.tex}

\let\OLDthebibliography\thebibliography
\renewcommand\thebibliography[1]{
  \OLDthebibliography{#1}
  \setlength{\parskip}{0pt}
  \setlength{\itemsep}{0pt plus 0.3ex}
}

\title[Computing cobordism maps in link Floer homology]
{Computing cobordism maps in link Floer homology and the reduced Khovanov TQFT}

\author{Andr\'as Juh\'asz}%
\address{Mathematical Institute, University of Oxford, Andrew Wiles Building,
Radcliffe Observatory Quarter, Woodstock Road, Oxford, OX2 6GG, UK}%
\email{juhasza@maths.ox.ac.uk}%

\author{Marco Marengon}%
\address{University of California, Los Angeles, Department of Mathematics,
520 Portola Plaza, Box 951555,
Los Angeles, CA 90095-1555 }%
\email{marengon@math.ucla.edu}%


\date{}

\newtheorem{proposition}{Proposition}[section]
\newtheorem{theorem}[proposition]{Theorem}
\newtheorem{lemma}[proposition]{Lemma}
\newtheorem{corollary}[proposition]{Corollary}
\newtheorem{conjecture}[proposition]{Conjecture}

\newtheorem{prop}[proposition]{Proposition}

\newtheorem{lem}[proposition]{Lemma}

\theoremstyle{definition}
\newtheorem{definition}[proposition]{Definition}
\newtheorem{defi}[proposition]{Definition}
\newtheorem{question}[proposition]{Question}
\newtheorem*{acknowledgements}{Acknowledgements}
\newtheorem*{organisation}{Organisation}

\theoremstyle{remark}
\newtheorem{remarkpro}[proposition]{Remark}
\newtheorem{remark}[proposition]{Remark}
\newtheorem{example}[proposition]{Example}
\newtheorem{rem}{Remark}

\begin{document}

\begin{abstract}
We study the maps induced on link Floer homology
by elementary decorated link cobordisms. We compute these for
births, deaths, stabilizations, and destabilizations,
and show that saddle cobordisms can be computed in terms of
maps in a decorated skein exact triangle that extends the oriented skein
exact triangle in knot Floer homology.
In particular, we completely determine the Alexander and Maslov grading shifts.

As a corollary, we compute the maps induced by elementary cobordisms between
unlinks. We show that
these give rise to a $(1+1)$-dimensional TQFT that coincides
with the reduced Khovanov TQFT.
Hence, when applied to the cube of resolutions of a marked link
diagram, it gives the complex defining the reduced
Khovanov homology of the knot.
Finally, we define a spectral sequence from (reduced)
Khovanov homology using these cobordism maps, and we
prove that it is an invariant of the (marked) link.
\end{abstract}

\maketitle
\thispagestyle{empty}

\section{Introduction}
\input{introduction.tex}

\section{Preliminaries}
\label{sec:preliminaries}
\input{preliminaries.tex}

\section{Elementary link cobordisms} \label{sec:elementary}
\label{sec:morse}
\input{morse.tex}

\subsection{Elementary link cobordisms between unlinks}
\label{sec:computations}
\input{ComputationsSFHcobmaps.tex}

\section{A disjoint union formula}
\label{sec:disjointunionformula}
\input{disjointunionformula.tex}

\section{A $(1+1)$-dimensional TQFT}
\label{sec:TQFT}
\input{TQFT.tex}

\section{A spectral sequence from Khovanov homology}
\label{sec:ss}
\input{ss.tex}

\bibliographystyle{amsalpha}
\bibliography{bibliography}

\end{document}

%% file: introduction.tex
This paper is devoted to developing tools for computing maps induced on link
Floer homology by decorated link cobordisms. We also exhibit a connection between
these maps and the reduced Khovanov TQFT.

Knot Floer homology, denoted by $\HFKh$, and defined independently by
Ozsv\'ath-Szab\'o~\cite{HFK1} and Rasmussen~\cite{HFK2},
is a knot invariant that categorifies the Alexander polynomial.
Khovanov~\cite{kh, redkh} introduced
a categorification of the Jones polynomial, called
\emph{Khovanov homology}, and a version of it for marked links,
called \emph{reduced Khovanov homology}.
Rasmussen~\cite{rasmussen2005knot} conjectured that there exists
a spectral sequence from reduced Khovanov homology
to knot Floer homology.

Link Floer homology, denoted by $\HFLh$,
and defined by Ozsv\'ath and Szab\'o~\cite{HFL},
is an extension of $\hat\HFK$ to links that
categorifies the multivariable Alexander polynomial.
The invariant $\HFLh$ is only functorial for decorated links
according to Dylan Thurston and the first author~\cite{naturality},
and moving the decorations around a link often induces a nontrivial automorphism of~$\HFLh$ by
the work of Sarkar~\cite{sarkar2011moving}; see also Zemke~\cite{quasistab}.
Hence, when considering link cobordisms, it is necessary to keep track of the decorations.
We say that $(X,F)$ is a link cobordism from the link $(Y_0,L_0)$ to $(Y_1,L_1)$
if $X$ is a cobordism from $Y_0$ to $Y_1$ and $F \subset X$ is a properly embedded
oriented surface with $\partial F = -L_0 \cup L_1$.
Loosely speaking, the decoration consists of a properly embedded 1-manifold $\sigma$
in the surface~$F$ that divides it into subsurfaces $R_+(\sigma)$ and~$R_-(\sigma)$.
For the precise definition, see Section~\ref{sec:decorated}.
Decorated link cobordisms functorially induce homomorphisms on link Floer homology
according to the work of the first author~\cite{cobordisms}.
In this paper we consider link cobordisms only in $X = Y \times I$ for a closed oriented
3-manifold~$Y$.

Via Morse theory, we can write every link cobordism in $Y \times I$ as a composition of
elementary link cobordisms. We define these elementary cobordisms in Section~\ref{sec:elementary},
and prove that they generate the category of decorated links in cylindrical 4-manifolds
in Proposition~\ref{prop:generation}.
In an elementary cobordism $(F,\sigma)$, the height function $h \colon Y \times I \to I$
has at most one generic critical point when restricted to $F$ and to $\sigma$,
and $\text{Crit}(h|_F) \subseteq \text{Crit}(h|_\sigma)$.
In a \emph{birth} cobordism, $h|_F$ has a local minimum.
In a \emph{saddle}, $h|_F$ has a critical point of index one and $h|_\sigma$ has a local maximum.
We distinguish \emph{merge} and \emph{split}
saddles depending on whether the number of link components decreases or increases by one, respectively.
In a \emph{death} cobordism, $h|_F$ has a local maximum.
A \emph{stabilisation} is an elementary cobordism where $h|_F$ has no critical points,
and $h|_\sigma$ has a local minimum, and a \emph{destabilisation} is when $h|_\sigma$ has a local
maximum. A (de)stabilisation is \emph{positive} (resp.~\emph{negative}) if the bigon component
of $F \setminus \sigma$ lies in $R_+(\sigma)$ (resp.~in $R_-(\sigma))$.
Finally, an \emph{isotopy} is a cobordism where $h|_F$ and $h|_\sigma$ have no critical points.
Our aim is to compute the maps induced by these elementary link cobordisms.

In Section~\ref{sec:morse}, we discuss elementary link cobordisms in full generality.
Let  $V = \F_2\langle B, T \rangle$ be the 2-dimensional bigraded vector space generated
by two homogeneous vectors $B$ (bottom-graded) and $T$ (top-graded),
where $B$ lives in Maslov grading~$-1/2$ and Alexander grading~$0$,
while $T$ lives in Maslov grading~$1/2$ and Alexander grading~$0$.
In Theorem~\ref{thm:births}, we show that if $\mc B$ is a birth cobordism
from $(L_0,P_0)$ to $(L_1,P_1)$, then there is an isomorphism
\[
\HFLh(L_1, P_1) \cong \HFLh(L_0,P_0) \otimes V
\]
such that
$F_{\mc B}(x) = x \otimes T$ for every $x \in \HFLh(L_0,P_0)$.

For a death cobordism $\mc D$ from $(L_0,P_0)$ to $(L_1,P_1)$,
we will show in Theorem~\ref{thm:deaths}
that there is an isomorphism
\[
\HFLh(L_0,P_0) \cong \HFLh(L_1,P_1) \otimes V
\]
such that $F_{\mc D}(x \otimes T) = 0$ and $F_{\mc D}(x \otimes B) = x$
for every $x \in \HFLh(L_1,P_1)$.

Now let $V' = \F_2\langle b, t \rangle$ be the 2-dimensional bigraded vector space generated
by two homogeneous vectors $b$ (bottom-graded) and $t$ (top-graded),
where $b$ lives in Maslov grading~$-1/2$ and Alexander grading~$-1/2$,
while $t$ lives in Maslov grading~$1/2$ and Alexander grading~$1/2$.
If $\mc X$ is a stabilisation from $(L,P_0)$ to $(L,P_1)$, then we will prove in Proposition~\ref{prop:stab}
that there is an isomorphism
\[
\HFLh(L,P_1) \cong \HFLh(L,P_0) \otimes V'
\]
such that $F_{\mc X}(x) = x \otimes t$ if $\mc X$ is positive and
$F_{\mc X}(x) = x \otimes b$ if $\mc X$ is negative for every
$x \in \HFLh(L,P_0)$.

We shall see in Proposition~\ref{prop:destab}
that if $\mc X$ is a destabilisation from $(L,P_0)$ to $(L,P_1)$, then there is
an isomorphism
\[
\HFLh(L,P_0) \cong \HFLh(L,P_1) \otimes V'
\]
such that $F_{\mc X}(x \otimes t) = 0$ and $F_{\mc X}(x \otimes b) = x$
if $\mc X$ is positive, and
$F_{\mc X}(x \otimes t) = x$ and $F_{\mc X}(x \otimes b) = 0$ if $\mc X$
is negative.

The maps associated to saddle link cobordisms are more complicated.
These maps consist of a contact gluing map and a special cobordism map.
We show that the special cobordism map consists of a single 4-dimensional
2-handle attachment, and prove that it coincides
with the appropriate map in the skein exact triangle of Ozsv\'ath and Szab\'o~\cite{HFK1},
or, more precisely, a decorated version of it for decorated links stated in Theorem~\ref{thm:skein}.
However, the definition of the map still involves the count of holomorphic triangles.
More precisely, we prove the following in Theorem~\ref{thm:exacttriangle}:
Suppose that we have a saddle cobordism from $(L_0,P_0)$ to $(L_+,P)$, and let
the link $(L,P)$ be as in Figure~\ref{fig:surgeryet}; i.e., obtained by adding a full right-handed twist
to~$L_+$ along the crossing disk. Then, by Theorem~\ref{thm:skein}, there is an exact triangle
\[
\dots \longrightarrow \HFLh(L,P) \stackrel{e}{\longrightarrow}
\HFLh(L_0,P_0) \stackrel{f}{\longrightarrow} \HFLh(L_+,P) \longrightarrow \dots,
\]
and the map~$f$ coincides with the cobordism map given by the saddle cobordism from~$(L_0,P_0)$ to~$(L_+,P)$.
Note that Theorem~\ref{thm:skein} is a decorated, and hence natural extension of the skein exact triangles~(7) and~(8)
of Ozsv\'ath and Szab\'o~\cite{HFK1}, and there is no need to distinguish between two cases depending on whether
the two strands of~$L_+$ intersecting the crossing disc belong to the same component or not.
Furthermore, it also applies to unoriented saddles and to multi-pointed links.
In Section~\ref{sec:OSzconstruction}, we also give a Heegaard diagrammatic description of the 2-handle map
when we have a saddle cobordism in~$S^3 \times I$.

Finally, the map induced by an \emph{isotopy} agrees with a diffeomorphism map
according to Subsection~\ref{sec:isotopy}.

In particular, in Theorem~\ref{thm:grading},
we compute the Alexander and Maslov grading shifts of any
decorated link cobordism $\mc X = (Y \times I, F,\sigma)$, when the gradings are defined.
The map $F_{\mc X}$ shifts the Maslov grading by $\chi(F)/2 + (\chi(R_+(\sigma)) - \chi(R_-(\sigma)))/2$,
and the Alexander grading by $(\chi(R_+(\sigma)) - \chi(R_-(\sigma)))/2$.
In particular, the shift in the $\delta$-grading, which is $-\chi(F)/2$,
only depends on the topology of $F$ and is independent of the decorations.
A key technical tool used for computing the grading shifts is Proposition~\ref{prop:pa},
which shows that the contact gluing map along a product annulus in a link complement
preserves both the Alexander and the Maslov gradings, when they are defined.

Zemke also defined maps induced by decorated link cobordisms in \cite{zemkefunctoriality}.
The grading shift of his maps, computed in~\cite{zemkegradings},
seem to agree with what we have obtained.

In the special case when we have an elementary cobordism between unlinks in~$S^3 \times I$,
we completely determine the induced maps on link Floer homology.
For a summary of our results, see Figure~\ref{fig:defcobordisms}.

We give a formula for the distant disjoint union of two decorated link cobordisms.
We actually develop a much more general formula that holds for connected sums of sutured
manifold cobordisms, see Theorem~\ref{thm:cs}. Using this, we can compute
maps induced by distant disjoint unions of the link cobordisms in Figure~\ref{fig:defcobordisms}.
By taking compositions and forgetting about the embedding of the link cobordisms into $S^3 \times I$,
we obtain a (1+1)-dimensional TQFT that,
surprisingly, coincides with the reduced Khovanov TQFT.
We explain this connection between link Floer homology and Khovanov homology next.

Khovanov's construction works as follows.
Given a link diagram, consider the corresponding
\emph{cube of resolutions}: Its vertices are \emph{resolutions} of~$K$,
obtained by smoothing all the crossings of the diagram,
and its edges correspond to pair-of-pants cobordisms between
resolutions that differ only at a single crossing.

Let $V$ be a $2$-dimensional vector space generated by two
vectors $v_+$ and $v_-$. (We are using Bar-Natan's
notation from \cite{bar2002khovanov}. Use the identification $v_- \mapsto X$
and $v_+ \mapsto \mathbf1$ to recover Khovanov's notation.) Khovanov associates
to every resolution of~$K$ the vector space $V^{\otimes n}$,
where $n$ is the number of components of the resolution,
and to every elementary cobordism a linear map. Khovanov's cobordism maps
are of two kinds: There is a map for the pair-of-pants
cobordism that merges two components, and a map for the
pair-of-pants cobordism that splits a component into two.
Such cobordism maps are given by a $(1+1)$-dimensional
TQFT, and they turn the vector
space obtained by direct summing the vector spaces
associated to all resolutions into a chain complex.
The homology of this complex is called
Khovanov homology.

The reduced version of Khovanov homology is obtained by
putting a basepoint on the knot. Consequently,
all the resolutions have a marked component.
By quotienting each vector space $V^{\otimes n}$ of the
complex by $\langle v_- \rangle \otimes V^{\otimes (n-1)}$,
where the first factor $V$ is associated to the
marked component of the resolution, one obtains a new complex
that defines \emph{reduced Khovanov homology}, which is
again an invariant of the knot.

We build a $(1+1)$-dimensional TQFT as follows:
We associate to each unlink $U_n$ the $\F_2$ vector space
$\HFLh(U_n)$, and to every cobordism the cobordism map in link
Floer homology. Our main result linking $\HFLh$ and Khovanov homology
is that the $(1+1)$-dimensional TQFT induced by $\HFLh$ is the same as
Khovanov's reduced TQFT; see Theorem~\ref{thm:triangle}.
Consequently, the homology of the complex obtained from
a cube of resolutions for a knot $K$ by applying the
functor $\HFLh$ is the reduced Khovanov homology of $K$ with $\F_2$
coefficients.

To prove Theorem~\ref{thm:triangle}, we need to overcome the following obstacle.
The cobordism maps obtained from the disjoint union formula
(i.e.; Theorem~\ref{thm:cs}) are expressed in some inconvenient basis of the
link Floer homologies.
Therefore, we actually need to fix a canonical basis of~$\HFLh(U_n)$,
which is achieved by marking a component of $U_n$.
We then express the cobordism maps that we have computed in
this basis, and we find that they are exactly the same as
Khovanov's maps induced by the same cobordisms.
This concludes the proof of Theorem~\ref{thm:triangle}.

The last section of the paper is devoted to a side topic:
We define a spectral sequence from Khovanov homology
and an analogous one from reduced Khovanov homology that
are obtained by putting a filtration on the Khovanov and
reduced Khovanov complexes, and by endowing them with higher differentials
coming from the cobordism maps. We call these filtered complexes
the \emph{Khovanov filtered complex} and the \emph{reduced Khovanov filtered
complex} associated to a link diagram $D$. By adjusting Bar-Natan's proof on the invariance of
Khovanov homology~\cite{bar2002khovanov}, we prove that the
spectral sequence defined by the (reduced) Khovanov filtered complex
is an invariant of the (marked) link up to isomorphism; see Theorem~\ref{thm:ss}.

This, in particular, implies that all the pages, which are bigraded $\F_2$ vector
spaces $\Kh_{p,q}^r$ (or $\rKh_{p,q}^r$), are (marked) link invariants.
We are aware that Baldwin, Hedden, and Lobb have also studied these spectral sequences.
It is an instance of a Khovanov-Floer theory, and hence natural and
functorial under link cobordisms.

The limits of the spectral sequences above are unknown.
Baldwin and Lobb made some computations to
understand the behaviour of the limit, but could not find any
non-trivial examples. This led them to conjecture that the
spectral sequence always collapses on the second page.

\begin{conjecture}[{Baldwin-Lobb}]
\label{conj}
The limit $\Kh_{p,q}^\infty$ of the spectral sequence is the
ordinary Khovanov homology $\Kh_{p,q}=\Kh_{p,q}^2$.
\end{conjecture}

\begin{organisation}
In Section~\ref{sec:preliminaries}, we recall the definitions of
sutured Floer homology and of the cobordism maps associated to
decorated link cobordisms.
In Section~\ref{sec:morse-pa}, we prove that the contact gluing map along
a product annulus in a link complement is an isomorphism (Lemma~\ref{lem:PhiAiso}),
and that it preserves the Alexander and Maslov gradings when they are defined (Proposition~\ref{lem:PhiAiso}).
In Section~\ref{sec:skein}, we prove the decorated skein exact triangle of Theorem~\ref{thm:skein}.
In Section~\ref{sec:morse}, we study the maps induced on $\HFLh$
by elementary link cobordisms.
In Section~\ref{sec:computations}, we compute the maps induced on~$\HFLh$
by the pair-of-pants and birth and death cobordisms in Figure~\ref{fig:defcobordisms}; see Theorem~\ref{thm:unlinkcob}.
In Section~\ref{sec:disjointunionformula}, we define the disjoint union
of two decorated link cobordisms, and we study the behaviour of the
cobordism maps under disjoint union. In Theorem~\ref{thm:cs},
we prove a formula for a more general case of connected sums of sutured manifold cobordisms.
In Section~\ref{sec:TQFT}, we define a canonical basis for the link
Floer homology of a marked unlink in~$S^3$. The rest of the section
is devoted to expressing in this canonical basis the TQFT that arises
by applying the functor~$\HFLh$ to cobordisms between unlinks
that are disjoint unions of identity cobordisms and cobordisms from
Figure~\ref{fig:defcobordisms}. By comparing the $(1+1)$-dimensional
TQFT that we obtain with Khovanov's TQFT, we prove Theorem~\ref{thm:triangle}.
Lastly, Section~\ref{sec:ss} is devoted to the definition and the proof
of invariance of the spectral sequence from Khovanov homology defined
using the cobordism maps on $\HFLh$; see Theorem~\ref{thm:ss}.
We conclude the section with some remarks
about the behaviour of the spectral sequence.
\end{organisation}

\begin{acknowledgements}
We thank John Baldwin, Matt Hedden, Tom Hockenhull, Joan Licata, Andrew Lobb,
Ciprian Manolescu, Tom Mrowka, Jacob Rasmussen, Ian Zemke,
and the anonymous referee for their comments and suggestions.

This project has received funding from the European Research Council (ERC) under the European
Union's Horizon 2020 research and innovation programme (grant agreement No.~674978).
The first author was supported by a Royal Society Research Fellowship.
The second author was supported by an EPSRC Doctoral Training Award
and LMS grant PMG 16-17 07.
The first author would also like to thank the Isaac Newton Institute for its hospitality.
\end{acknowledgements} 

%% file: preliminaries.tex
Throughout this paper, $\F_2$ denotes the field with two
elements. We will make extensive use of sutured Floer homology,
so we give a short introduction in the following subsection.
As it generalises the hat flavour of link Floer homology,
it provides us with suitable tools for studying link cobordisms.


\subsection{Sutured Floer homology and cobordism maps}

Sutured Floer homology \cite{SFH} is a module over~$\mb F_2$ associated to a
balanced sutured manifold. Thus, we first review what a (balanced) sutured manifold is.

\begin{defi}[{\cite[Definition 2.6]{suturedmanifolds}}]
A \emph{sutured manifold} is a compact oriented $3$-manifold $M$ with boundary,
together with a set $\gamma \subseteq \de M$ of pairwise disjoint
annuli~$A(\gamma)$ and tori~$T(\gamma)$. Furthermore, the interior of each component
of $A(\gamma)$ contains a homologically non-trivial oriented simple closed curve,
called a \emph{suture}. The union of the sutures is denoted by $s(\gamma)$.

Finally, every component of $R(\gamma)=\de M \setminus \Int(\gamma)$ is oriented
in such a way that $\de R(\gamma)$ is coherent with the sutures. Let $R_+(\gamma)$
(resp.~$R_-(\gamma)$) denote the components of $R(\gamma)$ where the normal vector
points outwards (resp.~inwards).
\end{defi}

In this paper, we will only consider sutured manifolds $(M,\gamma)$ where $T(\gamma) = \emptyset$.

\begin{defi}[{\cite[Definition 2.2]{SFH}}]
We say that a sutured manifold $(M,\gamma)$ is \emph{balanced} if
$M$ has no closed components, $\chi(R_+(\gamma))=\chi(R_-(\gamma))$,
and the map $\pi_0(A(\gamma))\to\pi_0(\de M)$ is surjective.
\end{defi}

We will sometimes view $\gamma$ as a ``thickened'' oriented 1-manifold,
so often we do not distinguish between $\gamma$ and $s(\gamma)$.
It shall be clear from the context which one we mean.

The sutured Floer homology~\cite{SFH} of a balanced sutured manifold $(M, \gamma)$
is a finite dimensional $\mb F_2$ vector space $\SFH(M,\gamma)$.
Similarly to Heegaard Floer homology, sutured Floer homology admits a splitting along
relative $\SpinC$ structures on $M$. These are defined in
\cite[Section 4]{SFH}, and form an affine space $\SpinC(M,\gamma)$ over $\H^2(M,\de M)$.
For each $\mf s \in \SpinC(M,\gamma)$, we have an invariant $\SFH(M,\gamma,\mf s)$,
such that
\[\SFH(M,\gamma) = \bigoplus_{\mf s \in \SpinC(M,\gamma)} \SFH(M,\gamma, \mf s).\]

We briefly recall the definition of relative $\SpinC$ structures.

\begin{definition}[{\cite[Definition 3.1]{cobordisms}}]
\label{def:spincMg}
Given a sutured manifold $(M, \gamma)$, we say that a vector field $v$ defined on a subset of $M$ containing $\de M$ is \emph{admissible} if it is nowhere vanishing, it points into $M$ along $R_-(\gamma)$, it points out of $M$ along $R_+(\gamma)$, and $v|_\gamma$ is tangent to $\de M$ and either points into $R_+(\gamma)$ or is positively tangent to $\gamma$ (we think of $\de M$ as a smooth surface, and of $\gamma$ as a 1-manifold).

Let $v$ and $w$ be admissible vector fields on $M$. We say that $v$ and $w$ are \emph{homologous}, and we write $v \sim w$, if there is a collection of balls $B \subseteq M$, one in each component of $M$, such that $v$ and $w$ are homotopic on $M \setminus B$ through admissible vector fields.
Then the set of \emph{relative $\SpinC$ structures}, denoted by $\SpinC(M, \gamma)$, is the set of homology classes of admissible vector fields on $M$.
\end{definition}

If $(M,\gamma)$ is balanced, $\SpinC(M,\gamma)$ is an affine space over $H^2(M,\de M)$. We will denote relative $\SpinC$ structures by $\s$, to distinguish them from ordinary $\SpinC$ structures on oriented $3$-manifolds,
which we will usually denote by~$\mf t$.

\begin{rem} \label{rem:v0}
Let $v_0$ be a fixed vector field on $\de M$ arising as $v|_{\de M}$ for some admissible vector field~$v$ on~$M$. We define $\text{Spin}^c_{v_0}(M,\gamma)$ as the set of nowhere vanishing vector fields on $M$ that restrict to $v_0$ on $\de M$, up to isotopy in the complement of a collection of balls through such vector fields.
Since the space of all possible $v_0$ is contractible, $\text{Spin}^c_{v_0}(M,\gamma)$ can be canonically identified with $\SpinC(M,\gamma)$. This was the approach taken in \cite{SFH}.
\end{rem}

In \cite{cobordisms}, the first author defined a map induced on $\SFH$
by a cobordism of balanced sutured manifolds. We briefly recall
the definition.

\begin{definition}[{\cite[Definition 2.3]{cobordisms}}]
Let $(M, \gamma)$ be a sutured manifold, and suppose that $\xi_0$
and $\xi_1$ are contact structures on $M$ such that $\de M$ is a
convex surface with dividing set $\gamma$ with respect to both
$\xi_0$ and $\xi_1$. Then we say that $\xi_0$ and $\xi_1$ are
\emph{equivalent} if there is a one-parameter family
$\left\{\,\xi_t \,\colon\,\, t \in I \,\right\}$ of contact structures
such that $\de M$ is convex with dividing set $\gamma$ with respect
to $\xi_t$ for every $t \in I$. In this case, we write $\xi_0 \sim \xi_1$,
and we denote by $[\xi]$ the equivalence class of a contact structure $\xi$.
\end{definition}

\begin{definition}[{\cite[Definition 2.4]{cobordisms}}]
Let $(M_0, \gamma_0)$ and $(M_1, \gamma_1)$ be sutured manifolds.
A \emph{cobordism} from $(M_0,\gamma_0)$ to $(M_1, \gamma_1)$ is
a triple $\mc W=(W, Z, [\xi])$, where
\begin{itemize}
\item{$W$ is a compact oriented $4$-manifold with boundary,}
\item{$Z \subseteq \de W$ is a compact, codimension-$0$ submanifold with boundary
of $\de W$ such that $\de W \setminus \Int(Z) = -M_0 \sqcup M_1$,}
\item{$\xi$ is a positive contact structure on $Z$ such that $\de Z$ is
a convex surface with dividing set $\gamma_i$ on $\de M_i$ for $i \in \{0,1\}$.}
\end{itemize}
A cobordism is called \emph{balanced} if both $(M_0, \gamma_0)$ and $(M_1, \gamma_1)$
are balanced.
\end{definition}

\begin{definition}[{\cite[Definition 2.7]{cobordisms}}]
Two cobordisms $\mc W=(W, Z, [\xi])$ and $\mc W'=(W', Z', [\xi'])$ from $(M_0, \gamma_0)$
to $(M_1, \gamma_1)$ are called \emph{equivalent} if there is an
orientation preserving diffeomorphism $\vphi: W \to W'$ such that $d(Z)=Z'$,
$d_*(\xi)=\xi'$, and $d|_{M_0 \cup M_1}=\id$.
\end{definition}

\begin{definition}[{\cite[Definition 10.4]{cobordisms}}]
A cobordism $\mc W = (W, Z, [\xi])$ from $(M_0, \gamma_0)$ to
$(N, \gamma_1)$ is called a \emph{boundary cobordism} if
\begin{itemize}
\item $\mc W$ is balanced,
\item $N$ is parallel to $M_0 \cup (-Z)$,
\item we are given a deformation retraction $r \colon W \times [0,1] \to M_0 \cup (-Z)$
such that $r_0|_W = \id_W$ and $r_1|_N$ is an orientation preserving
diffeomorphism from $N$ to $M_0 \cup (-Z)$.
\end{itemize}
\end{definition}

\begin{definition}[{\cite[Definition 5.1]{cobordisms}}]
We say that a cobordism $\mc W = (W, Z, [\xi])$ from $(M_0, \gamma_0)$
to $(M_1, \gamma_1)$ is \emph{special} if
\begin{itemize}
\item $\mc W$ is balanced,
\item $\de M_0 = \de M_1$, and $Z = \de M_0 \times I$ is the trivial
cobordism between them,
\item $\xi$ is an $I$-invariant contact structure on $Z$ such that each
$\de M_0 \times \left\{t\right\}$ is a convex surface with dividing set
$\gamma_0 \times \left\{t\right\}$ for every $t \in I$ with respect to
the contact vector field $\de/\de t$.
\end{itemize}
In particular, it follows from the last condition that $\gamma_0 = \gamma_1$.
\end{definition}

\begin{remarkpro}
\label{rem:decomposition}
According to \cite[Definition 10.1]{cobordisms}, every sutured manifold cobordism
can be seen as the composition of a boundary cobordism and a special cobordism.
Let $\mc W = (W,Z,[\xi])$ be a balanced cobordism from $(M_0, \gamma_0)$
to $(M_1, \gamma_1)$. Let $(N, \gamma_1)$ be the sutured manifold
$(M_0 \cup (-Z), \gamma_1)$. Then we can think of the cobordism $\mc W$
as a composition $\Ws \circ \Wb$, where $\Wb$ is a boundary cobordism
from $(M_0, \gamma_0)$ to $(N, \gamma_1)$ and $\Ws$ is a special cobordism
from $(N, \gamma_1)$ to $(M_1, \gamma_1)$.
\end{remarkpro}

By~\cite{cobordisms}, every balanced cobordism~$\mc W$ from
$(M_0,\gamma_0)$ to $(M_1,\gamma_1)$ induces a linear map from
$\SFH(M_0,\gamma_0)$ to $\SFH(M_1, \gamma_1)$ that is independent of the
equivalence class of~$\mc W$. As in Remark~\ref{rem:decomposition},
choose a decomposition $\Ws \circ \Wb$ of~$\mc W$.
By~\cite[Section~5]{cobordisms}, the special cobordism
$\Ws$ naturally yields a map
\[
F_{\mc W^s} : \SFH(N, \gamma_1) \to \SFH(M_1, \gamma_1).
\]
This is constructed via composing maps associated to handle attachments,
similarly to the case of cobordisms of closed 3-manifolds~\cite{holomorphictriangles}.

If every component of $N \setminus {M_0}$ intersects $\de N$,
then, by \cite[Theorem~1.1]{gluingmap}, there is a contact gluing map
\[
\Phi_{-\xi}: \SFH(M_0, \gamma_0) \to \SFH(N, \gamma_1),
\]
associated to the inclusion of the sutured manifold $(-M_0, -\gamma_0) \subseteq (-N, -\gamma_1)$
and the contact structure $\xi$ on $(-N) \setminus \mbox{Int}{(-M_0)}$.
The map associated to the cobordism~$\mc W$ is then defined as
\[
F_\mc W = F_{\mc W^s} \circ \Phi_{-\xi}: \SFH(M_0,\gamma_0) \to \SFH(M_1, \gamma_1).
\]
The map $\Phi_{-\xi}$ is usually called the \emph{gluing map}, whereas
the map $F_{\mc W^s}$ is usually called the \emph{handle attachment map},
\emph{surgery map}, or \emph{special cobordism map}.

Recall that we supposed that there are no components of $N\setminus {M_0}$
that do not meet $\de N$; these are called \emph{isolated components}.
This requirement was needed for the definition of the gluing map, and
also because otherwise $N$ might not be balanced.
In the presence of isolated components, the map associated to
a cobordism $\mc W$ is defined as follows: Remove a standard contact ball $B^3$
(i.e.; a ball with a single suture~$\gamma_{B^3}$ on the boundary) from each isolated
component of $N$, and add them to $M_1$. What we get is a cobordism $\mc W'$
from $(M_0, \gamma_0)$ to $\left(M_1 \sqcup \coprod B^3, \gamma_1 \sqcup \coprod \gamma_{B^3}\right)$
with no isolated components. By composing the already defined map
\[F_{\mc W'}: \SFH(M_0, \gamma_0) \to \SFH \left(M_1 \sqcup \coprod B^3, \gamma_1 \sqcup \coprod \gamma_{B^3} \right)\]
with the natural isomorphism
\[
\SFH \left(M_1 \sqcup \coprod B^3, \gamma_1 \sqcup \coprod \gamma_{B^3} \right) \xrightarrow{\sim}
\SFH(M_1, \gamma_1) \otimes \bigotimes \mathbb{F}_2
\xrightarrow{\sim} \SFH(M_1, \gamma_1),
\]
we obtain the cobordism map $F_{\mc W}$.

The main properties of the cobordism maps are summarised
in the following theorem.
Here $\BSut$ denotes the category of balanced sutured manifolds
and equivalence classes of cobordisms, whereas $\Vect$ denotes
the category of $\mb F_2$ vector spaces.

\begin{theorem}[{\cite[Theorem 11.11]{cobordisms}}]
Sutured Floer homology, together with the above cobordism maps,
defines a functor \emph{$\SFH \colon \BSut \to \Vect$} that is a (3+1)-dimensional
TQFT in the sense of Atiyah~\cite{atiyah1988topological}
and Blanchet-Turaev~\cite{blanchet2006axiomatic}.
\end{theorem}


\subsection{Decorated link cobordisms}
\label{sec:decorated}

Sutured Floer homology is particularly well-suited to defining maps
induced on link Floer homology by decorated link cobordisms. We recall
the necessary definitions, starting with reviewing the real
blow-up construction.

\begin{definition}
Suppose that $M$ is a smooth manifold, and let $L \subset M$ be a properly embedded submanifold.
For every $p \in L$, let $N_pL = T_pM/T_pL$ be the fiber of the normal bundle of $L$ over $p$,
and let $UN_pL = (N_pL \setminus \{0\})/\R_+$ be the fiber of the unit normal bundle of $L$ over $p$.
Then the \emph{(spherical) blowup} of $M$ along $L$, denoted by $\Bl_L(M)$, is a manifold with
boundary obtained from $M$ by replacing each point
$p \in L$ by $UN_pL$. There is a natural projection $\Bl_L(M) \to M$. For further details, see
Arone and Kankaanrinta~\cite{AK}.
\end{definition}

We now review decorated links.

\begin{definition}[{\cite[Definition~4.4]{cobordisms}}]
\label{def:dl}
A \emph{decorated link} is a triple $(Y,L,P)$, where $L$ is a
non-empty link in the connected oriented 3-manifold $Y$, and
$P \subset L$ is a finite set of points. We require that, for
every component~$L_0$ of~$L$, the number $|L_0 \cap P|$ is positive and even.
Furthermore, we are given a decomposition of~$L$ into compact
$1$-manifolds~$R_+(P)$ and~$R_-(P)$ such that $R_+(P) \cap R_-(P) = P$.

We can canonically assign a balanced sutured manifold
$Y(L,P) = (M,\gamma)$ to every decorated link $(Y,L,P)$, as follows.
Let $M = \Bl_L(Y)$ and $\gamma = \bigcup_{p \in P} UN_pL$. Furthermore,
\[
R_\pm(\gamma) := \bigcup_{x \in R_\pm(P)} UN_xL,
\]
oriented as $\pm \partial M$, and we orient~$\gamma$ as~$\partial R_+(\gamma)$. Note that $Y(L,P)$ is independent of the orientation of $L$.
\end{definition}

\begin{definition}[{\cite[Definition~4.2]{cobordisms}}]
A \emph{surface with divides} $(S,\sigma)$
is a compact orientable surface~$S$, possibly with boundary, together with a properly embedded $1$-manifold~$\sigma$
that divides~$S$ into two compact subsurfaces that meet along~$\sigma$.
\end{definition}

\begin{definition}[{\cite[Definition~4.1]{cobordisms}}]
For $i \in \{0,1\}$, let $Y_i$ be a connected, oriented 3-manifold, and let $L_i$ be a non-empty link in~$Y_i$.
Then a \emph{link cobordism} from $(Y_0,L_0)$ to $(Y_1,L_1)$ is a pair~$(X,F)$, where
\begin{enumerate}
\item $X$ is a connected, orientable cobordism from $Y_0$ to $Y_1$,
\item $F$ is a properly embedded, compact, oriented surface in $X$,
\item $\partial F = L_0 \cup L_1$.
\end{enumerate}
\end{definition}

We are now ready to define decorated link cobordisms.

\begin{definition}[{\cite[Definition~4.5]{cobordisms}}]
\label{def:dlc}
We say that the triple $\mc X = (X,F,\sigma)$ is a \emph{decorated link cobordism} from $(Y_0,L_0,P_0)$ to $(Y_1,L_1,P_1)$ if
\begin{enumerate}
\item $(X,F)$ is a link cobordism from $(Y_0,L_0)$ to $(Y_1,L_1)$,
\item $(F,\sigma)$ is a surface with divides such that the map
\[
\pi_0(\partial\sigma) \to \pi_0((L_0 \setminus P_0) \cup (L_1 \setminus P_1))
\]
is a bijection,
\item 
we can orient each component~$R$ of~$F \setminus \sigma$ such that whenever~$\partial\closure{R}$
crosses a point of~$P_0$, it goes from~$R_+(P_0)$ to~$R_-(P_0)$,
and whenever it crosses a point of~$P_1$, it goes from~$R_-(P_1)$ to~$R_+(P_1)$,
\item if $F_0$ is a closed component of $F$, then $\sigma \cap F_0 \neq \emptyset$.
\end{enumerate}
Furthermore, we say that two decorated link cobordisms $\mc X = (X, F, \sigma)$
and $\mc X'=(X', F', \sigma')$ from $(Y_0,L_0,P_0)$ to $(Y_1,L_1,P_1)$ are \emph{equivalent}
if there is an orientation-preserving diffeomorphism $\varphi: X \to X'$
such that $\varphi(F) = F'$, $\varphi(\sigma)=\sigma'$, $\varphi|_{Y_0}=\id_{Y_0}$ and $\varphi|_{Y_1}=\id_{Y_1}$.
\end{definition}

Decorated links and equivalence classes of decorated link cobordisms
form a category that is denoted by $\DLink$.
Note that we do not require the surface $F$ to be oriented, just orientable;
see Section~\ref{sec:origrad}.

According to \cite[Definition~4.9]{cobordisms}, there is a natural sutured manifold cobordism
complementary to a decorated link cobordism whose construction we recall in the next definition.
For this purpose, we first discuss $S^1$-invariant contact structures on circle bundles;
see also~\cite[Section~4]{cobordisms}.
Let $\pi \colon M \to F$ be a principal circle bundle over a compact oriented surface~$F$.
An $S^1$-invariant contact structure $\xi$ on $M$ determines a dividing set $\sigma$ on the
base $F$, by requiring that $x \in \sigma$ if and only if $\xi$ is tangent to $\pi^{-1}(x)$,
and a splitting of $F$ as $R_+(\sigma) \cup R_-(\sigma)$. The image of any local
section of~$\pi$ is a convex surface with dividing set projecting onto~$\sigma$.
In the opposite direction, according to Lutz~\cite{Lutz} and Honda~\cite[Theorem~2.11 and Section~4]{Ko},
given a dividing set $\sigma$ on $F$ that intersects each component of $F$
non-trivially and divides $F$ into subsurfaces $R_+(\sigma)$ and $R_-(\sigma)$,
there is a unique $S^1$-invariant contact structure~$\xi_\sigma$ on $M$, up to isotopy,
such that the dividing set associated to $\xi_\sigma$ is exactly~$\sigma$,
the coorientation of~$\xi_\sigma$ induces the splitting~$R_\pm(\sigma)$, and the
boundary~$\de M$ is convex.

\begin{definition}[{\cite[Definition 4.9]{cobordisms}}] \label{def:linkcob}
Let $(X,F,\sigma)$ be a decorated link cobordism from $(Y_0,L_0,P_0)$ to~$(Y_1,L_1,P_1)$.
Then we define the sutured cobordism $\W = \W(X,F,\sigma)$ from $Y_0(L_0, P_0)$ to $Y_1(L_1, P_1)$ as follows.
Choose an arbitrary splitting of~$F$ into~$R_+(\sigma)$ and~$R_-(\sigma)$ such that $R_+(\sigma) \cap R_-(\sigma) = \sigma$,
and orient~$F$ such that~$\partial R_+(\sigma)$ (with $R_+(\sigma)$ oriented as a subsurface of $F$)
crosses~$P_0$ from~$R_+(P_0)$ to~$R_-(P_0)$ and~$P_1$ from~$R_-(P_1)$ to~$R_+(P_1)$.
Then~$\W$ is defined to be the triple $(W,Z,[\xi])$, where $W = \Bl_F(X)$ and $Z = UNF$, oriented
as a submanifold of~$\partial W$, finally $\xi = \xi_{\sigma}$ is an $S^1$-invariant contact structure
with dividing set~$\sigma$ on~$F$ and convex boundary~$\partial Z$ with dividing set projecting to~$P_0 \cup P_1$.
\end{definition}

The contact structure $\xi$ is independent of the splitting of~$F$ into~$R_+(\sigma)$ and~$R_-(\sigma)$.
According to \cite[Proposition~4.10]{cobordisms}, if $(X, F, \sigma)$ and
$(X', F', \sigma')$ are equivalent, so are the cobordisms
of sutured manifolds associated to them.
Therefore, we obtain a functor
\[
\mc W: \DLink \to \BSut.
\]
Notice that $\mc W(Y, L, P)$ is what we called $Y(L,P)$.
By composing with $\SFH$, we obtain a functor from $\DLink$ to $\Vect$.
This functor is a generalisation of the \emph{link Floer homology} functor $\HFLh$
due to the following proposition.

\begin{proposition}[{\cite[Proposition 9.2]{SFH}}]
\label{prop:HFL}
If $(Y,L,P)$ is a decorated link such that $P$ intersects each
component of $L$ in exactly 2 points, then
\[
\HFLh(Y, L) \cong \SFH(Y(L, P)).
\]
\end{proposition}

The above proposition motivates the following definition; cf.~\cite[p.~957]{cobordisms}.

\begin{definition}
\label{def:HFL}
We define the functor
\[
\HFLh \colon \DLink \to \Vect
\]
to be the composition $\SFH \circ \W$.
\end{definition}

\begin{remark}
Given a decorated link cobordism, with a slight abuse of notation, we
will also denote the cobordism of sutured manifolds associated
to it with the same letter.
\end{remark}

\subsection{Orientations and gradings}
\label{sec:origrad}

Note that, in Definition~\ref{def:dlc}, we require the surface~$F$ to be \emph{orientable} rather than \emph{oriented}. Indeed, the functor~$\mc W$ introduced in the previous section does not depend on the orientation chosen on the links~$L$ and the surfaces~$F$.
Therefore, the cobordism map induced on sutured Floer homology is independent of the orientations chosen.
However, starting from Section~\ref{sec:morse}, when we consider the Alexander and the Maslov gradings on~$\HFLh$,
we need to specify orientations of~$L$ and~$F$.

An alternative piece of data usually given to define gradings is a labelling of the points of~$P$ by~$w$ or~$z$.
This was the original approach taken by Ozsv\'ath-Sza\-b\'o~\cite{HFK1, HFL} and Rasmussen~\cite{HFK2},
and was also used by Zemke~\cite{zemkefunctoriality}.
It is equivalent to choosing an orientation of a decorated link.
We now explain why this is the case, and we fix the conventions we are going to use in this paper.

Given a non-empty link $L$, and a finite set of points $P$ on $L$ such that, for every component $L'$ of $L$,
the number $|P \cap L'|$ is even and non-zero, one can define the following additional structures on $(L,P)$:
\begin{enumerate}
\item \label{item:L1} A splitting $P = \boldsymbol{w} \sqcup \boldsymbol{z}$ such that on each component of $L$ points
    in~$\boldsymbol{w}$ and~$\boldsymbol{z}$ alternate.
\item \label{item:L2} A splitting $L \setminus P = R_-(P) \cup R_+(P)$ such that $(L,P)$ is a decorated link.
\item \label{item:L3} An orientation of $L$.
\end{enumerate}
We say that \eqref{item:L1}, \eqref{item:L2}, and \eqref{item:L3} are \emph{compatible} if $L$ is oriented from $R_-(P)$ to $R_+(P)$ when it crosses $\boldsymbol{z}$. Any two structures among \eqref{item:L1}, \eqref{item:L2}, and \eqref{item:L3} determine the third. Note that this implies that an oriented decorated link can be represented by a multi-pointed Heegaard diagram $(\Sigma, \ba, \bb, \boldsymbol{w}, \boldsymbol{z})$.

Analogously, given a surface with divides $(F, \sigma)$ such that $F \setminus \sigma$ has no closed components, one can define the following additional structures:
\begin{enumerate}
\item \label{item:F1} A splitting $F \setminus \sigma = F_w \sqcup F_z$ such that $\de F_w \cap \de F_z = \sigma$
($F_w$ and $F_z$ are called \emph{type-$\boldsymbol{w}$} and \emph{type-$\boldsymbol{z}$} regions in \cite{zemkefunctoriality}, respectively).
\item \label{item:F2} An orientation of each component of $F \setminus \sigma$
     such that $\de(F \setminus \sigma) = 2\sigma$, for some choice of orientation of $\sigma$.
\item \label{item:F3} An orientation of $F$.
\end{enumerate}
We say that \eqref{item:F1}, \eqref{item:F2}, and \eqref{item:F3} are \emph{compatible} if $F_w$ is oriented as a submanifold of $F$. Any two structures among \eqref{item:F1}, \eqref{item:F2}, and \eqref{item:F3} determine the third.

We now recall the definitions and basic properties of gradings on $\HFLh$.
Suppose that $(Y, L, P)$ is a null-homologous \emph{oriented} decorated link. 
Let $(M, \gamma)$ denote the complementary sutured manifold $Y(L,P)$.
Given a relative $\SpinC$ structure $\s \in \SpinC(M, \gamma)$
and a Seifert surface $S$ for $L$, its \emph{Alexander grading} is
\[
\A_S(\s) := \frac12 \langle c_1(\s, t) , [S] \rangle,
\]
where $t$ is the trivialisation of $v_0^\perp$ given by the meridional direction.
As we explained in \cite[Section~5.1]{concordance}, this agrees with the original definition due to
Ozsv\'ath and Szab\'o~\cite{HFK1}. If $Y$ is a rational homology 3-sphere,
then the Alexander grading does not depend on the choice of~$S$.
Given a multi-pointed Heegaard diagram $\mc H = (\Sigma, \ba, \bb, \bw, \bz)$ for $(Y, L, P)$,
one can associate to each generator $\x \in \T_{\ba} \cap \T_{\bb}$ a relative $\SpinC$ structure $\s(\x) \in \SpinC(Y(L,P))$,
and define $\A_S(\x) := \A_S(\s(\x))$.
If two generators $\x$ and $\y$ are connected by a domain $D$ on $\mc H$, then, by \cite[Lemma~2.5]{HFK1} and \cite[p.~25]{HFK2},
\[
\A_S(\x) - \A_S(\y) = n_{\bz}(D) - n_{\bw}(D).
\]

Let $\x$ and $\y$ be generators lying in the same $\SpinC$ structure $\mf t$ on $Y$.
Then one can define their relative $\Z_{\mf d(\mf t)}$ grading, where $\mf d(\mf t)$ is
the divisibility of~$c_1({\mf t})$, by choosing a domain $D$ connecting $\x$ to $\y$, and setting
\[
\gr(\x, \y) = \mi(D) - 2 n_{\bw}(D) \bmod \mf d(\mf t).
\]
If $\mf t \in \SpinC(Y)$ is torsion (i.e., it has torsion first Chern class),
then the relative $\Z$ grading can be lifted to an absolute $\Q$ grading,
denoted by $\agr$. This grading is known as the \emph{homological} or \emph{Maslov} grading.


\section{Product annuli in link complements}
\label{sec:morse-pa}

We now study the behaviour of sutured Floer homology under product annulus decompositions.
For more general sutured manifold decompositions, see~\cite{surface, gluingmap}.

\begin{definition}[{\cite[Definition~3.1]{suturedmanifolds}}]
Let $(M,\gamma)$ be a sutured manifold without toroidal sutures. A \emph{decomposing surface} is
a properly embedded oriented surface~$S$ in~$M$ such that each component~$\lambda$
of~$S \cap \gamma$ is either a properly embedded non-separating arc in~$\gamma$
such that $|\lambda \cap s(\gamma)| = 1$, or it is a simple closed curve in
an annular component~$A$ of~$\gamma$ in the same homology class as~$A \cap s(\gamma)$.
\end{definition}

As explained by Gabai~\cite[Definition~3.1]{suturedmanifolds},
a decomposing surface~$S$ in $(M,\gamma)$ defines a sutured manifold decomposition
\[
(M, \gamma) \overset{S}{\rightsquigarrow} (M', \gamma'),
\]
where $M' = M \setminus \Int(\nbd(S))$, the new sutures are
\[
\gamma' = (\gamma \cap M') \cup \nbd(S'_+ \cap R_-(\gamma)) \cup \nbd(S'_- \cap R_+(\gamma)),
\]
and $R_\pm(\gamma') = ((R_\pm(\gamma) \cap M') \cup S'_\pm) \setminus \Int(\gamma')$,
where $S'_+$ ($S'_-$) is the component of $\nbd(S) \cap M'$ whose normal vector points out of (into) $M'$.

\begin{definition}[{\cite[Definition~2.1]{Gabai1}}]
A \emph{product annulus} $A$ is a properly embedded annulus in a sutured manifold $(M, \gamma)$ such that one component of $\partial A$ is contained in $R_-(\gamma)$ and the other one is contained in $R_+(\gamma)$.
\end{definition}

If $(M, \gamma) \overset{A}{\rightsquigarrow} (M', \gamma')$ is a sutured manifold decomposition
along a product annulus~$A$, then the resulting sutured manifold $(M',\gamma')$ can
be described by $M' = M \setminus \Int(\nbd(A))$ and $\gamma' = \gamma \cup (\nbd(A) \cap M')$,
while $R_\pm(\gamma') = R_\pm(\gamma) \cap M'$.

\begin{definition}[{\cite[Definition~1.2]{surface}}]
Suppose that $R$ is a compact, oriented surface with non-empty boundary.
Let $c$ be an oriented simple closed curve in $\Int(R)$. If $[c]=0$ in $H_1(R)$,
then $R \setminus c$ can be written as $R_1 \sqcup R_2$,
where $R_1$ is the component of $R \setminus c$
that is disjoint from $\de R$ and satisfies $\de R_1 = c$.
We call $R_1$ the \emph{interior} and $R_2$ the \emph{exterior} of $c$.

We say that the curve $c$ is \emph{boundary-coherent} if either $[c] \neq 0$ in $H_1(R)$, or,
if $[c] = 0$ in $H_1(R)$ and $c$ is oriented as the boundary of its interior.
\end{definition}

\begin{lemma}[{\cite[Lemma~8.9]{surface}}]
\label{lem:pa0}
Suppose that $(M,\gamma)$ is a balanced sutured manifold such that $H_2(M) = 0$. Let $A \subset (M, \gamma)$ be a product annulus such that at least one component of $\de A$ is non-zero in $H_1(R(\gamma))$ or both components are boundary-coherent in $R(\gamma)$. Then $A$ defines a sutured manifold decomposition $(M, \gamma) \overset{A}{\rightsquigarrow} (M', \gamma')$ such that
\[
\SFH(M', \gamma') \cong \SFH(M, \gamma).
\]
\end{lemma}

We will need to study product annulus decompositions of link complements
with non-vanishing second homology, and our goal is to extend the above
lemma to this case, and to make the isomorphism canonical.

Suppose that $S$ is a decomposing surface in the balanced sutured manifold $(M,\gamma)$ such that
each component of $\partial S$ intersects $s(\gamma)$ nontrivially,
and let $(M, \gamma) \overset{S}{\rightsquigarrow} (M',\gamma')$
be the corresponding sutured manifold decomposition.
Then Honda, Kazez, and Mati\'c~\cite[Section~6]{HKMcontact} constructed a monomorphism
\[
\Phi_S \colon \SFH(M', \gamma') \to \SFH(M, \gamma).
\]
It coincides with the gluing map $\Phi$ of~\cite[Theorem~1.3]{gluingmap}
according to~\cite[Proposition~6.4]{gluingmap}.
The map $\Phi$ is defined as follows.
By slightly shrinking~$M'$, we can view it as a submanifold of $\Int(M)$.
Let $\xi$ be a contact structure on $-M \setminus \Int(M')$ with convex boundary
and dividing set $-(\gamma \cup \gamma')$. We require $\xi$ to be translation invariant in $\nbd(S)$
and in $\nbd(\partial M)$, where $\nbd(S) \cup \nbd(\partial M) = M \setminus \Int(M')$.
Furthermore, $S$ is convex with Legendrian boundary and with boundary-parallel dividing set.
Then $\Phi$ is defined to be the contact gluing map $\Phi_\xi$ that also features in
the construction of the sutured manifold cobordism maps.

If $A$ is a product annulus in the balanced sutured manifold $(M,\gamma)$,
then the gluing map $\Phi_A$ is not defined as $A \cap \gamma = \emptyset$.

\begin{definition} \label{def:arcs}
We said that the embedded arcs $a_\pm \subset R_\pm(\gamma)$ are \emph{adapted} to~$A$ if
$\{a_\pm(0)\} = a_\pm \cap A \cap R_\pm(\gamma)$
and $\{a_\pm(1)\} = a_\pm \cap \gamma$, and write $a = a_+ \cup a_-$.
We further require that $a_\pm$ is transverse to $\partial A$ in $\partial M$,
and we have the equality $a_\pm \cdot \partial A = a_\pm \cdot \g$ of intersection signs.
\end{definition}

If $a = a_+ \cup a_-$ are arcs adapted to the product annulus~$A$,
then we can perform finger moves on~$A$ along~$a_+$ and~$a_-$ to obtain
a new decomposing surface $A_a$. Let $(M',\gamma')$ be the result
of decomposing $(M,\gamma)$ along~$A$, and $(M_a,\gamma_a)$ the result of
decomposing $(M,\gamma)$ along~$A_a$.
Both components of $\partial A_a$ intersect the sutures,
hence $A_a$ induces a gluing map  $\Phi_{A_a} \colon \SFH(M_a,\gamma_a) \to \SFH(M,\gamma)$.
There is an isotopically unique diffeomorphism $d_a \colon (M',\gamma') \to (M_a,\gamma_a)$;
cf.~\cite[Lemma~4.5]{surface}.
We let
\[
\Phi_{A,a} := \Phi_{A_a} \circ (d_a)_* \colon \SFH(M',\gamma') \to \SFH(M,\gamma).
\]
It follows from the naturality of the gluing maps~\cite[Section~5]{gluingmap}
that if $b = b_+ \cup b_-$ is another pair of arcs adapted to $A$ that is
isotopic to~$a$ through adapted arcs, then $\Phi_{A,a} = \Phi_{A,b}$.

Suppose that the sutured manifold $(M,\gamma)$ is complementary to a decorated link in a 3-manifold.
If $A$ is a product annulus in $(M,\gamma)$, then there is a unique isotopy class of
arcs adapted to~$A$. Indeed, each component of~$R(\gamma)$ is an annulus, so each component of~$\partial A$
lies between two oppositely oriented sutures, and the arcs~$a_+$ and~$a_-$ are uniquely determined
up to isotopy by the coorientation of~$A$. Hence, when $A$ is a product annulus in a link
complement, we write $\Phi_A := \Phi_{A,a}$ for any adapted pair of arcs~$a= a_+ \cup a_-$.
Note that if we decompose a link complement along a product annulus,
we always obtain another link complement.

\begin{lemma}
\label{lem:PhiAiso}
Let $(Y, L, P)$ and $(Y', L', P')$ be decorated links
with complementary sutured manifolds $(M,\gamma)$ and $(M',\gamma')$, respectively,
such that there is a product annulus decomposition $(M, \gamma) \overset{A}{\rightsquigarrow} (M', \gamma')$.
Then the map
\[
\Phi_A \colon \SFH(M', \gamma') \to \SFH(M, \gamma)
\]
is an isomorphism.
\end{lemma}

Before proving the lemma, we recall \cite[Definition~4.3]{surface}.

\begin{definition}
A balanced diagram \emph{adapted} to the decomposing surface $S$ in $(M, \gamma)$ is a quadruple $(\Sigma, \ba, \bb, C)$ that satisfies the following conditions: The triple
$(\Sigma, \ba, \bb)$ is a balanced diagram for $(M, \gamma)$; furthermore, $C$ is a compact subsurface of $\Sigma$ whose boundary is a union of polygonal curves such that $C \cap \de \Sigma$ is the set of vertices of $\partial C$.
We are also given a decomposition $\de C = A \cup B$, where both $A$ and $B$ are unions of pairwise disjoint edges of $\partial C$,
and such that $\ba \cap B = \emptyset$ and $\bb \cap A = \emptyset$.
Finally, $S$ is given, up to equivalence, by smoothing the corners of the surface
\[
(C \times \{1/2\}) \cup (A \times [1/2,1]) \cup (B \times [0, 1/2]) \subset (M,\gamma).
\]
The orientation of $S$ is coherent with orientation of $C \subset \Sigma$.
We call a tuple $(\Sigma, \ba, \bb, C)$ satisfying the above conditions a \emph{surface diagram}.
\end{definition}

\begin{proof}[{Proof of Lemma \ref{lem:PhiAiso}}]
Choose a pair of arcs $a$ adapted to $A$. By definition, $\Phi_A = \Phi_{A_a} \circ (d_a)_*$.
The map $(d_a)_*$ is an isomorphism and $\Phi_{A_a}$ is injective, so $\Phi_A$ is also injective.
Hence, the result follows once we show that
\[
\SFH(M, \gamma) \cong \SFH(M', \gamma').
\]
By Lemma~\ref{lem:pa0}, this holds when $H_2(M)=0$.
The technical assumption $H_2(M) = 0$ ensures that there is
an \emph{admissible} surface diagram~$(\mc H, C)$ adapted to the product annulus~$A$
such that $\CF(\mc H, \mf s) = 0$ for all $\mf s \not \in O_A$,
where $O_A$ is the set of outer $\SpinC$ structures with respect to~$A$; see \cite[Definition~1.1]{surface}.
Although $H_2(M)$ may be non-zero in our case, we can still find an admissible surface diagram adapted to~$A$,
as follows.

Let $P' \setminus P = \{p,q\}$.
Given a Heegaard diagram $\mc H' = (\S',\ba,\bb)$ for $(M', \gamma')$,
we obtain a Heegaard diagram $\mc H = (\S,\ba,\bb)$ for $(M, \gamma)$
by connecting the components $s$ and $t$ of $\partial \S'$
corresponding to $p$ and $q$ with an annulus.
If the core of the connecting annulus is~$c$, then $A = c \times I \subset M$.

If we glue product 2-handles to $(M',\gamma')$ along~$s$ and~$t$,
then the resulting sutured manifold $(N,\nu)$ is still balanced.
Indeed, as $\chi(R_+(\nu)) = \chi(R_-(\nu))$,
the only thing we need to check is that each component of~$\partial N$ has a suture.

First, suppose that~$p$ and~$q$ belong to different components of~$L'$.
Then, by definition, $P'$ has at least two decorations on each component of~$L'$.
Hence gluing product 2-handles along~$s$ and~$t$ results in two $S^2$ components of $\partial N$ with at least one suture on each.

Now suppose that~$p$ and~$q$ lie on the same component~$K'$ of~$L'$ and $|L| = |L'| + 1$.
As the components $K_0$ and $K_1$ of $L$ corresponding to $K'$ each contain at least two points of~$P$,
we have $|K' \cap P'| \ge 6$.
Furthermore, $p$ and~$q$ cannot be consecutive decorations,
as otherwise one of $K_0$ and $K_1$ would have no decorations.
Hence $\partial N$ has two $S^2$ components with at least one suture on each.

Finally, suppose that~$p$ and~$q$ lie on the same component $K'$ of~$L'$ and $|L| = |L'|$.
Then $K'$ must contain at least four decorations, and $p$ and $q$
cannot be consecutive since $A$ has one boundary component in $R_-(\g)$,
and one boundary component in $R_+(\g)$.
Hence $\partial N$ has two $S^2$ components with at least one suture on each.

Let $\mc H_N = (\S_N, \ba, \bb)$ be the Heegaard diagram for $(N, \nu)$ obtained from $\mc H'$ by gluing disks
to $\S'$ along~$s$ and~$t$.
Since $(N, \nu)$ is still a balanced sutured manifold by the discussion above,
it follows from \cite[Proposition~3.15]{SFH} that $\mc H_N$ can be made admissible by winding~$\bb$ on $\S_N$,
giving rise to a diagram $(\S_N, \ba, \bb')$.
Furthermore, we can wind along arcs that are disjoint from~$s$ and~$t$,
hence $(\S',\ba,\bb')$ is also an admissible diagram of $(M',\gamma')$.
Without loss of generality, we can assume that $\mc H'$ is already such that $\mc H_N$ is admissible.
Then we claim that $\mc H$ is also admissible.
Indeed, if $\mc P$ is a non-zero periodic domain in $\mc H$,
then compressing it along the core curve~$c$,
we obtain a periodic domain~$\mc P'$ in~$\mc H_N$.
As $\mc H_N$ is admissible, $\mc P'$, and hence also~$\mc P$, has both positive and negative multiplicities.

In particular, we can assume that both diagrams $\mc H$ and $\mc H'$ are admissible.
Let $C$ be an annular neighborhood of the curve~$c$ disjoint from~$\ba$ and~$\bb$.
Then $(\mc H, C)$ is a surface diagram adapted to~$A$.
Every $\x \in \T_\a \cap \T_\b$ has $\x \cap C = \emptyset$, hence $O_C = \T_\a \cap \T_\b$.
By \cite[Lemma~5.5]{surface}, we have $\s(\x) \in O_A$ if and only if $\x \in O_C$,
so $\CF(\H,\s) = 0$ for every $\s \not\in O_A$.
Using the sutured manifold decomposition formula \cite[Theorem~1.3]{surface},
we obtain that
\[
\SFH(M',\gamma') \cong \bigoplus_{\s \in O_A} \SFH(M,\gamma,\s) = \SFH(M,\g),
\]
and the result follows.
\end{proof}

In the next key result, we show that the map~$\Phi_A$ preserves the Alexander and Maslov gradings
if~$A$ is a product annulus in a link complement.

\begin{prop}
\label{prop:pa}
Let $(Y, L, P)$ and $(Y', L', P')$ be decorated null-homologous links
with complementary sutured manifolds $(M,\gamma)$ and $(M',\gamma')$, respectively,
such that there is a product annulus decomposition $(M, \gamma) \overset{A}{\rightsquigarrow} (M', \gamma')$.
Let $S$ be a Seifert surface for $L$ that intersects~$A$ in an essential arc.
Then $S' := S \cap M'$ is a Seifert surface for~$L'$.
With the orientations of $L$ and $L'$ induced by the Seifert surfaces, the map
\[
\Phi_A \colon \HFLh(Y', L', P') \to \HFLh(Y, L, P)
\]
satisfies $A_{S'}(x) = A_S(\Phi_A(x))$ for every homogeneous $x \in \HFLh(Y', L', P')$.

If both $\HFLh(Y, L, P)$ and $\HFLh(Y', L', P')$ admit absolute Maslov gradings,
then $\Phi_A$ also preserves the Maslov grading.
\end{prop}

\begin{proof}
Choose a pair of arcs $a$ adapted to~$A$ and disjoint from~$S$.
By definition, $\Phi_A = \Phi_{A,a} = \Phi_{A_a} \circ (d_a)_*$.
Note that the diffeomorphism $d_a \colon (M',\gamma') \to (M_a, \gamma_a)$
fixes the Seifert surface $S'$ pointwise,
and hence $(d_a)_*$ preserves the $S'$-Alexander grading and the Maslov grading.
So it suffices to prove that $\Phi_{A_a}$ preserves the Alexander grading
and the Maslov grading.

We saw above that $\Phi_{A_a} = \Phi_\xi$ for a contact structure $\xi$ on $-M \setminus \Int(M_a)$
with convex boundary and dividing set $-(\gamma \cup \gamma')$.
For
\[
\s = [v] \in \SpinC(M_a,\gamma_a) = \SpinC(-M_a,-\gamma_a),
\]
let $f_{\xi}(\s) = [v \cup \xi^\perp] \in \SpinC(M,\g)$.
By \cite[Proposition~9.11]{cobordisms},
\[
\Phi_\xi(\SFH(M_a, \gamma_a, \s)) \subseteq \SFH(M, \gamma, f_{\xi}(\s)),
\]
hence $A_S(\Phi_\xi(x)) = A_S(f_\xi(\s))$ for any $x \in \SFH(M_a, \g_a, \s)$.

We first show that $\Phi_{\xi}$ preserves the relative Alexander grading; i.e.,
\begin{equation} \label{eqn:relative}
A_{S'}(x) - A_{S'}(y) = A_S(\Phi_\xi(x)) - A_S(\Phi_\xi(y))
\end{equation}
for any homogeneous elements $x$, $y \in \SFH(M_a,\g_a)$.
If $x \in \SFH(M_a,\g_a,\s)$ and $y \in \SFH(M_a,\g_a,\t)$,
then choose admissible representatives $v$ and $w$ of $\s$ and $\t$, respectively.
Then $f_\xi(\s) = [v \cup \xi^\perp]$ and $f_\xi(\t) = [w \cup \xi^\perp]$.
The class $\PD[\s - \t] \in H_1(M_a)$
is represented by the 1-manifold~$D \subset M_a$ where $v = -w$. By construction, $\PD[f_\xi(\s) - f_\xi(\t)] \in H_1(M)$
is also represented by~$D$. Hence
\[
\begin{split}
A_{S'}(\s) - A_{S'}(\t) = \scal{\s - \t}{[S']} = D \cap [S'] =\\
D \cap [S] = \scal{f_\xi(\s) - f_\xi(\t)}{[S]} = A_S(f_\xi(\s)) - A_S(f_\xi(\t)),
\end{split}
\]
which implies equation~\eqref{eqn:relative}.

Both $(\SFH(M_a,\gamma_a), A_{S'})$ and $(\SFH(M,\gamma), A_S)$ are nonzero, finite dimensional,
graded vector spaces with support symmetric in~$0$.
By Lemma~\ref{lem:PhiAiso}, the map $\Phi_A$ is an isomorphism, hence so is $\Phi_\xi$.
As $\Phi_{\xi}$ preserves the relative Alexander grading,
it shifts the absolute Alexander grading by some constant~$n$. But $n = 0$,
as otherwise the support of $A_S$ would not be symmetric in~$0$.

Suppose now that both $\hat\HFL(Y,L,P)$ and $\hat\HFL(Y',L',P')$ admit absolute Maslov gradings
(i.e., if they are supported in torsion $\SpinC$ structures in $\SpinC(Y)$
and $\SpinC(Y')$, respectively).
To show that $\Phi_A$ preserves the Maslov grading, it suffices to prove
the $\Phi_\xi$ preserves the Maslov grading.

\begin{figure}
\begin{center}
\includegraphics{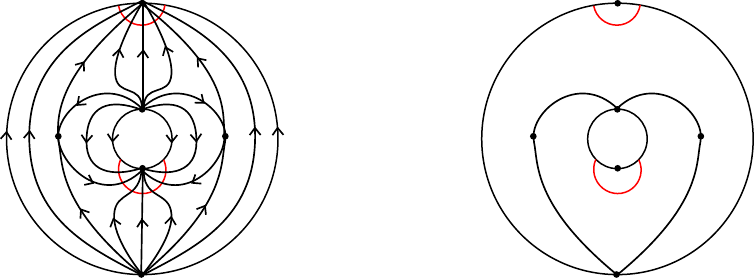}
\caption{A characteristic foliation on the annulus $A_a$, and a Legendrian skeleton on the right.
The dividing set is drawn in red.}
\label{fig:foliation}
\end{center}
\end{figure}

We use the definition of the contact gluing map given in the proof of~\cite[Theorem~6.2]{HKMcontact},
which agrees with~$\Phi_\xi$ by~\cite[Proposition~6.4]{gluingmap}.
First, choose an arbitrary extension of~$\xi$ to~$-M$.
By construction, the annulus $A_a$ is a convex surface in~$\xi$
with dividing set two boundary-parallel arcs, one for each component of $\partial A_a$.
The half-disks cut off by these arcs form the minus region $R_-(A_a)$, and $R_+(A_a)$ is the rest.
In fact, we can choose~$\xi$ such that the characteristic foliation on $A_a$ is as on the left of Figure~\ref{fig:foliation}.
Then a Legendrian skeleton~$C$ of $R_+(A_a)$ is shown on the right of Figure~\ref{fig:foliation}. It consists of two
arcs connecting the two points in $R_+(A_a) \cap s(\g)$, and their union is homotopic to the
core of~$A_a$. We then take a neighbourhood $A_a \times [-1,1]$ of~$A_a$ in~$M$ where $\xi$ is $[-1,1]$-invariant,
and choose parallel copies $A_{-\varepsilon} = A_a \times \{-\varepsilon\}$ and
$A_{\varepsilon} = A_a \times \{\varepsilon\}$, together with Legendrian skeletons
$C_{-\varepsilon} = C \times \{-\varepsilon\}$ and $C_{\varepsilon} = C \times \{\varepsilon\}$.
As in the proof of~\cite[Theorem~6.2]{HKMcontact}, extend $C_{-\varepsilon} \cup C \cup C_{\varepsilon}$
to a Legendrian skeleton~$K$ of~$(-M,-\g,\xi)$ such that the extension is
disjoint from $A_a \times [-\varepsilon,\varepsilon]$ and from $C_{-\varepsilon} \cup C \cup C_{\varepsilon}$.
From the decomposition $M = N(K) \cup (M \setminus N(K))$, we obtain a partial open
book decomposition $(\mc S, h \colon P \to \mc S)$ of $(M,\g)$, which is shown in Figure~\ref{fig:annulus}.
(For a slightly different partial open book and arc basis, see \cite[Figure~29]{HKMcontact}.)

\begin{figure}
\begin{center}
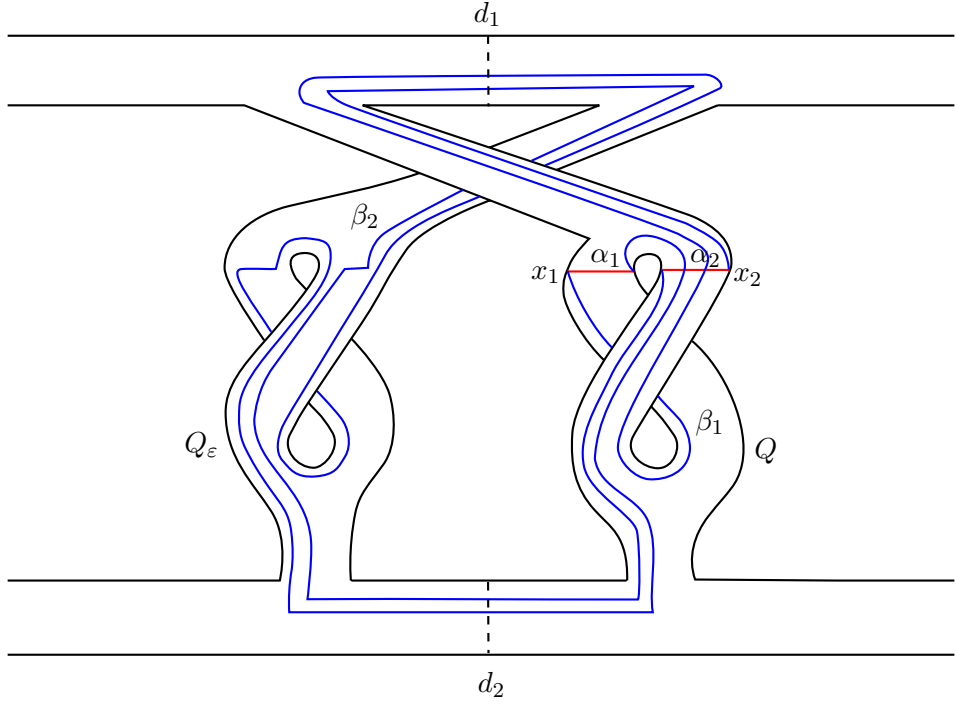
\caption{The sutured diagram arising from a partial open book decomposition adapted
to the annulus gluing along~$A_a$. The red arcs are $a_1$ and $a_2$ (or $\alpha_1$ and $\alpha_2$
once we glue their endpoints), the blue arcs are $h(a_1)$ and $h(a_2)$
(or $\beta_1$ and $\beta_2$ once we glue their endpoints). We obtain a partial open book for $(M',\g')$
by cutting along the dashed arcs $d_1$ and $d_2$ and removing $Q$ from $P$.}
\label{fig:annulus}
\end{center}
\end{figure}

We denote by $Q$ and $Q_{\pm \varepsilon}$ the components of $P = \mc S \setminus \overline{R_+(\g)}$
corresponding to $L$ and $L_{\pm \varepsilon}$, respectively. Given a properly embedded arc~$u$
on~$Q$ with endpoints in $\partial Q \setminus s(\gamma)$, we let $u_{\varepsilon} = u \times \{\varepsilon\}$.
Then $h(u) = u'u_{\varepsilon}u''$, where $u'$ and $u''$ are arcs that connect the corresponding
endpoints of $u$ and $u_\varepsilon$ by following $\partial \mc S$ in the positive direction,
and crossing to the other side of $\partial \mc S$ when reaching $\partial Q_\varepsilon \setminus \partial \mc S$;
see Figure~\ref{fig:annulus}.

We obtain a partial open book decomposition $(\mc S',h' \colon P' \to \mc S')$ of $(M',\g')$
by cutting~$\mc S$ along the arcs~$d_1$ and~$d_2$ shown in Figure~\ref{fig:annulus},
setting $P' = P \setminus Q$, and $h' = h|_{P'}$.
Note that $D_i = d_i \times [-1,1]$ are product disks corresponding to
the components of $R_-(A_a)$. We obtain a sutured diagram $\mc H = (\S,\ba,\bb)$
of $(M,\g)$ from $(\mc S,h)$ as follows: Take an arc basis $a_1, \dots, a_n$
of $H_1(P, P \cap \partial \mc S)$ such that $a_1$, $a_2 \subset Q$ are as the
red arcs in Figure~\ref{fig:annulus}, and
for $i \in \{1, \dots, n\}$, let $b_i$ be a translate of $a_i$
in the positive direction along $\partial \mc S$ such that $|a_i \cap b_i| = 1$.
We obtain $\S$ by gluing $\mc S \times \{0\}$ to $P \times \{1\}$ along $P \cap \partial \mc S$.
Furthermore, we set $\alpha_i = a_i \times \{0,1\}/ \sim$ and
$\beta_i = (b_i \times \{1\}) \cup (h(b_i) \times \{0\})/ \sim$, and let $x_i = (a_i \cap b_i) \times \{1\}$.
We obtain the diagram $\mc H'$ of $(M',\g')$ from $(\mc S', h')$ in an analogous manner.
For $\x' \in \T_{\alpha'} \cap \T_{\beta'}$, Honda, Kazez, and Mati\'c defined
$\Phi_\xi(\x') = (\x',x_1,x_2) \in \T_{\alpha} \cap \T_{\beta}$.
One also obtains~$\S$ from $S$ by identifying the two components
of~$\nbd(a_i) \cap \partial \mc S$. This way, Figure~\ref{fig:annulus} also represents~$\mc H$:
The red arcs are~$\a_1$ and~$\a_2$, while the blue arcs are~$\b_1$ and~$\b_2$.
There might be other $\b$-curves crossing $\a_1$ and $\a_2$ that we have omitted from the figure.

Given $\x'$, $\y' \in \T_{\alpha'} \cap \T_{\beta'}$, let $\x = \Phi_\xi(\x')$ and $\y = \Phi_\xi(\y')$.
Consider the multi-pointed Heegaard diagram $\bar{\mc H} = (\bar{\Sigma},\ba,\bb, \bw, \bz)$,
where $\bar{\S}$ is obtained by gluing disks to $\S$ along $\partial \S$, and adding
a $z$ or $w$ basepoint in each disk depending on the type of suture we are capping off.
Similarly, we obtain $\bar{\mc H}' = (\bar{\S}',\ba',\bb',\bw',\bz')$ by capping off~$\mc H'$.
Let $\mc D'$ be a domain on $\bar{\mc H}'$ from~$\x'$ to~$\y'$. We can visualize $\mc D'$ on $\mc H'$
as a domain that might have non-zero multiplicities along~$\partial \S'$.
For $i \in \{1,2\}$, let $d_i^\pm$ denote the two sides of~$d_i$ in~$\S'$,
and let~$k_i^\pm$ be the multiplicity of~$\mc D'$ along~$d_i^\pm$.
Note that $d_1^+$, $d_1^-$, $d_2^+$, and $d_2^-$ belong to different components
$s_1^+$, $s_1^-$, $s_2^+$, and $s_2^-$ of $s(\g') = \partial \S'$.
The sutures $s_i^+$ and $s_i^-$ are glued to a suture $s_i$ of $\S'$ for $i \in \{1,2\}$.
Note that $A_a$ intersects $s(\g)$ in $s_1$ and $s_2$.

\begin{lem} \label{lem:gluing}
Suppose that $k_1^+ = k_1^- = k_2^+ = k_2^-$; call this common value~$k$.
Then $\gr(\x,\y) = \gr(\x',\y')$.
\end{lem}

\begin{proof}
In this case, $\mc D'$ glues to a domain $\mc D$ in $\bar{\mc H}$ from $\x$ to $\y$
if we glue~$d_i^+$ to~$d_i^-$, and identify the components of $\nbd(a_i) \cap \partial \mc S$
for $i \in \{1,2\}$.
In particular, $\mc D$ has multiplicity~$k$ in $Q \times \{1\}$.
We claim that $\mu(\mc D') = \mu(\mc D) + 2k$.
By Lipshitz's formula \cite[Corollary 4.10]{Lipshitz},
\[
\mi(\mc D) = n_{\x}(\mc D) + n_{\y}(\mc D) + e(\mc D),
\]
and an analogous formula holds for $\mc D'$.
Here $n_{\x}(\mc D)$ denotes the \emph{point measure} of the domain,
obtained by summing the multiplicities of $\mc D$ in all regions adjacent
to $\x$ and dividing by~4, while $e(\mc D)$ denotes the \emph{Euler measure} of~$\mc D$.
If $F$ is a surface with boundary and corners,
then $e(F)$ is $\frac{1}{2\pi}$ times the integral over $F$ of the curvature
of a metric on $F$ for which the edges of $\de F$ are geodesics that meet at right angles.

We first compare $e(\mc D)$ and $e(\mc D')$.
We obtain $\mc D'$ from $\mc D$ as follows: We remove the disks glued to~$s_1$ and~$s_2$.
Both have multiplicity~$k$, hence this decreases the Euler measure by~$2k$.
We then cut the resulting surface along~$d_1$ and~$d_2$, which increases the Euler measure
by~$2k$. We next remove $Q \times \{1\}$, which is the same as cutting $\S$ along $\nbd(a_1 \cap \partial \mc S)$
and $\nbd(a_2 \cap \partial \mc S)$.
This again increases the Euler measure by $2k$. We finally fill the sutures
$s_1^+$, $s_1^-$, $s_2^+$, and $s_2^-$ with disks of multiplicity~$k$ each,
increasing the Euler measure by~$4k$. In summary,
\[
e(\mc D') = e(\mc D) -2k +2k +2k +4k = e(\mc D) + 6k.
\]
On the other hand,
\[
n_{\x'}(\mc D') + n_{\y'}(\mc D') = n_{\x}(\mc D) + n_{\y}(\mc D) - 4k,
\]
since $\x = (\x',x_1,x_2)$ and $\y = (\y',x_1,x_2)$, and the multiplicity of~$\mc D$
at~$x_1$ and~$x_2$ is~$k$. It follows that $\mu(\mc D') = \mu(\mc D) + 2k$.

Since $|\bw'| = |\bw|+1$, and the extra $w$ suture has multiplicity~$k$, we have
\[
n_{\bw'}(\mc D') = n_{\bw}(\mc D) + k.
\]
By definition of the relative Maslov grading and the above computation, we have
\[
\gr(\x',\y') = \mi(\mc D') - 2 n_{\bw'} (\mc D') =
\mi(\mc D) +2k - 2n_{\bw}(\mc D) -2k = \gr(\x, \y),
\]
as claimed.
\end{proof}

We now proceed in six steps:

\emph{Step 1: If $\s(\x') = \s(\y')$ in $\SpinC(M',\g')$, then $\gr(\x,\y) = \gr(\x',\y')$.}
Indeed, since $\s(\x') = \s(\y')$, we can choose a domain $\mc D'$ from $\x'$ to $\y'$ in $\mc H'$.
In particular, $\mc D'$ has multiplicity zero along $\partial \S'$.
As the multiplicities $k_i^\pm$ of $\mc D'$ are all zero, the claim immediately follows from
Lemma~\ref{lem:gluing}.

\emph{Step 2: Let $L^+$ (respectively $L^-$)
be the component of $L'$ containing the decorations corresponding to~$s_1^+$ and $s_2^+$
(respectively $s_1^-$ and $s_2^-$).
If $L^+ = L^-$ and $|P' \cap L^+| \ge 6$, or if $L^+ \neq L^-$ and $|P' \cap L^+| \ge 4$
and $|P' \cap L^-| \ge 4$, then $\gr(\x,\y) = \gr(\x',\y')$.}
It follows from the assumptions that there are components $R_i^\pm$ of $R(\g')$ for $i \in \{1,2\}$ such that
$d_i^\pm \subset \partial R_i^\pm$, but every element of $\{d_1^+,d_1^-,d_2^+,d_2^-\} \setminus \{d_i^\pm\}$
is disjoint from~$R_i^\pm$.
Then $R_i^\pm$ corresponds to a periodic domain $\mc P_i^\pm$ in $\bar{\mc H}'$
such that $\partial \mc P_i^\pm$ has multiplicity 1 along $d_i^\pm$ and 0 along
every element of $\{d_1^+,d_1^-,d_2^+,d_2^-\} \setminus \{d_i^\pm\}$.
Hence the domain
\[
\mc D' + \sum_{i =1}^2 (k - k_i^+) \mc P_i^+ + (k - k_i^-) \mc P_i^-
\]
has multiplicity $k$ along every $d_i^\pm$.
So, without loss of generality, we can suppose that $k_1^+ = k_1^- = k_2^+ = k_2^-$,
and the claim follows from Lemma~\ref{lem:gluing}.
From now on, our aim is to be able to increase the number of decorations on
the link components~$L^+$ and~$L^-$ involved in the gluing to reduce to this case.

\emph{Step 3: Suppose that $L^+$ and $L^-$ lie in different components $Y^+$ and $Y^-$ of $Y'$,
that $|P' \cap L_-| \ge 4$, and $\s(\x')|_{Y^+ \cap (M',\g')} = \s(\y')|_{Y^+ \cap (M',\g')}$
in $\SpinC(Y^+ \cap (M',\g'))$. Then $\gr(\x,\y) = \gr(\x',\y')$.}
This is a combination of Steps~1 and~2. In particular, we can assume that $\mc D' \cap Y^+$
has multiplicity zero along $d_1^+$ and $d_2^+$. Furthermore, there are components
$R_1^-$ and $R_2^-$ of $R(\g)$ such that
$d_i^- \subset \partial R_i^-$, but every element of $\{d_1^+,d_1^-,d_2^+,d_2^-\} \setminus \{d_i^-\}$
is disjoint from~$R_i^-$ for $i \in \{1,2\}$.
Then $R_i^-$ corresponds to a periodic domain $\mc P_i^-$ in $\bar{\mc H}'$
such that $\partial \mc P_i^-$ has multiplicity 1 along $d_i^-$ and 0 along
every element of $\{d_1^+,d_1^-,d_2^+,d_2^-\} \setminus \{d_i^-\}$.
Then
\[
\mc D' - k_1^- \mc P_1^- - k_2^- \mc P_2^-
\]
has multiplicity zero along every $d_i^\pm$,
and the claim again follows from Lemma~\ref{lem:gluing}.

\emph{Step 4: The map $\Phi_\xi$ preserves the Maslov grading when
\[
(Y',L',P') = (Y,L,P_0) \sqcup (S^3,U,P^4),
\]
where $U$ is the unknot and $|P^4| = 4$.}
In this case, $Y = Y'$ and $L = L'$, while $|P| = |P'| + 2$.

Note that $\hat\HFL(S^3,U,P^4) \cong V'$, where
$V' = \F_2\langle b, t \rangle$ is the 2-dimensional bigraded vector field generated
by two homogeneous vectors: $b$ in Maslov grading~$-1/2$ and Alexander grading~$-1/2$,
and $t$ in Maslov grading~$1/2$ and Alexander grading~$1/2$.
Indeed, $\mc H_U = (S^2,\a,\b,\bw_U,\bz_U)$ is a diagram of $(S^3,U,P^4)$
if $|\a \cap \b| = 2$, and there is exactly one basepoint in each component
of $S^2 \setminus (\a \cup \b)$ such that there is one point of $\bz$ and one point
of $\bw$ in each component of $S^2 \setminus \a$ and $S^2 \setminus \b$.
The elements $t$ and $b$ are represented by the points of $\a \cap \b$,
labeled such that there is a holomorphic disk from $t$ to $b$ that has a single $z$ basepoint in its interior.
Hence $\gr(t,b) = 1$ and $A_{D^2}(t) - A_{D^2}(b) = 1$, where $D^2$ is the unit disc spanning~$U$.
Note that $\W(S^3,U,P^4) = (S^1 \times D^2, \g_4)$, where $\g_4$ consists of four longitudinal sutures.
Furthermore, $\s(t) = 1/2$ and $\s(b) = -1/2$ in $\SpinC(S^1 \times D^2, \g_4) \cong \Z + 1/2$,
where the isomorphism is given by the $D^2$-Alexander grading.

By the K\"unneth formula,
\[
\hat\HFL(Y',L',P') \cong \hat\HFL(Y,L,P_0) \otimes V'.
\]
This is graded by
\[
\SpinC(M',\g') \cong \SpinC(M,\g_0) \times (\Z + 1/2),
\]
where $(M,\g_0) = \W(Y,L,P_0)$.
We now explicitly compute $\hat\HFL(Y,L,P)$, together with the Maslov and $\SpinC$ gradings;
cf.~\cite[Section~9]{polytope}.
Let
\[
\mc H_0 = (\S_0,\ba_0,\bb_0,\bw_0,\bz_0)
\]
be a weakly admissible diagram of $(Y,L,P_0)$.
We obtain a diagram $\mc H_0 \# \mc H_U$ of $(Y,L,P)$ by taking the connected sum of~$\mc H_0$ and~$\mc H_U$
at base points corresponding to the two sutures that appear when decomposing $(M,\g)$ along~$A$.
As in \cite[Section~6.1]{HFL}, we have
\[
\CF(\mc H_0 \# \mc H_U) \cong \CF(\mc H_0) \otimes \CF(\mc H_U)
\]
for a sufficiently long connected sum neck. The Maslov gradings agree on the two sides.
The only difference is in the $\SpinC$ gradings.
In particular, let $m$ be the meridian of the component of~$L_0$ involved in the annulus gluing.
If $\x_0$, $\y_0 \in \T_{\a_0} \cap \T_{\b_0}$,
then $\s(\x_0 \otimes t) - \s(\y_0 \otimes t) = \s(\x_0) - \s(\y_0)$ and
$\s(\x_0 \otimes b) - \s(\y_0 \otimes b) = \s(\x_0) - \s(\y_0)$,
while $\s(\x_0 \otimes t) - \s(\y_0 \otimes b) = \s(\x_0) - \s(\y_0) + m$.
We write $\s_{\pm 1/2} \in \SpinC(M,\g)$ for the image of gluing
$(\s, \pm 1/2) \in \SpinC(M,\g_0) \times \frac12\Z$. Then $\s_{1/2} - \s_{-1/2} = m$.
Note that
\[
\begin{split}
G(\s,1/2) &:= \hat\HFL\left(Y',L',P',(\s,1/2)\right) \cong \hat\HFL(Y,L,P_0,\s) \otimes t \text{ and} \\
G(\s,-1/2) &:= \hat\HFL\left(Y',L',P',(\s,-1/2)\right) \cong \hat\HFL(Y,L,P_0,\s) \otimes b.
\end{split}
\]
It follows from the above discussion that
\begin{equation} \label{eqn:stab}
G(\s,1/2) \oplus G(\s + m,-1/2) \cong \hat\HFL(Y,L,P,\s_{1/2})
\end{equation}
as Maslov graded groups.

By Lemma~\ref{lem:PhiAiso}, the map $\Phi_\xi$ is an isomorphism.
It maps $G(\s,1/2) \oplus G(\s + m,-1/2)$ to $\hat\HFL(Y,L,P,\s_{1/2})$
according to \cite[Proposition~9.11]{cobordisms}.
It preserves the relative Maslov grading in each relative $\SpinC$ structure
on $(M',\g')$ according to Step~1. If $\s \in \SpinC(M,\g_0)$, then let
$c_{\s,\pm 1/2}$ be the Maslov grading shift of $\Phi_\xi$ on $G(\s,\pm 1/2)$.
(We let $c_{\s,\pm 1/2} = 0$ if $G(\s,\pm 1/2) = 0$.)
By Step~3, we have $c_{\s, 1/2} = c_{\s, -1/2}$.
For a fixed $\s \in \SpinC(M,\g_0)$ such that $G(\s,\pm 1/2) \neq 0$,
let $k$ be the largest integer such that
\[
\hat\HFL(Y,L,P_0,\s - km) \neq 0.
\]
Then $c_{\s - km, -1/2} = 0$, since
\[
G(\s - km, -1/2) \cong \hat\HFL\left(Y,L,P, (\s-(k+1)m)_{1/2}\right)
\]
as absolutely graded groups by equation~\eqref{eqn:stab},
and $\Phi_\xi$ induces an isomorphism of relatively graded groups.
Hence $c_{\s - km, 1/2} = 0$ by Step~3. So the map
\[
\Phi_\xi \colon G(\s-km,1/2) \oplus G(\s - (k-1)m,-1/2) \to \hat\HFL(Y,L,P, (\s-km)_{1/2})
\]
can only be an isomorphism if $c_{\s - (k-1)m, -1/2} = 0$. Repeating this process, we get that
$c_{\s,\pm 1/2} = 0$ for every~$\s$, and $\Phi_\xi$ preserves the absolute Maslov grading.
This concludes the proof of Step~4.

\emph{Step~5: $\Phi_A$ preserves the relative Maslov grading in the general case.}
We replace both $s_1$ and $s_2$ (the sutures that $A_a$ intersects) with three parallel
sutures each. This corresponds to the decoration~$P_1$ of~$L$,
and we write $B$ for $A$ viewed in $(M_1,\g_1) := \W(Y,L,P_1)$.
If we decompose $(M_1,\g_1)$ along $B$, then the result satisfies the conditions of Step~2.
Let $A'$ be the union of the two product annuli such that
decomposing $(M_1,\g_1)$ along $A'$, we obtain $(M,\g)$ and two copies of $(S^1 \times D^2, \g_4)$.
We can also assume that $A \cap A' = \emptyset$.
Let $b$ be the union of two pairs of arcs adapted to the components of $A'$,
such that if $R_0$ is a component of $R(\g_1)$,
and $a \cap R_0 \neq \emptyset$ and $b \cap R_0 \neq \emptyset$, then~$b$ is a subarc of~$a$.
This ensures that $A_a \cap A'_b = \emptyset$. By the composition property of the
contact gluing maps~\cite[Proposition~6.2]{gluingmap}, we have
\[
\Phi_{B_a} \circ \Phi_{B_b'} = \Phi_{A_b'} \circ \left(\Phi_{A_a} \otimes \id_{V' \otimes V'}\right),
\]
where $B_b'$ is~$A_b'$ considered in the manifold obtained by decomposing $(M_1,\g_1)$ along~$B_a$.
Both $\Phi_{A_b'}$ and $\Phi_{B_b'}$ preserve the absolute Maslov grading by Step~4,
and $\Phi_{B_a}$ preserves the relative Maslov grading by Step~2.
It follows that $\Phi_{A_a}$ also preserves the relative Maslov grading.

\emph{Step~6: $\Phi_A$ preserves the absolute grading.}
Since $\Phi_A$ preserves the relative Maslov grading,
it shifts the absolute Maslov grading by some constant~$n$.
The proof of \cite[Proposition~8.2]{HFL} gives the symmetry property
\[
\hat\HFL_i(Y^*,L^*,P^*,j) \cong \hat\HFL_{i-2j}(Y^*,L^*,P^*,-j)
\]
for $(Y^*,L^*,P^*) \in \{\,(Y,L,P) \text{, } (Y',L',P')\,\}$.
It follows that $n = 0$.
\end{proof}


\section{A decorated skein exact triangle}
\label{sec:skein}

Ozsv\'ath and Szab\'o~\cite[Theorem~10.2]{HFK1} defined an oriented skein exact triangle in knot Floer homology.
Recall that $\hat\HFK(Y,L)$ of an $n$-component oriented link~$L$ in a closed, connected, oriented 3-manifold~$Y$ is obtained as follows:
Choose $2n-2$ points $p_1,\dots,p_{n-1}$ and $q_1,\dots,q_{n-1}$ in~$L$ such that if we identify each~$p_i$ and~$q_i$ in~$L$,
then we obtain a connected graph. After attaching 3-dimensional 1-handles to~$Y$ with feet at $p_i$ and $q_i$,
and inside each 1-handle taking the connected sum of the link components, we obtain a knot~$\kappa(L,\{p_i,q_i\})$ in a 3-manifold
$\kappa(Y,\{p_i,q_i\})$ diffeomorphic to $Y\#^{n-1}(S^2 \times S^1)$. Then
\[
\hat\HFK(Y, L) := \hat\HFK\left(\kappa(Y,\{p_i,q_i\}), \kappa(L,\{p_i,q_i\})\right).
\]
The skein exact triangle is stated in terms of this invariant~$\hat\HFK$.
However, the construction of~$\hat\HFK(Y,L)$ is not natural, as it depends on the choice of points~$p_i$ and~$q_i$,
and is not stated in terms of decorated knots or links. In this subsection, we outline a natural version of the Ozsv\'ath-Szab\'o
skein exact triangle in terms of~$\hat\HFL$ that works for multi-pointed and unoriented (but decorated) links.

We first define a natural version of $\hat\HFK(Y,L)$ for an oriented link~$L$.
Fix an ordering $K_1, \dots, K_n$ of the components of $L$, and let $P$
be a decoration such that $|P \cap K_i| = 2$ for every $i \in \{\,1,\dots, n\,\}$.
Then we let $\partial(R_+(P) \cap K_i) = p_i - q_{i-1}$. In particular,
\[
P = \{\,p_1,\dots,p_n,q_0,\dots,q_{n-1}\,\}.
\]
Furthermore, we let $\kappa(P) = \{p_n,q_0\}$ be the corresponding decoration of $\kappa(L,P) := \kappa(L,\{p_i,q_i\})$
in $\kappa(Y,P) := \kappa(Y,\{p_i,q_i\})$.
Then we define
\[
\hat\HFK(Y,L,P) = \hat\HFK\left(\kappa(Y,P), \kappa(L,P), \kappa(P)\right).
\]
This depends on the decoration $P$ and the ordering of the link components.
We have an isomorphism
\[
\Phi_{L,P} \colon \hat\HFL(Y,L,P) \to \hat\HFK(Y,L,P),
\]
obtained by composing the annulus gluing maps from Section~\ref{sec:morse-pa} that glue the sutures corresponding
to $p_i$ and $q_i$ in $Y(L,P)$ for every $i \in \{\,1,\dots,n-1\,\}$.

Let $Y$ be a closed, connected, oriented 3-manifold, and let $(L,P)$ be a decorated link in~$Y$.
Note that now we do not fix an orientation of $L$.
Suppose that~$D$ is an oriented disk embedded in~$Y$ that intersects both~$R_+(P)$ and~$R_-(P)$ transversely in a single point.
Let $K = \partial D$; we view this as a knot in the sutured manifold $(M,\gamma) := Y(L,P)$.
Let~$m$, $l$, and $s$ be oriented simple closed curves on~$\partial \nbd(K)$ such that $m \cdot l = l \cdot s = s \cdot m = 1$,
the curve $m$ is a meridian of~$K$, and $l = D \cap \partial \nbd(K)$ is a longitude.
For $c \in \{\,m,l,s\,\}$,
let $M_c(K)$ denote the manifold obtained by Dehn filling $M \setminus \nbd(K)$ along~$c$.
Then, as noted by Lipshitz (see Theorem~5.1 and the subsequent paragraph in \cite{Lipshitznotes}), there is a surgery exact triangle
\begin{equation} \label{eqn:surgery}
\dots \longrightarrow \SFH(M_m(K),\gamma) \stackrel{e'}{\longrightarrow} \SFH(M_l(K),\gamma) \stackrel{f'}{\longrightarrow} \SFH(M_s(K),\gamma) \longrightarrow \dots
\end{equation}
Here $M_m(K) = M$, so $\SFH(M_m(K),\gamma) = \hat\HFL(L,P)$. In $M_l(K)$ (see Figure~\ref{fig:L}),
consider the annulus $A$ obtained by
capping off $D \setminus \nbd(K)$ with the disk $\{1\} \times D^2 \subset S^1 \times D^2$ whose boundary is glued to~$l$.
We orient $A$ coherently with~$D$.
Then $A$ is a product annulus since $|D \cap R_\pm(P)| = 1$. Decomposing $(M_l(K),\gamma)$ along~$A$, we obtain
a sutured manifold $(M',\gamma')$ that coincides with $Y(L_0,P_0)$, where $L_0$ is obtained from $L \setminus \nbd(D)$ by reconnecting the
endpoints parallel to~$D$, and adding one decoration to~$P$ along each of these parallel arcs.
As explained in Subsection~\ref{sec:morse-pa}, there is a gluing map
\[
\Phi_A \colon \SFH(Y(L_0,P_0)) \to \SFH(M_l(K),\gamma),
\]
which is an isomorphism by Lemma~\ref{lem:PhiAiso}.
Finally, $(M_s(K), \gamma) = Y(L_+,P)$, where $L_+$ is a link in~$Y$ obtained
by adding a full left-handed twist to~$L$ along $D$.
For an illustration showing $L$, $L_0$, and $L_+$, see Figure~\ref{fig:surgeryet}.
\begin{figure}
\begin{center}
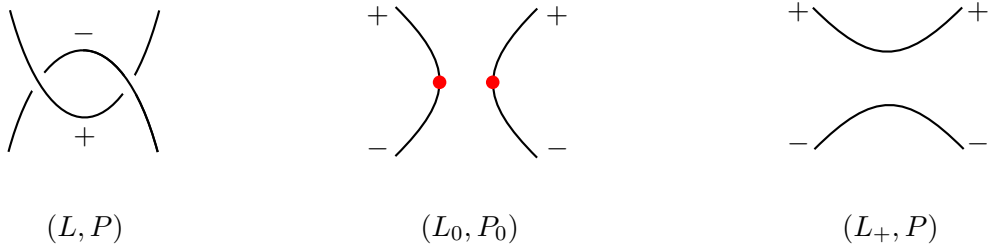
\caption[{Figure \ref{fig:surgeryet}}]{A local picture of the decorated links $(Y, L, P)$,
$(Y, L_0, P_0)$, and $(Y, L_+, P)$ involved in the decorated skein exact triangle.}
\label{fig:surgeryet}
\end{center}
\end{figure}
Hence, we obtain the following natural version of the Ozsv\'ath-Szab\'o skein exact triangle.

\begin{theorem} \label{thm:skein}
Let $Y$ be a closed, connected, oriented 3-manifold, and let $(L,P)$ be a decorated link in~$Y$.
Suppose that~$D$ is an oriented disk embedded in~$Y$ that
intersects both $R_+(P)$ and $R_-(P)$ transversely in a single point, and let $(L_0,P_0)$
and $(L_+,P)$ be the decorated links defined above (see Figure~\ref{fig:surgeryet}).
Then there is an exact triangle
\begin{equation} \label{eqn:exact}
\dots \longrightarrow \hat\HFL(Y,L,P) \stackrel{e}{\longrightarrow}
\hat\HFL(Y,L_0,P_0) \stackrel{f}{\longrightarrow} \hat\HFL(Y,L_+,P) \longrightarrow \dots,
\end{equation}
where $e = \Phi_A^{-1} \circ e'$ and $f = f' \circ \Phi_A$.

If $L$, $L_0$, and $L_+$ are null-homologous and oriented coherently, $S$ is a Seifert surface of $L$ such that
$S \cap D$ consists of a single arc, $S_0$ is the Seifert surface of $L_0$ obtained
by cutting $S$ along $D$, and $S_+$ is obtained by adding a full left-hand twist to $S$
along $D$, then the maps in \eqref{eqn:exact} preserve the Alexander gradings $A_S$, $A_{S_0}$,
and $A_{S_+}$. Furthermore, when the absolute homological gradings are defined (e.g., when $Y$ is a homology sphere),
the maps $e$ and $f$ are homogeneous and decrease the absolute gradings by $1/2$.
\end{theorem}

\begin{proof}
  We prove the second part of the theorem on grading shifts.
  The map $\Phi_A$ does not change the Maslov and the Alexander gradings
  by Proposition~\ref{prop:pa}.
  By~\cite[Theorem~8.2]{HFK1}, for a knot with two decorations,
  all the maps in exact triangle~\eqref{eqn:surgery} preserve the Alexander grading,
  and the maps $e'$ and $f'$ decrease the Maslov grading by~$1/2$.
  A completely analogous proof shows that this also holds for arbitrary decorated links.
\end{proof}

If $|L_0| = |L| \pm 1$ and $P$ consists of two decorations on each component of~$L$ and~$L_+$,
then the above exact triangle is isomorphic to one of the exact triangles~(31) and~(32)
of Ozsv\'ath and Szab\'o~\cite[Theorem~10.2]{HFK1}
(these specialize to the triangles (7) and (8) of~\cite{HFK1} when $Y = S^3$).
The condition $|L_0| = |L| \pm 1$ means that the links $L$, $L_0$, and~$L_+$ can be oriented coherently,
which was an assumption made in~\cite{HFK1}, but which is not necessary for Theorem~\ref{thm:skein}.

Note that Ozsv\'ath and Szab\'o distinguish two cases depending on whether the two strands
of~$L$ (or equivalently, of $L_+$) meeting~$D$ belong to the same component.
Fix an ordering of the components of $L$
such that if two components meet $D$, then they are consecutive. This induces an ordering of the
components of $L_+$. We also get an ordering of the components of $L_0$ such that if exactly one component
of $L$ meets $D$, then the resulting components of $L_0$ are consecutive.

In both cases,
$\hat\HFL(Y,L,P) \cong \hat\HFK(Y,L,P)$ via $\Phi_{L,P}$ and $\hat\HFL(Y,L_+,P) \cong \hat\HFK(Y,L_+,P)$
via $\Phi_{L_+,P}$.
If the two strands meeting $D$ belong to the same component of~$L$, then $|L_0| = |L| + 1$,
and $P_0$ consists of exactly two decorations on each component of~$L_0$, hence
$\hat\HFL(Y,L_0,P_0) \cong \hat\HFK(Y,L_0,P_0)$ via $\Phi_{L_0,P_0}$.
So, in this case, exact sequence~\eqref{eqn:exact} is isomorphic to
\[
\dots \longrightarrow \hat\HFK(Y,L,P) \longrightarrow
\hat\HFK(Y,L_0,P_0) \longrightarrow \hat\HFK(Y,L_+,P) \longrightarrow \dots,
\]
which is the natural version of exact sequence~(31) in~\cite{HFK1}.

On the other hand, if the two strands belong to different components $K_j$ and $K_{j+1}$ of~$L$, then
$P_0$ consists of six decorations on one component of~$L_0$ and two on all other components.
Hence, if we let $P_0' = P \setminus \{p_j,q_j\}$, then
$\hat\HFL(Y,L_0,P_0) \cong \hat\HFK(Y,L_0,P_0') \otimes V' \otimes V'$. Here $V' \otimes V'$ is
a 4-dimensional vector space with a 1-dimensional summand generated by $b \otimes b$ in Maslov and Alexander gradings~$-1$,
a 2-dimensional summand generated by $b \otimes t$ and $t \otimes b$ in Maslov and Alexander gradings~$0$,
and a 1-dimensional summand generated by $t \otimes t$ in Maslov and Alexander gradings~$1$. Ozsv\'ath and Szab\'o~\cite{HFK1}
denoted this vector space~$V$, though in this paper we use that notation for a different purpose.
Hence, in this case, our exact sequence is isomorphic to
\[
\dots \longrightarrow \hat\HFK(Y,L,P) \longrightarrow
\hat\HFK(Y,L_0,P_0') \otimes V' \otimes V' \longrightarrow \hat\HFK(Y,L_+,P) \longrightarrow \dots,
\]
which is a natural version of exact sequence~(32) in~\cite{HFK1}.
Indeed, by \cite[Proposition~9.2]{HFK1}, if $B(0,0)$ is the Borromean knot in $\#^2(S^1 \times S^2)$,
then $V' \otimes V' \cong \hat\HFK(\#^2(S^1 \times S^2), B(0,0))$.
If we consider the surgery exact triangle for $K$ with framings $\{m,l,s\}$
in the complement of the decorated knot $(\kappa(Y,P), \kappa(L,P), \kappa(P))$,
then $l$-surgery gives $\kappa(L_0,P_0') \# B(0,0)$ inside
\[
\kappa(Y,P)_l(K) \approx \kappa(Y,P_0') \# (S^1 \times S^2) \# (S^1 \times S^2).
\]
It follows that
\[
\hat\HFL(Y,L_0,P_0) \cong \hat\HFK\left(Y \# (S^1 \times S^2) \# (S^1 \times S^2),L_0 \# B(0,0), P_0'\right),
\]
which appears in \cite[Theorem~10.2]{HFK1}. The isomorphism is given by decomposing
the sutured manifold complementary to the decorated knot $\kappa(L_0 \# B(0,0), P_0')$
along $A$ and the product annuli corresponding to the cores of the 1-handles attached along~$P$.

Unoriented skein exact triangles in link Floer homology have appeared in the literature,
due to Manolescu~\cite{manolescuexact}, Manolescu and Ozsv\'ath~\cite{manolescu2007khovanov}, and Wong~\cite{wong}. However, such unoriented triangles relate the link Floer homology of a link with that of its two smoothings at a crossing,
as in Figure~\ref{fig:smoothings}. The result of Theorem~\ref{thm:skein} is more similar to Ozsv\'ath and Szab\'o's oriented skein exact triangle~\cite{HFK1}.

%% file: annulus.pdf_tex
\begingroup%
  \makeatletter%
  \providecommand\color[2][]{%
    \errmessage{(Inkscape) Color is used for the text in Inkscape, but the package 'color.sty' is not loaded}%
    \renewcommand\color[2][]{}%
  }%
  \providecommand\transparent[1]{%
    \errmessage{(Inkscape) Transparency is used (non-zero) for the text in Inkscape, but the package 'transparent.sty' is not loaded}%
    \renewcommand\transparent[1]{}%
  }%
  \providecommand\rotatebox[2]{#2}%
  \ifx\svgwidth\undefined%
    \setlength{\unitlength}{356.76148bp}%
    \ifx\svgscale\undefined%
      \relax%
    \else%
      \setlength{\unitlength}{\unitlength * \real{\svgscale}}%
    \fi%
  \else%
    \setlength{\unitlength}{\svgwidth}%
  \fi%
  \global\let\svgwidth\undefined%
  \global\let\svgscale\undefined%
  \makeatother%
  \begin{picture}(1,0.74266383)%
    \put(0,0){\includegraphics[width=\unitlength,page=1]{annulus.pdf}}%
    \put(0.7888712,0.25342863){\color[rgb]{0,0,0}\makebox(0,0)[lb]{\smash{$Q$}}}%
    \put(0.18666922,0.26005528){\color[rgb]{0,0,0}\makebox(0,0)[lb]{\smash{$Q_{\varepsilon}$}}}%
    \put(0.49234121,0.71516013){\color[rgb]{0,0,0}\makebox(0,0)[lb]{\smash{$d_1$}}}%
    \put(0.49659606,0.00630002){\color[rgb]{0,0,0}\makebox(0,0)[lb]{\smash{$d_2$}}}%
    \put(0.61488128,0.45901737){\color[rgb]{0,0,0}\makebox(0,0)[lb]{\smash{}}}%
    \put(0.61578048,0.46065838){\color[rgb]{0,0,0}\makebox(0,0)[lb]{\smash{$\alpha_1$}}}%
    \put(0.72087398,0.46065773){\color[rgb]{0,0,0}\makebox(0,0)[lb]{\smash{$\alpha_2$}}}%
    \put(0.72646799,0.28351584){\color[rgb]{0,0,0}\makebox(0,0)[lb]{\smash{$\beta_1$}}}%
    \put(0.36225187,0.5013648){\color[rgb]{0,0,0}\makebox(0,0)[lb]{\smash{$\beta_2$}}}%
    \put(0.55315546,0.44351624){\color[rgb]{0,0,0}\makebox(0,0)[lb]{\smash{$x_1$}}}%
    \put(0.76656531,0.44070313){\color[rgb]{0,0,0}\makebox(0,0)[lb]{\smash{$x_2$}}}%
  \end{picture}%
\endgroup%

%% file: surgeryet.pdf_tex
\begingroup%
  \makeatletter%
  \providecommand\color[2][]{%
    \errmessage{(Inkscape) Color is used for the text in Inkscape, but the package 'color.sty' is not loaded}%
    \renewcommand\color[2][]{}%
  }%
  \providecommand\transparent[1]{%
    \errmessage{(Inkscape) Transparency is used (non-zero) for the text in Inkscape, but the package 'transparent.sty' is not loaded}%
    \renewcommand\transparent[1]{}%
  }%
  \providecommand\rotatebox[2]{#2}%
  \ifx\svgwidth\undefined%
    \setlength{\unitlength}{374.24244325bp}%
    \ifx\svgscale\undefined%
      \relax%
    \else%
      \setlength{\unitlength}{\unitlength * \real{\svgscale}}%
    \fi%
  \else%
    \setlength{\unitlength}{\svgwidth}%
  \fi%
  \global\let\svgwidth\undefined%
  \global\let\svgscale\undefined%
  \makeatother%
  \begin{picture}(1,0.24877369)%
    \put(0,0){\includegraphics[width=\unitlength,page=1]{surgeryet.pdf}}%
    \put(0.40604881,0.01377641){\color[rgb]{0,0,0}\makebox(0,0)[lb]{\smash{ $(L_0,P_0)$}}}%
    \put(0.02925337,0.01377641){\color[rgb]{0,0,0}\makebox(0,0)[lb]{\smash{ $(L,P)$}}}%
    \put(0.83076563,0.01377641){\color[rgb]{0,0,0}\makebox(0,0)[lb]{\smash{ $(L_+,P)$}}}%
    \put(0,0){\includegraphics[width=\unitlength,page=2]{surgeryet.pdf}}%
    \put(0.05566881,0.21034151){\color[rgb]{0,0,0}\makebox(0,0)[lb]{\smash{\textsl{ $-$}}}}%
    \put(0.05670508,0.10632271){\color[rgb]{0,0,0}\makebox(0,0)[lb]{\smash{\textsl{ $+$}}}}%
    \put(0.35170102,0.22857747){\color[rgb]{0,0,0}\makebox(0,0)[lb]{\smash{\textsl{ $+$}}}}%
    \put(0.53255562,0.22857747){\color[rgb]{0,0,0}\makebox(0,0)[lb]{\smash{\textsl{ $+$}}}}%
    \put(0.35178907,0.0948327){\color[rgb]{0,0,0}\makebox(0,0)[lb]{\smash{\textsl{ $-$}}}}%
    \put(0.53264368,0.0948327){\color[rgb]{0,0,0}\makebox(0,0)[lb]{\smash{\textsl{ $-$}}}}%
    \put(0.7753515,0.23273575){\color[rgb]{0,0,0}\makebox(0,0)[lb]{\smash{\textsl{ $+$}}}}%
    \put(0.95620598,0.23273575){\color[rgb]{0,0,0}\makebox(0,0)[lb]{\smash{\textsl{ $+$}}}}%
    \put(0.77543956,0.09899189){\color[rgb]{0,0,0}\makebox(0,0)[lb]{\smash{\textsl{ $-$}}}}%
    \put(0.95629403,0.09899189){\color[rgb]{0,0,0}\makebox(0,0)[lb]{\smash{\textsl{ $-$}}}}%
    \put(0,0){\includegraphics[width=\unitlength,page=3]{surgeryet.pdf}}%
  \end{picture}%
\endgroup%

%% file: morse.tex
In this section, we compute the maps induced by elementary
decorated link cobordisms where the ambient 4-manifold is of the form $Y \times I$.
Recall that a decorated link is a triple $(Y, L, P)$, where $Y$ is a 3-manifold,
$L$ is a link in~$Y$, and $P$ is a set of points on~$L$; see Definition \ref{def:dl}.
From now on, all decorated links and decorated link cobordisms will be oriented.

\begin{definition}
A decorated cobordism $(Y \times I, F,\sigma)$ from $(Y,L,P_0)$ to $(Y,L,P_1)$
is a \emph{stabilization} (resp.~\emph{destabilization}) if
\begin{itemize}
\item  $F = L \times I$,
\item the only critical point of the height function $h \colon L \times I \to I$ on $\sigma$
is a non-degenerate minimum (resp.~maximum),
\item all but one component of $\sigma$ is of the form $\{x\} \times I$ for some $x \in L$.
\end{itemize}
The orientation of $F$ induces a splitting $F = R_+(\sigma) \cup R_-(\sigma)$
as in Definition~\ref{def:linkcob}. If the bigon component of $F \setminus \sigma$
lies in $R_+(\sigma)$, then we say that $(Y \times I, F,\sigma)$ is a positive stabilization (resp.~destabilization),
and is negative otherwise. Note that this depends on the orientation of~$F$.
\end{definition}

\begin{definition}
\label{def:birth}
A \emph{birth} cobordism from $(Y, L_0, P_0)$ to $(Y, L_1, P_1)$ is a decorated cobordism
$(Y \times I, F, \sigma)$ such that
\begin{itemize}
\item{$(Y, L_1, P_1) = (Y \# S^3, L_0 \sqcup U_1, P_0 \sqcup P_{U_1})$,
where $(S^3, U_1, P_{U_1})$ is the unknot with two decorations,}
\item{$F = F_0 \sqcup D \subseteq Y \times I$, where $F_0 = L_0 \times I$
and $D$ is a disk with boundary~$U_1$ such that
the height function~$h|_D$ has a single critical point of index~0,}
\item{if $\sigma_0 = \sigma \cap F_0$ and $\sigma_D = \sigma \cap D$, then
$\sigma_0$ consists of $|P_0|$ vertical lines, one for each component of $L_0 \setminus P_0$,
and $\sigma_D$ is a single arc on $D$ such that $\text{Crit}(h|_{\sigma_D}) = \text{Crit}(h|_D)$.}
\end{itemize}
\end{definition}

\begin{definition}
\label{def:death}
A \emph{death} cobordism from $(Y, L_0, P_0)$ to $(Y, L_1, P_1)$ is a decorated cobordism
$(Y \times I, F, \sigma)$ such that
\begin{itemize}
\item{$(Y, L_0, P_0) = (Y \# S^3, L_1 \sqcup U_1, P_1 \sqcup P_{U_1})$,
where $(S^3, U_1, P_{U_1})$ is the unknot with two decorations,}
\item{$F = F_1 \sqcup D \subseteq Y \times I$, where $F_1 = L_1 \times I$
and $D$ is disk with boundary~$U_1$
such that the height function~$h|_D$ has a single critical point of index~2,}
\item{if $\sigma_1 = \sigma \cap F_1$ and $\sigma_D = \sigma \cap D$, then
$\sigma_1$ consists of $|P_1|$ vertical lines, one for each component of $L_1 \setminus P_1$,
and $\sigma_D$ is a single arc on $D$ such that $\text{Crit}(h|_{\sigma_D}) = \text{Crit}(h|_D)$.}
\end{itemize}
\end{definition}

Notice that, by turning a birth (death) cobordism upside down, one obtains
a death (birth) cobordism.

\begin{definition}
\label{def:mergesaddle}
A \emph{saddle} cobordism from $(Y, L_0, P_0)$ to $(Y, L_1, P_1)$ is a decorated cobordism
$(Y \times I, F, \sigma)$ such that
\begin{itemize}
\item{$|P_1| = |P_0| - 2$,}
\item{the diffeomorphism $f \colon Y \times \{0\} \to Y \times \{1\}$ defined by $f(x,0) = (x,1)$ maps $L_0$
to $L_1$, except in the neighbourhood of a small square $S \approx I \times I$ in~$Y$
that intersects~$L_0$ and~$L_1$ in the arcs $I \times \partial I$ and $\partial I \times I$, respectively,
as illustrated in Figure~\ref{fig:square}; furthermore, $f^{-1}(P_1) \subseteq P_0$,}
\item{the surface $F$ in $Y \times I$ is a product outside a neighbourhood of $S \times I$,
and it interpolates between $L_0$ and $L_1$ inside $S \times I$, as shown in Figure~\ref{fig:square},}
\item{if $p_Y \colon Y \times I \to Y$ is the projection, then
$p_Y(\sigma) \cap S$ consists of the arc shown in green in Figure~\ref{fig:square},
and $p_Y(\sigma) \setminus S$ is a small translate of $P_0$ such that
\[
\pi_0(\de \sigma) \to \pi_0\left((L_0 \setminus P_0)\sqcup(L_1 \setminus P_1)\right)
\]
is a bijection. In particular, $\text{Crit}(h|_{\sigma}) = \text{Crit}(h|_F)$, and
$\text{Crit}(h|_{\sigma})$ is a local maximum.}
\end{itemize}
We say that $(Y \times I, F,\sigma)$ is a \emph{merge} saddle if $|L_1| = |L_0| - 1$, and
it is a \emph{split} saddle if $|L_1| = |L_0| + 1$.
\end{definition}

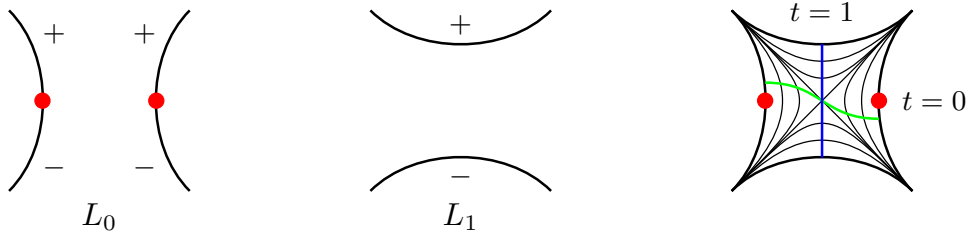
\begin{figure}[t]
\resizebox{0.95\textwidth}{!}{
\input{square.tex}
}
\caption{This figure illustrates a saddle cobordism.
The first two pictures show how $L_0$ and $L_1$ differ in a neighbourhood
of the square~$S$. The last picture shows how the square $S$ is embedded in
$Y \times I$, interpolating from $L_0$ when $t=0$ to $L_1$ when $t=1$.
The green arc in the last picture is the projection of the decoration $\sigma$ in $S$.}
\label{fig:square}
\end{figure}

\begin{definition}
Let $(Y \times I, F,\sigma)$ be a decorated link cobordism from $(Y, L_0, P_0)$ to $(Y, L_1, P_1)$,
and let $h \colon Y \times I \to I$ be the height function.
Then we say that $(Y \times I, F, \sigma)$ is an \emph{isotopy} if $h|_F$ and $h|_\sigma$
have no critical points.
\end{definition}

Let $\DLink'$ be the subcategory of $\DLink$ whose objects are decorated links,
and whose morphisms are decorated link cobordisms $(X,F,\sigma)$ where $X = Y \times I$
for some closed oriented 3-manifold~$Y$.

\begin{prop} \label{prop:generation}
The category $\DLink'$ is generated by stabilizations, destabilizations, births,
deaths, saddles, and isotopies.
\end{prop}

\begin{proof}
\begin{figure}
\includegraphics{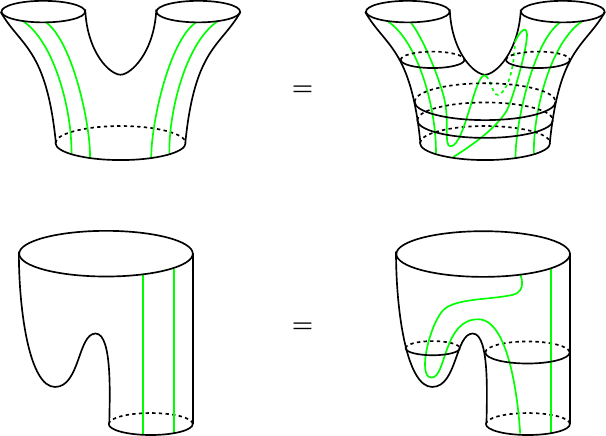}
\caption{Isotoping the dividing set $\sigma$ such that it passes through the critical points of
the height function on~$F$, after which it can be written as a product of elementary cobordisms.}
\label{fig:generation}
\end{figure}
By a slight perturbation of~$F$,
we can assume $h|_F$ is a generic Morse functions.
Then isotope $\sigma$ such that it passes through the critical points of $h|_F$,
and such that $h|_\sigma$ is also a generic Morse function with $\text{Crit}(h|_F) \subset \text{Crit}(h|_\sigma)$.
We also require  that if~$p$ is a saddle point of $h|_F$, then $h|_\sigma$
has a local maximum. Finally, if $t \in I$ is a regular value of $h|_{\sigma}$, then
each component of $h^{-1}(t) \cap F$ contains at least two points of $h^{-1}(t) \cap \sigma$.
For an illustration, see Figure~\ref{fig:generation}.
The cobordism that we obtain is equivalent to the original.
Then let $c_1 < c_2 < \dots < c_n$ be an enumeration of the critical values of~$h|_\sigma$,
where $n = |\text{Crit}(h|_\sigma)|$.
We write $c_0 = 0$ and $c_{n+1} = 1$, and let
\[
\eps = \min\{\, (c_i - c_{i-1})/3 \,\colon\, 1 \le i \le n+1 \,\}.
\]
We set $t_{2i} = c_i - \eps$ and $t_{2i+1} = c_i + \eps$ for $i \in \{1, \dots, n\}$,
and write $t_1 = 0$ and $t_{2n+2} = 1$.

Fix an orientation of $F$ and
a splitting $F = R_+(\sigma) \cup R_-(\sigma)$ as in Definition~\ref{def:linkcob}.
Let $L_j = F \cap h^{-1}(t_j)$, and let $P_j$ be a decoration of~$L_j$
such that its points alternate with $\sigma \cap L_j$.
Then $(F,\sigma)$ is the product of the cobordisms
\[
(F_j,\sigma_j) = (F \cap h^{-1}[t_j,t_{j+1}], \sigma \cap h^{-1}[t_j,t_{j+1}])
\]
for $j \in \{1, \dots, 2n+1\}$ from $(L_j,P_j)$ to $(L_{j+1},P_{j+1})$.
The splitting $L_j = R_+(P_j) \cup R_-(P_j)$ is chosen such that
$\partial R_+(\sigma_j)$ goes from $R_-(P_j)$ to $R_+(P_j)$ as it passes~$P_j$,
where $R_+(\sigma_j) = R_+(\sigma) \cap F_j$.

When $j$ is even,
$h|_{\sigma_j}$ has a single critical point~$p_j$ (that is possibly also a critical point of $h|_{F_j}$),
and we can isotope $F$ and $\sigma$ such that
$F_j$ and~$\sigma_j$ are products outside a small neighbourhood of the critical point,
and all of~$F_j$ is a product if $p_j \not\in \text{Crit}(h|_{F_j})$.
Then every $(F_j,\sigma_j)$ is an elementary cobordism. In particular, it is
an isotopy if and only if $j$ is odd.
\end{proof}

\subsection{Link cobordism maps induced by isotopies} \label{sec:isotopy}

Let $\mc X = (Y \times I, F,\sigma)$ be an isotopy from $(Y, L_0, P_0)$ to $(Y, L_1, P_1)$,
and let $\sigma' \subset F$ be a push-off of $\sigma$ in a normal direction
such that $\partial \sigma' = P_0 \cup P_1$.
If we set $L_t = F \cap h^{-1}({t})$ and
$P_t = \sigma' \cap h^{-1}({t})$ for $t \in I$, then there is
an ambient isotopy $d_t \colon Y \to Y$ for $t \in I$
such that $L_t = d_t(L_0)$ and $P_t = d_t(P_0)$. Then $d_1$
is well-defined up to isotopy, and it follows from
\cite[Lemma~2.22]{surgery} that the link cobordism map $F_{\mc X}$ agrees with the diffeomorphism map
\[
(d_1)_* \colon \hat\HFL(Y,L_0,P_0) \to \hat\HFL(Y,L_1,P_1)
\]
defined in~\cite[Definition~2.42]{naturality}. In particular, it preserves both the homological
and the Alexander gradings.

\subsection{The maps induced by stabilizations and destabilizations}

\begin{proposition} \label{prop:stab}
Let $\mc S = (Y \times I, F, \sigma)$ be a stabilisation from $(Y,L,P_0)$ to $(Y,L,P_1)$.
We denote by $V'$ the bigraded vector space $\F_2 \oplus \F_2$,
where the first summand is generated by the element~$b$
in Maslov and Alexander gradings~$-1/2$, and the second summand is generated by the element~$t$
in Maslov and Alexander gradings~$1/2$.
Then, there is a canonical isomorphism
\[
I_{\mc S} \colon \HFLh(Y,L,P_0) \otimes V' \stackrel{\sim}{\longrightarrow} \HFLh(Y,L,P_1)
\]
that preserves the bigrading when defined, such that the link cobordism map
\[
F_{\mc S} \colon \HFLh(Y,L,P_0) \to \HFLh(Y,L,P_1)
\]
factorises as $I_{\mc S} \circ s_{\mc S}$, where
\[
s_{\mc S} \colon \HFLh(Y,L,P_0) \to \HFLh(Y,L,P_0) \otimes V'
\]
is given by $s_{\mc S}(x) = x \otimes t$ in case of a positive stabilisation, and $s_{\mc S}(x) = x \otimes b$ in case of a negative stabilisation.
\end{proposition}

\begin{proof}
Let $\W = \W(Y \times I, F, \sigma) = (W,Z,\xi)$ be the sutured manifold cobordism
complementary to $(Y \times I, F,\sigma)$ from the sutured manifold $(M,\gamma_0) = \W(Y,L,P_0)$
complementary to~$(Y,L,P_0)$ to the sutured manifold $(M,\gamma_1) = \W(Y,L,P_1)$ complementary to~$(Y,L,P_1)$.
Here $Z \cong -F \times S^1$ is oriented as the boundary of~$W$,
and hence the orientation of~$F$ induces an orientation of the~$S^1$ factor.
This is the orientation for which $S^1$ has linking number one with~$F$.

By definition, the cobordism map $F_\W$ is the composition of the gluing map
\[
\Phi_{-\xi} \colon \SFH(M,\gamma_0) \to \SFH(N,\gamma_1),
\]
where $N = M \cup (-Z)$, and a special cobordism map
\[
F_{\W^s} \colon \SFH(N,\gamma_1) \to \SFH(M,\gamma_1).
\]
As $\W^s$ is a product, it induces the identity map. So, it suffices to compute $\Phi_{-\xi}$.

We denote by $\sigma_0$ the component of $\sigma$ that contains the critical point of the height function $L \times I \to I$.
Let $a$ be a properly embedded arc in~$F \setminus \sigma$ that is parallel to~$\sigma_0$, is disjoint from the bigon component of $F \setminus \sigma$, and such that there are two points of~$P_1$ in the rectangular region between $a$ and $\sigma_0$
that lie on opposite edges of the rectangle; see the right-hand side of Figure~\ref{fig:merge}.
Then $A = a \times S^1$ is a product annulus in $(N,\gamma_1)$.
If we decompose $(N,\gamma_1)$ along~$A$, then we obtain the disjoint union of $(M,\gamma_0)$
and the sutured manifold $(D^2 \times S^1,\gamma)$, where $\gamma$ consists of four longitudinal sutures.
By \cite[Proposition~9.1]{polytope}, we have $\SFH(D^2 \times S^1,\gamma) \cong V'$.
The isomorphism is uniquely determined by the orientation of the $S^1$ factor,
which gives an orientation of $H_1(D^2 \times S^1)$, and hence allows one to distinguish between
the two $\SpinC$ structures in which the two $\F_2$ summands are supported.
Hence, by Lemma~\ref{lem:PhiAiso}, the gluing map $\Phi_A$ gives a canonical isomorphism
\begin{equation} \label{eqn:gluing}
I_{\mc S} \colon \HFLh(L,P_0) \otimes V' \stackrel{\sim}{\longrightarrow} \HFLh(L,P_1).
\end{equation}
Note that $I_{\mc S}$ only depends on the choice of the product annulus~$A$,
which in turn depends only on the bigon region in the stabilisation.
By Proposition~\ref{prop:pa}, the map $I_{\mc S}$ preserves the Alexander and Maslov gradings
when they are defined.

We write the gluing map $\Phi_{-\xi}$ as a composition of two gluing maps, which is possible according to \cite[Proposition~6.2]{gluingmap}. The first gluing map $s_{\mc S}$ corresponds to the sutured submanifold $(-M,-\gamma_0)$ of $(-M,-\gamma_0) \sqcup (-D^2 \times S^1,-\gamma)$ with the contact structure $-\xi|_{D^2 \times S^1}$ on the difference.
The second gluing map $I_{\mc S}$, which is the one in equation~\eqref{eqn:gluing}, corresponds to the sutured submanifold $(-M,-\gamma_0) \sqcup (-D^2 \times S^1,-\gamma)$ of $(-N,-\gamma_1)$, using the contact structure $-\xi$ on $(-N) \setminus ((-M) \sqcup (-D^2 \times S^1))$.

By construction, the gluing map $s_{\mc S}$ maps $x \in \HFLh(L,P_0)$ to
\[
x \otimes \EH(-\xi|_{D^2 \times S^1}) \in \HFLh(L,P_0) \otimes \SFH(D^2 \times S^1,\gamma).
\]
Recall that $-\xi$ is a positive contact structure on $(-D^2 \times S^1,-\gamma)$.
In case of a positive stabilisation, the dividing set of $-\xi$ on $-D^2 \times \{p\}$ for $p \in S^1$
is given by the left-hand side of  \cite[Figure~14]{gluingmap},
while it is given by the right-hand side in case of a negative stabilisation.
As explained there, the evaluation of the relative first Chern class of the $\SpinC$ structure
of the contact element $\EH(-\xi|_{D^2 \times S^1})$ on the surface $D^2 \times \{p\}$
is given by $\chi(R_+) - \chi(R_-)$, where $R_+$ and $R_-$ are the positive and negative
regions of the complement of the dividing set on $D^2 \times \{p\}$, respectively.
Hence $\EH(-\xi|_{D^2 \times S^1}) = t$ in case of a positive stabilisation
and $\EH(-\xi|_{D^2 \times S^1}) = b$ in case of a negative stabilisation.
\end{proof}

\begin{remark} \label{rem:annulus}
\label{rem:otherproductannulus}
Consider the arc~$a$ on the left-hand side of Figure~\ref{fig:merge},
with the two basepoints lying on the same side of the rectangular region between $\sigma_0$
and $a$. Then the product annulus $A' = a \times S^1$ induces a different isomorphism
\[
I_{\mc S}' \colon \HFLh(L,P_0) \otimes V' \stackrel{\sim}{\longrightarrow} \HFLh(L,P_1)
\]
than the map $I_{\mc S}$ that appears in Proposition~\ref{prop:stab}.
However, if we decompose $(N, \gamma_1)$ along $A'$, then the restriction
of $-\xi$ to the $D^2 \times S^1$ component has the same contact element
in $\SFH(D^2 \times S^1, \gamma)$ as for~$A$, since the convex surface $D^2 \times \{p\}$
carries the same dividing set and decomposition into $R_+$ and $R_-$.
\end{remark}

\begin{figure}
\begin{center}
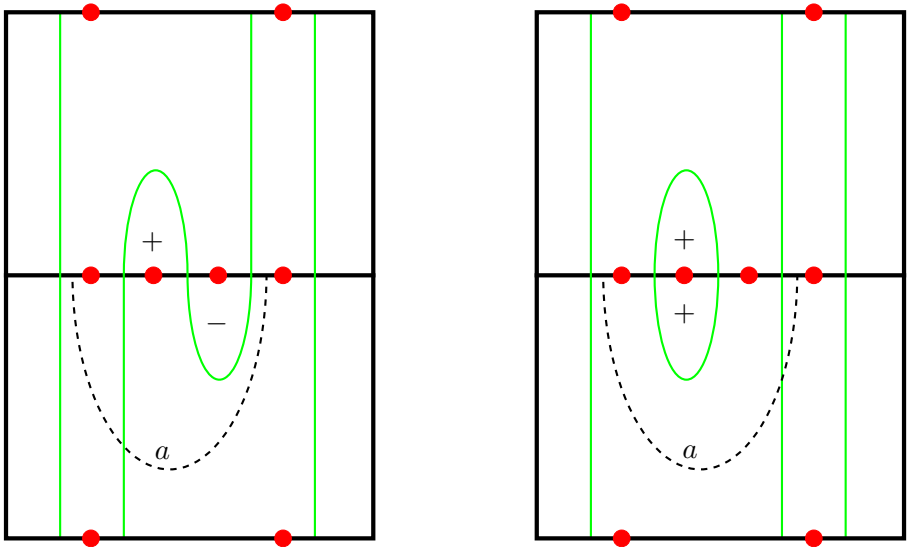
\caption{The composition of a negative stabilisation and positive destabilisation on the left, and the composition of a positive stabilisation and a positive destabilisation on the right.}
\label{fig:merge}
\end{center}
\end{figure}

\begin{proposition}
\label{prop:destab}
Suppose that $\mc D = (Y \times I, F,\sigma)$ is a destabilisation from $(Y,L,P_0)$ to $(Y,L,P_1)$.
Let $V'$ be as in Proposition~\ref{prop:stab}.
Then, there is a canonical isomorphism
\[
I_{\mc D} \colon \HFLh(Y,L,P_0) \stackrel{\sim}{\longrightarrow} \HFLh(Y,L,P_1) \otimes V'
\]
that preserves the bigrading when defined, such that the link cobordism map
\[
F_{\mc D} \colon \HFLh(Y,L,P_0) \to \HFLh(Y,L,P_1)
\]
factorises as $d_{\mc D} \circ I_{\mc D}$, where
\[
d_{\mc D} \colon \HFLh(Y,L,P_1) \otimes V' \to \HFLh(Y,L,P_1)
\]
is given by $d_{\mc D}(x \otimes b) =  x$ and $d_{\mc D}(x \otimes t) = 0$ in case of a positive destabilisation, and by $d_{\mc D}(x \otimes t) = x$ and $d_{\mc D}(x \otimes b) = 0$ in case of a negative destabilisation.
\end{proposition}

\begin{proof}
Given a positive destabilisation $\mc D = (Y \times I, F,\sigma)$,
let $\mc S_- = (Y \times I, F,\sigma_-)$ be a negative stabilisation from $(Y,L,P_1)$ to $(Y,L,P_0)$
such that $\mc D \circ \mc S_-$ exists and is as on the left-hand side of Figure~\ref{fig:merge}.
Then $\mc D \circ \mc S_-$ is the identity cobordism from $(Y,L,P_1)$ to itself,
and hence it induces the identity of $\HFLh(Y,L,P_1)$.

Now let $\mc S_+ = (Y \times I, F, \sigma_+)$ be the positive stabilisation shown on the right-hand side of Figure~\ref{fig:merge}.
Then $F_{\mc D} \circ F_{\mc S_+} = 0$, since the dividing set $\sigma \cup \sigma_+$ contains a closed component that bounds a disc, so the corresponding $S^1$-invariant contact structure is overtwisted, and hence the gluing map vanishes by Honda, Kazez, and Mati\'c~\cite{gluingmap}.

Let $a$ be the properly embedded dashed arc in $F$ shown in Figure~\ref{fig:merge}.
The bigon component~$D$ of $F \setminus a$ contains the bigon components of $F \setminus \sigma_+$ and $F \setminus \sigma_-$ and the three points of~$P_0$ involved in the destabilisation, and intersects both $\sigma_+$ and $\sigma_-$ in two arcs.
Then $A = a \times S^1$ is a product annulus in $(M,\gamma_0) = \W(Y,L,P_0)$.

By Proposition~\ref{prop:stab},
\[
F_{\mc S_+} = I_{\mc S_+} \circ s_{\mc S_+},
\]
where we use the above product annulus~$A$ to define the gluing map
\[
I_{\mc S_+} \colon \HFLh(Y,L,P_1) \otimes V' \to \HFLh(Y,L,P_0).
\]
Let
\[
d_{\mc D} = F_{\mc D} \circ I_{\mc S_+}.
\]
If we set $I_{\mc D} := I_{\mc S_+}^{-1}$, then
\[
F_{\mc D} = d_{\mc D} \circ I_{\mc D}.
\]
Given $x \in \HFLh(L,P_1)$, we obtain that
\[
0 = F_{\mc D} \circ F_{\mc S_+}(x) = d_{\mc D} \circ s_{\mc S_+}(x) = d_{\mc D}(x \otimes t),
\]
as claimed.

In case of $\mc S_-$, we repeat the proof of Proposition~\ref{prop:stab} with the same product annulus~$A$,
as explained in Remark~\ref{rem:otherproductannulus}, to obtain that
\[
F_{\mc S_-} = I_{\mc S_+} \circ s_{\mc S_-}.
\]
Indeed, the dividing set of the $S^1$-invariant contact structure corresponding to~$\sigma_-$ on the bigon~$D$ consists of two arcs such that $\chi(R_+) - \chi(R_-) = -1$. Given $x \in \HFLh(Y,L,P_1)$, we obtain that
\[
x = F_{\mc D} \circ F_{\mc S_-}(x) = d_{\mc D} \circ s_{\mc S_-}(x) = d_{\mc D}(x \otimes b),
\]
which completes the proof when $\mc D$ is a positive destabilisation.
The case when $\mc D$ is a negative destabilisation is analogous.
\end{proof}

\subsection{The maps induced by saddles}
\label{sec:morse-saddles}
\label{sec:mapmergesaddle}

In the following theorem, we give a description of the map
associated to a saddle cobordism~$\mc S = (Y \times I, F, \sigma)$ from $(Y,L_0,P_0)$ to $(Y,L_+,P)$.
We work with \emph{oriented} decorated link cobordisms. Therefore, the two
decorations~$w$ and~$z$ in $P_0 \setminus P$ are a $\bw$-basepoint and a $\bz$-basepoint;
cf.~Definition~\ref{def:mergesaddle} and Section~\ref{sec:origrad}.

\begin{theorem}
\label{thm:exacttriangle}
Let the triple $(Y,L,P)$, $(Y,L_0,P_0)$, $(Y,L_+,P)$ be as in Figure~\ref{fig:surgeryet},
and suppose that they can be oriented coherently. Then the map
\[
f \colon \HFLh(Y,L_0,P_0) \longrightarrow \HFLh(Y,L_+,P)
\]
appearing in the decorated skein exact triangle of Theorem~\ref{thm:skein} coincides with the cobordism map~$F_{\mc S}$
induced by the saddle cobordism~$\mc S$ from~$(Y,L_0,P_0)$ to~$(Y,L_+,P)$.
\end{theorem}

We will prove Theorem~\ref{thm:exacttriangle} in two steps.
We decompose the sutured cobordism~$\W = (W,Z,[\xi])$ from $(M_0,\g_0) = \W(Y,L_0,P_0)$ to
$(M_+,\g) = \W(Y,L_+,P)$ complementary to the saddle cobordism~$\mc S$ into
a \emph{boundary cobordism} $\Wb$ from $(M_0,\g_0)$ to $(N,\g)$
and a \emph{special cobordism} $\Ws$ from $(N,\g)$ to $(M_+,\g)$.
In Proposition~\ref{prop:saddles-b}, we will show that $F_{\Wb} = \Phi_A$,
where $A$ is the product annulus in $(N,\g)$ obtained by gluing the sutures
$\g_0(z)$ and $\g_0(w)$ of $(M_0,\g_0)$ corresponding to the decorations~$w$ and~$z$, respectively.
In Proposition~\ref{prop:saddles-s}, we will prove that $\Ws$ consists of a single
2-handle, and identify the attaching circle. Finally, we prove
Theorem~\ref{thm:exacttriangle} by comparing the 2-handle of $\W^s$ with the surgery involved in
the definition of the map $f$ of the skein exact triangle.

We first consider the boundary cobordism~$\Wb$ associated to a saddle.
Since the surface $F$ is oriented,
\[
N \approx M_0 \cup_{(L_0 \cap S) \times S^1} (S \times S^1),
\]
where $S$ is the square from Definition~\ref{def:mergesaddle}; see Figure~\ref{fig:square}.

\begin{proposition}
\label{prop:saddles-b}
Let $\mc S = (Y \times I, F,\sigma)$ be an oriented saddle cobordism from $(Y, L_0, P_0)$ to $(Y, L_+, P)$.
Let $P_0 \setminus P = \{w,z\}$, and let $\gamma_0(w)$ and $\gamma_0(z)$
denote the corresponding components of $\gamma_0$. If $A$ is the product annulus in $(N,\g)$
obtained by gluing $\gamma_0(w)$ and $\gamma_0(z)$, then
$F_{\Wb} = \Phi_A$ is an isomorphism that preserves the Alexander and the Maslov gradings
when defined.
\end{proposition}

\begin{proof}
Consider the annulus $A = b \times S^1 \subseteq N$,
where $b = I \times \{1/2\}$ is the arc contained in the square $S$ drawn in blue on the right-hand side
of Figure~\ref{fig:square}.
Then $A$ is a \emph{product annulus} in $(N,\g)$, and gives a sutured manifold decomposition
\[
(N, \gamma) \overset{A}{\rightsquigarrow} (M_0, \gamma_0).
\]
By Lemma~\ref{lem:PhiAiso}, the annulus $A$ induces an isomorphism
\[
\Phi_{A} \colon \SFH(M_0, \gamma_0) \to \SFH(N, \gamma).
\]

We claim that $\Phi_A$ coincides with the gluing map $\Phi_{-\xi}$.
Indeed, let $a$ be a pair of arcs adapted to $A$, as in Definition~\ref{def:arcs}.
Then $\Phi_A = \Phi_{A_a} \circ (d_a)_*$. Here $(M_a,\g_a)$ is the result of
decomposing $(N,\g)$ along $A_a$, and $d_a \colon (M_0,\g_0) \to (M_a,\g_a)$ is a diffeomorphism.
Furthermore, $\Phi_{A_a} = \Phi_{\xi_a}$
for a positive contact structure $\xi_a$ on $-(N \setminus M_a)$ with dividing set $-(\g \cup \g_a)$.
Then there is a diffeomorphism $D \in \text{Diff}_0(N, \g)$ such that $D|_{M_0} = d_a$
and $(D|_{N \setminus M_0})_*(-\xi) = \xi_a$.
By the naturality of the gluing maps~\cite[Section~5]{gluingmap}, we have
\[
\Phi_{-\xi} = D_* \circ \Phi_{-\xi} = \Phi_{\xi_a} \circ (d_a)_* = \Phi_A.
\]
By Proposition~\ref{prop:pa}, the map $\Phi_A$ preserves the Alexander and Maslov gradings
when they are defined.
\end{proof}

We now turn our attention to the special cobordism $\Ws$ from $(N,\g)$
to $(M_+,\g)$ associated to a saddle.
Recall that we obtain $(N, \gamma)$ from $(M_0, \gamma_0)$ by gluing
the annuli $\gamma_0(w)$ and $\gamma_0(z)$ with a tube, which we denote by $T$,
whose core is the product annulus $A$.

As we explained in Definition~\ref{def:mergesaddle},
the saddle critical point of $F$ lies inside $S \times I \subset Y \times I$,
where $S \subset Y$ is a square; see Figure~\ref{fig:square}.
We let $B =  F \cap (S \times I)$. Then $d = p_Y|_B \colon B \to S$ is a diffeomorphism.
Note that the result of surgery on $L_0$ along $S$ is $L_+$.
The tube $T$ is the unit normal bundle $UNB \cong B \times S^1$ of $B$ in $\Bl_F(Y \times I)$.
Let $c$ be the core of the square $S$, which joins the basepoints $w$ and $z$ of $L_0$,
and let $c^*$ be its co-core. Furthermore, consider $\td c = d^{-1}(c)$ and $\td c^* = d^{-1}(c^*)$ in $B$.
Then $A = UNB|_{\td c^*} \cong \td c^* \times S^1$; see Figure~\ref{fig:square2}.

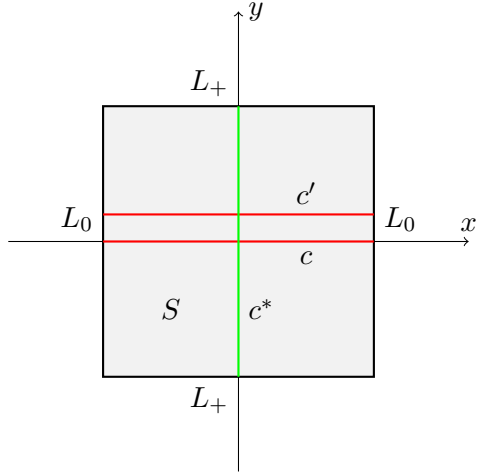
\begin{figure}
\begin{center}
\input{square2.tex}
\caption[{Figure \ref{fig:square2}}]{The figure shows the square $S \subset S^3$ associated to the saddle cobordism,
together with the core $c$ and the co-core $c^*$.}
\label{fig:square2}
\end{center}
\end{figure}

\begin{proposition}
\label{prop:saddles-s}
The special cobordism $\Ws$ corresponding to a saddle cobordism consists of a single 2-handle attached
to $N \times I$ along $N \times \{1\}$.
With the above notation, let $b \subset UNB|_{\td c}$ be a section joining $UNB_w \cap S$ to $UNz \cap S$.
Furthermore, let $c'$ be a parallel translate of $c$ in $S$,
and let $b'$ be a parallel translate of $b$ in $UNB$
that is contained in $UNB|_{d^{-1}(c')}$.
Then the attaching sphere of the 2-handle is the curve $b \cup c$ with framing $b'\cup c'$ in $N$.
\end{proposition}

\begin{proof}
We can suppose that, in a local coordinate system $\R^3 \subset Y$, the square
\[
S = [-1,1] \times [-1,1] \times \{0\},
\]
and that $L_0 \cap (S \times \R) = \{-1,1\} \times [-1,1] \times \{0\}$.
We now introduce a cancelling pair of 4-dimensional 1- and 2-handles
that define the product cobordism $Y \times I$.
The 1-handle $\mf h^1$ is an $\eps$-neighbourhood of the band $B$ for $\eps > 0$  sufficiently small,
and the 2-handle is
\[
\mf h^2 = \left\{\, (x,y,z,t) \in \R^3 \times I \,\middle|\, \exists\, \bar t \geq t \,\colon\, (x,y,z,\bar t) \in B \,\right\}.
\]
If we remove an $\eps$-neighbourhood of $F$ from $Y \times I$, then we obtain a cobordism diffeomorphic to the special cobordism~$\Ws$.
The 1-handle $\mf h^1$ is contained in $\nbd_\eps(F)$, and therefore disappears when we consider $\Ws$. On the other hand, $\mf h^2$ is contained in $\Ws$, except a small portion near the boundary, which is contained in $\nbd_\eps(F)$. An illustration of $\mf h^2$ is available in Figure \ref{fig:handles}.

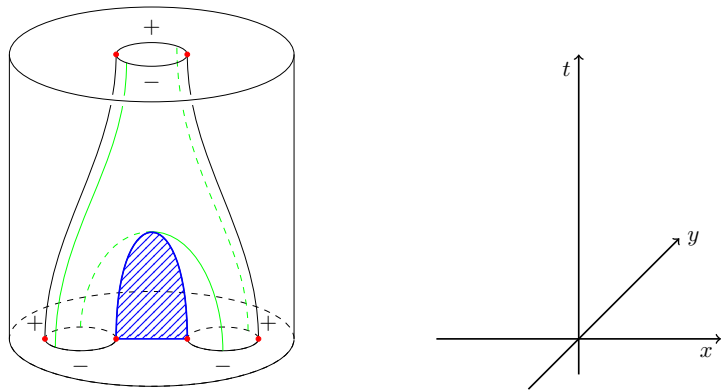
\begin{figure}
\centering
\resizebox{0.7\hsize}{!}{\input{handles1.tex}}
\caption{The figure shows a merge saddle cobordism from $(U_2, P_{U_2})$
-- the two-component unlink with standard decorations -- to $(U_1, P_{U_1})$
-- the unknot with standard decorations.
The core of the 2-handle described in Proposition~\ref{prop:saddles-s} is shown in blue.
The decorated links are embedded in $S^3$, and the cobordism is embedded in $S^3 \times I$.
The figure represents the $\{z=0\}$ section $\R^2 \times \{0\} \times I$
of $\R^3 \times I \subseteq S^3 \times I$, with coordinates as shown above.
Here, the coordinates on $\R^3$ are $x$, $y$, and $z$, and the coordinate on~$I$ is~$t$.
We use this convention every time a (decorated) cobordism
in $S^3 \times I$ is represented enclosed in a cylinder, as above.}
\label{fig:handles}
\end{figure}

The core of the 2-handle $\mf h^2$ is
\[
C = \mf h^2 \cap \{y=z=0\},
\]
whose boundary is $b \cup c$, which is therefore the attaching circle.
To obtain the framing of the attaching circle of $\mf h^2$, consider the disc
\[
C' = \mf h^2 \cap \{z=0, y=\eps\}
\]
in $\mf h^2$ parallel to $C$.
Its boundary, which gives the framing of $\mf h^2$, is $b' \cup c'$.

The height function on $\Ws \setminus \mf h^2$ does not have any critical points, and can be rescaled to be a Morse function. It follows that $\Ws$ consists of the single handle~$\mf h^2$ attached to $N \times I$ along $N \times \{1\}$.
\end{proof}

\begin{proof}[{Proof of Theorem~\ref{thm:exacttriangle}}]
Consider the saddle cobordism $\mc S$ from $(L_0, P_0)$ to $(L_+,P)$.
Then the cobordism map
\[
F_{\mc S} \colon \HFLh(L_0,P_0) \to \HFLh(L_+,P)
\]
is given by the isomorphism $\Phi_A \colon \SFH(M_0,\g_0) \to \SFH(N,\g)$
induced by gluing along the product annulus~$A$ obtained by identifying $\g_0(w)$ and $\g_0(z)$
as in Proposition~\ref{prop:saddles-b}, followed by the 2-handle map $F_{\Ws} \colon \SFH(N,\g) \to \SFH(M_+,\g)$
described in Proposition~\ref{prop:saddles-s}.

Consider the crossing disc~$D$ that intersects~$L$ in two points, and let $K = \partial D$.
As in equation~\eqref{eqn:surgery} in Section~\ref{sec:skein},
the framings $m$, $l$, and~$s$ of $K$ give rise to the surgery exact triangle
\begin{equation}
\dots \longrightarrow \SFH(M_m(K),\gamma) \stackrel{e'}{\longrightarrow} \SFH(M_l(K),\gamma) \stackrel{f'}{\longrightarrow} \SFH(M_s(K),\gamma) \longrightarrow \dots
\end{equation}
Let $A'$ be the product annulus in $(M_l(K),\g)$ obtained by capping off $D \setminus \nbd(K)$.
If we decompose $(M_l(K),\g)$ along $A'$, then we obtain a sutured manifold $(M_*,\g_*)$
complementary to the decorated link $(L_*,P_*)$ shown in the middle of Figure~\ref{fig:L}
(if we do not identify the two spheres).
We obtain $(L_0,P_0)$ from $(L_*,P_*)$ by performing two Reidemeister~I moves. Let
$d_0 \colon (M_*,\g_*) \to (M_0,\g_0)$ be the diffeomorphism induced between the
complementary sutured manifolds.
A similar twist gives a diffeomorphism $d_1 \colon (M_s(K),\g) \to (M_+,\g)$.
If
\[
\Phi_{A'} \colon \SFH(M_*,\g_*) \to \SFH(M_l(K),\gamma)
\]
is the gluing map induced by $A'$, then we obtain the exact triangle
\[
\dots \longrightarrow \HFLh(L,P) \stackrel{e}{\longrightarrow} \HFLh(L_0,P_0) \stackrel{f}{\longrightarrow} \HFLh(L_+,P) \longrightarrow \dots
\]
by setting $e = (d_0)_* \circ \Phi_{A'}^{-1} \circ e'$ and $f = (d_1)_* \circ f' \circ \Phi_{A'} \circ (d_0)_*^{-1}$.
Note that the identifications $d_0$ and $d_1$ were implicit in Theorem~\ref{thm:skein}.
Here we make them explicit to be able to compare different gluing and 2-handle maps.

Let $d \colon (M_l(K),\g) \to (N,\g)$ be the diffeomorphism obtained by
performing a left-handed Dehn twist on $M_l(K)$ along $A'$, as in Figure~\ref{fig:L}.
Then~$d(A') = A$, and the diffeomorphism induced between the decomposed manifolds is~$d_0$,
hence the diagram
\[
  \xymatrix{\SFH(M_*,\g_*) \ar[r]^{(d_0)_*} \ar[d]^{\Phi_{A'}} & \SFH(M_0,\g_0) \ar[d]^{\Phi_A} \\
    \SFH(M_l(K),\g) \ar[r]^{d_*} & \SFH(N,\g).}
\]
is commutative by the naturality of the gluing maps.

The result follows once we show that the diagram
\[
  \xymatrix{\SFH(M_l(K),\g) \ar[r]^{d_*} \ar[d]^{f'} & \SFH(N,\g) \ar[d]^{F_{\Ws}} \\
    \SFH(M_s(K),\g) \ar[r]^{(d_1)_*} & \SFH(M_+,\g)}
\]
is also commutative, since it implies that
\[
f = (d_1)_* \circ f' \circ \Phi_{A'} \circ (d_0)_*^{-1} = F_{\Ws} \circ \Phi_A = F_{\mc S}.
\]
As both $f'$ and $F_{\Ws}$ are 2-handle maps, and $d_1$ is the map induced by~$d$ between the results
of the handle attachments,
it suffices to prove that $d$ maps the framed circle $(m,s)$ for $f'$ to the framed circle $(b \cup c, b' \cup c')$ for $\Ws$
by the naturality of the 2-handle maps~\cite[Theorem~8.2]{cobordisms}.

\begin{figure}
\begin{center}
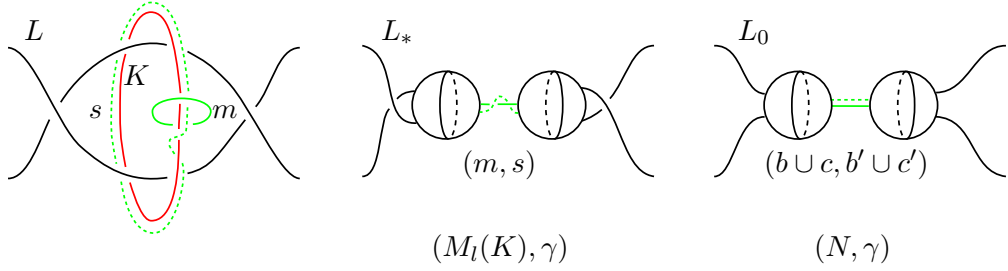
\caption[{Figure \ref{fig:L}}]{The left-hand side of the figure shows the link~$(Y,L,P)$
with the small unknot $K = \partial D$ in red,
the meridian~$m$ in green, and the framing $s$ in dashed green.
In the middle, we see the result of $l$-surgery along~$K$.
The manifold $Y_l(K)$ is represented here as the complement of two balls,
with their boundaries identified via a reflection.
After untwisting~$L_*$ via the diffeomorphism~$d$,
we get the figure on the right, taking the framed knot $(m,s)$ to $(b \cup c, b' \cup c')$.}
\label{fig:L}
\end{center}
\end{figure}

Consider a local projection of $L$ as on the left of Figure~\ref{fig:L}, and let $K$ be the boundary of the crossing disc~$D$.
The disc $D$ can be capped off in $Y_l(K)$ to obtain a sphere $S^2 \subset Y_l(K)$.
If we cut $Y_l(K)$ along this sphere, then the manifold that we obtain is the complement
of the two balls in the middle of Figure~\ref{fig:L}.
To recover $Y_l(K)$, glue the boundaries of the balls via an identification given by a reflection.

Recall that $m$ is a meridian of $K$,
and $s$ is a framing of $K$ such that $m \cdot l = l \cdot s = s \cdot m = 1$.
After performing $l$-surgery along $K$,
the meridian~$m$ becomes isotopic to the core of the solid torus that we glued in during the surgery.
Furthermore, in $M_l(K)$, we can slide $s$ over $K$ to obtain
the green dashed arc in the middle of Figure~\ref{fig:L}.
Performing surgery along the framed knot $(m,s)$ in $(M_l(K),\gamma)$,
we obtain $(M_s(K),\gamma)$.

The diffeomorphism~$d$ twists the left-hand sphere by $\pi$ to the left
and the right-hand sphere by $\pi$ to the right along the horizontal axis
(this extends to an automorphism of $Y_l(K)$),
removing the two crossings of~$L_*$, and untwisting the green pair of arcs~$m$ and~$s$;
see the right-hand side of Figure~\ref{fig:L}.
The twist sends $L_*$ to the link~$L_0$ on the right-hand side of Figure~\ref{fig:L}.

By Proposition~\ref{prop:saddles-s}, the map $F_{\Ws}$ is given by performing surgery on $(N,\gamma)$
along the framed knot $(b \cup c, b' \cup c')$.
On the right of Figure~\ref{fig:L}, the framed knot $(b \cup c, b' \cup c')$ corresponds
to the parallel pair of arcs.
Hence, the diffeomorphism $d$ maps the framed knot $(m,s)$ to $(b \cup c, b' \cup c')$, as required.
\end{proof}

\subsection{A Heegaard diagram arising from a link projection and
a triple diagram for the saddle map}
\label{sec:OSzconstruction}

We will usually only write $(L,P)$ instead of $(S^3,L,P)$.
Analogously -- see Definition \ref{def:dlc} -- we will denote any decorated link cobordism
$(S^3 \times I, F, \sigma)$ from $(L_0, P_0)$ to $(L_1, P_1)$ by $(F, \sigma)$.

Ozsv\'ath and Sza\-b\'o~\cite{HFK1} described how to construct a Heegaard diagram from a knot projection.
We first extend their construction to decorated links in~$S^3$; see Figure~\ref{fig:hopf} for an illustration.
We then give a triple diagram for the 2-handle map~$F_{\Ws}$ corresponding to a saddle cobordism.

Let $(L,P)$ be an oriented decorated $n$-component link in~$S^3$ with $|P|=2\ell$, and
let $\mc L$ be a connected projection of $L$ onto $\R^2 \subseteq \R^3$.
Let $\Sigma$ be the boundary of a regular neighbourhood of $\mc L$ in $\R^3$.
Then $\Sigma$ is a surface of some genus~$g$, oriented as the boundary of its exterior.
We place a basepoint on~$\Sigma$ right above each point of the decoration~$P$.
We label the basepoints alternatingly by~$w_i$ and~$z_i$
such that each component of~$R_-(P)$ is oriented from~$w_i$ to~$z_i$.
We let $\boldsymbol{w} = \{\,w_1,\dots,w_\ell\,\}$ and
$\boldsymbol{z} = \{\,z_1,\dots,z_\ell\,\}$.

For each crossing of $\mc D$, we add a $\beta$-curve on $\Sigma$ around
the crossing as in Figure~\ref{fig:hopf}. We call these curves $\beta_1,\ldots,\beta_{g-1}$.
Furthermore, for each component of~$R_-(P)$, we add a meridional $\beta$-curve
that we call $\beta_g, \ldots, \beta_{g+\ell-1}$.
What we finally get is a set
\[
\bb = \{\,\beta_1, \dots, \beta_{g+\ell-1} \,\}
\]
of attaching circles.

As for the $\alpha$-curves, we draw a curve around each bounded region
of $\R^2 \setminus \mc L$, and call these curves $\alpha_1, \ldots, \alpha_g$.
Furthermore, for each component of $R_-(P)$ except one, we add an inessential
$\alpha$-curve that contains $z_i$ and $w_i$ as in Figure~\ref{fig:hopf}.
We call these ``ladybugs,'' and label them $\alpha_{g+1}, \dots, \alpha_{g+\ell-1}$.
We set
\[
\ba = \{\,\alpha_1, \dots, \alpha_{g+\ell-1} \,\}.
\]

\begin{figure}
\resizebox{0.95\textwidth}{!}{
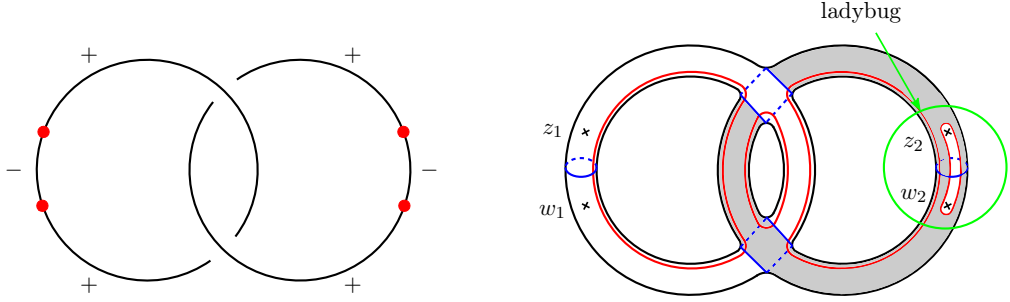
}
\caption{The picture on the left shows a projection of the
Hopf link with two decorations on each component. From this
projection, one can construct the Heegaard diagram on the
right-hand side. The grey region on the right-hand side is
the periodic domain corresponding to the second link component.}
\label{fig:hopf}
\end{figure}

The diagram $(\Sigma, \ba, \bb, \boldsymbol{w}, \boldsymbol{z})$ is
a $2\ell$-pointed Heegaard diagram representing the link~$L$.
A sutured Heegaard diagram for the complement of a link is obtained
from a $2\ell$-pointed Heegaard diagram by removing a small disc
about each of the basepoints~$w_i$ and~$z_i$.
For this reason, we will use $2\ell$-pointed Heegaard diagrams and sutured
Heegaard diagrams for decorated link complements without distinction.

It follows from Proposition~\ref{prop:saddles-s} and from \cite[Section~6]{cobordisms}
that the cobordism map $F_{\Ws}$ can be computed from any triple diagram subordinate to the 2-handle attachment.
We now explain how to construct such a diagram.

\begin{proposition}
\label{prop:saddles-hd}
Let $\Ws$ denote the special cobordism associated to
a saddle cobordism $\mc S = (F, \sigma)$ from $(L_0, P_0)$ to $(L_+, P)$.
Let $\mc H_0$ denote a Heegaard diagram for $(L_0,P_0)$ arising from a connected link projection as above,
which, in a neighbourhood of the square~$S$ corresponding to the saddle,
is shown on the left-hand side of Figure~\ref{fig:hdannulus}.
The right-hand side of Figure~\ref{fig:hdannulus} shows a Heegaard diagram for $(N, \gamma_1)$
obtained by gluing $\g_0(w)$ and $\g_0(z)$. The product annulus~$A$ is spanned by the curve~$a$.

Let $\mc T$ be the triple diagram obtained by locally modifying $\mc H_0$ as in Figure~\ref{fig:HDsaddle},
and by setting all the $\delta$ curves that are not shown to be small Hamiltonian isotopic translates of the respective $\beta$-curves.
Then $\mc T$ is a triple diagram subordinate to the 2-handle described in Proposition~\ref{prop:saddles-s}.
In particular, as in~\cite[Definition~6.8]{cobordisms}, the map~$F_{\Ws}$ is induced by the chain map
\[
f_{\Ws}(\x) = \sum_{\y \in \mb T_{\ba} \cap \mb T_{\bd}} \, \sum_{
					\substack{\psi \in \pi_2(\x, \Theta_{\beta,\delta}, \y)\\
					\mi(\psi)=0}
					} \#\mc M(\psi) \cdot \y.
\]
that counts pseudo-holomorphic triangles in $\mc T$, where $\Theta_{\beta,\delta}$ represents the
top-graded generator of $\SFH(\S,\bb,\bd,\s_0)$.
\end{proposition}

\begin{figure}
\begin{center}
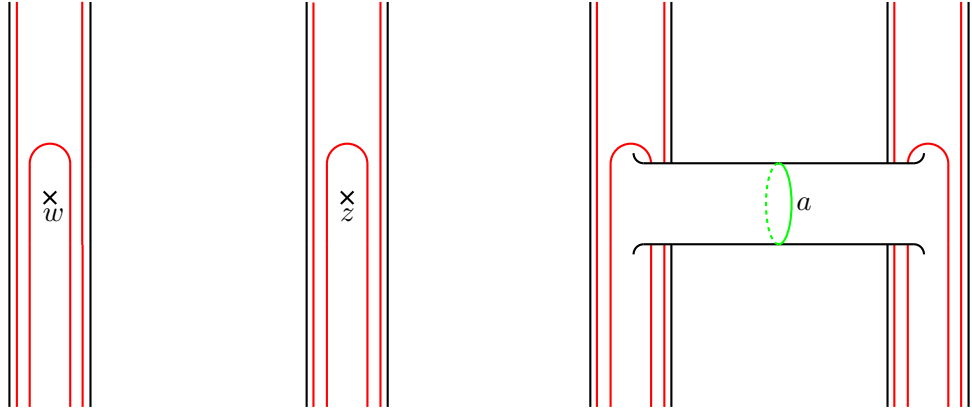
\caption[{Figure \ref{fig:hdannulus}}]{The figure on the left is a local picture of the Heegaard diagram
$\mc H_0$ for $(L_0, P_0)$ arising from a connected link projection.
The figure on the right is a Heegaard diagram for the sutured manifold $(N, \gamma)$ after the boundary cobordism,
obtained by connecting~$w$ and~$z$ by a tube. The core of the tube is the green curve~$a$.}
\label{fig:hdannulus}
\end{center}
\end{figure}

\begin{figure}
\begin{center}
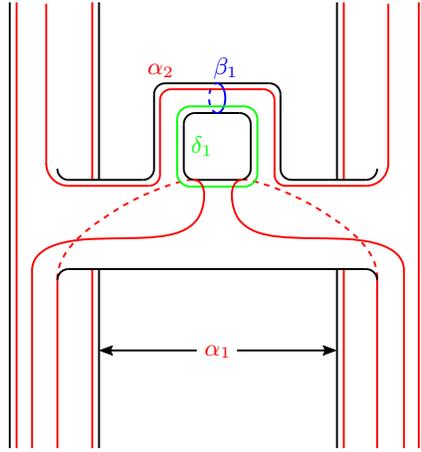
\caption[{Figure \ref{fig:HDsaddle}}]{The triple diagram $\mc T$
subordinate to the 2-handle attachment of the special cobordism associated to a saddle.}
\label{fig:HDsaddle}
\end{center}
\end{figure}

\begin{proof}
Let $\mc H_0 = (\S_0,\ba_0,\bb_0)$ be the diagram of $(L_0,P_0)$ obtained from a connected projection of~$L_0$;
see the left-hand side of Figure~\ref{fig:hdannulus}.
We obtain the diagram $\mc H = (\S,\ba_0,\bb_0)$ of $(N,\g)$ by gluing the components of~$\partial\S_0$
corresponding to the decorations~$w$ and~$z$ with a tube~$\td C$; see the right-hand side of Figure~\ref{fig:hdannulus}.
We suppose that~$\S$ is embedded in~$S^3$ as in the figure.
Let $\S \times I$ be a regular neighbourhood of the Heegaard surface~$\S$ in~$S^3$.
The translate $\S_+ = \S \times \{1\}$ appears on the inside
and $\S_- = \S \times \{0\}$ appears on the outside of~$\Sigma$ of $\S = \S \times \{1/2\}$.

With the notation established before Proposition~\ref{prop:saddles-s},
we have $\td C = UNB|_{\td c}$. We let $\td C_+ = \td C \times \{1\} \subset \Sigma_+$
and $\td C_- = \td C \times \{0\} \subset \Sigma_-$.
By Proposition~\ref{prop:saddles-s}, the attaching circle of the 2-handle in~$\Ws$ is $\mb L = b \cup c$,
where $b \subset \td C$ and $c$ lies on the square~$S \subset M_0$.
Let $B(\mb L)$ be the bouquet obtained by connecting~$\mb L$ to~$\td C_+$
via a straight arc contained entirely in~$\td C \times I$.

Consider now the triple diagram $\mc T = (\S',\ba,\bb,\bd)$ in Figure~\ref{fig:HDsaddle},
which is also embedded in~$S^3$.
In order to construct the sutured manifold associated to $(\Sigma', \ba, \bb)$,
take $\Sigma' \times I$, and attach 3-dimensional 2-handles to $\Sigma'_- = \S' \times \{0\}$
along the $\alpha$-curves and to $\Sigma'_+ = \S' \times \{1\}$ along the $\beta$-curves.
We denote these 2-handles by $H_{\alpha_i}$ and $H_{\beta_i}$.
Given subsets $\ba' \subset \ba$ and $\bb' \subset \bb$,
let $\Sigma'(\ba', \bb')$ be the sutured manifold obtained from $(\Sigma' \times I, \partial \S' \times I)$
by attaching $H_\a$ and $H_\b$ for $\a \in \ba'$ and $\b \in \bb'$.
The manifold $\Sigma' \times I$ can also be viewed as a subset of~$S^3$,
as well as the handle $H_{\alpha_1}$, whose core lies in the plane of the link projection
as it arose from the Ozsv\'ath-Szab\'o diagram described above.
So it follows that $\Sigma'(\{\alpha_1\},\emptyset)$ is also embedded in~$S^3$.
Note that the square~$S$ is contained within $H_{\alpha_1}$.

The manifold $\Sigma'(\{\alpha_1, \alpha_2\},\emptyset)$ is obtained from
$\Sigma'(\{\alpha_1\},\emptyset)$ by attaching the handle~$H_{\alpha_2}$.
The upper half of $\alpha_2$ is parallel to the upper half of $\alpha_1$
along the top of Figure~\ref{fig:HDsaddle} and outside of it,
hence we can slide the upper half of $H_{\alpha_2}$ along~$H_{\alpha_1}$
until the upper boundary becomes a horizontal arc across~$H_{\alpha_1}$,
and $H_{\alpha_2}$ becomes perpendicular to $H_{\alpha_1}$.
After the isotopy, $H_{\alpha_2}$ connects the small handle in the centre of Figure~\ref{fig:HDsaddle}
to $H_{\alpha_1}$.
This implies that $\Sigma'(\{\alpha_1, \alpha_2\},\emptyset)$ is also embedded in~$S^3$,
and we can isotope the small handle in the centre of Figure~\ref{fig:HDsaddle} onto $H_{\alpha_1}$.
Attaching all the other handles along the $\alpha$- and $\beta$-curves except $\beta_1$,
one obtains $\Sigma'(\{\alpha_1, \ldots, \alpha_g\}, \{\beta_2, \ldots, \beta_{g}\})$,
which is the same manifold as the one defined by the Heegaard diagram in Figure~\ref{fig:hdannulus}
-- i.e., $N$ -- with the difference that a neighbourhood of $B(\mb L)$ has been removed from it.
The curve $\beta_1$ is the meridian of $\mb L$, and $\delta_1$ is exactly the framing given
by Proposition~\ref{prop:saddles-s}. Therefore, this is a Heegaard diagram subordinate to the 2-handle attachment.
\end{proof}

\subsection{The map associated to a birth cobordism}
\label{sec:morse-births}
\label{sec:birth}

We now turn to the description of the map associated to a birth cobordism.
Let $V$ denote a 2-dimensional vector space over $\mb F_2$,
generated by two homogeneous elements: $T$ in Alexander grading $0$ and Maslov grading $1/2$,
and $B$ in Alexander grading $0$ and Maslov grading $-1/2$. Then we have the following result.

\begin{theorem}
\label{thm:births}
Suppose that $\mc B$ is a birth cobordism from $(Y, L, P)$ to
\[
(Y, L_1, P_1) = (Y \# S^3, L \sqcup U_1, P \sqcup P_{U_1}),
\]
where $(S^3, U_1, P_{U_1})$ is an unknot with two decorations.
Then there is an isomorphism
\[
\HFLh(Y, L_1, P_1) \cong \HFLh(Y,L,P) \otimes V
\]
such that $F_{\mc B}(x) = x \otimes T$ for every $x \in \HFLh(Y,L,P)$.
\end{theorem}

\begin{proof}
Let $(M, \gamma) = \mc W(Y, L, P)$ and
$(M_1, \gamma_1) = \mc W(Y \# S^3, L \sqcup U_1, P \sqcup P_{U_1})$ denote the sutured manifolds complementary to the links.
As usual, we write the cobordism $\mc W(\mc B)$ as the composition of a boundary cobordism~$\Wb$
from $(M,\gamma)$ to $(N, \gamma_1)$, and a special cobordism~$\Ws$ from $(N, \gamma_1)$ to~$(M_1,\gamma_1)$, where
\[
(N, \gamma_1) = (M, \gamma) \sqcup (S^1 \times D^2, \nu),
\]
and $\nu$ consists of two longitudinal sutures. Note that $\SFH(S^1 \times D^2,\nu) \cong \F_2$, hence
\[
\SFH(N, \gamma_1) \cong \SFH(M, \gamma) \otimes \F_2
\]
in a canonical way by the multiplicativity of sutured Floer homology;
see \cite[Theorem 11.11]{cobordisms}.
Since one component of the boundary cobordism is the identity of $(M, \gamma)$,
by multiplicativity of $\SFH$, the gluing map
\[
\Phi_{-\xi} \colon \SFH(M, \gamma) \to \SFH(M, \gamma) \otimes \mb F_2
\]
is either the zero map or the tautological isomorphism.
Since $\mc B$ has a left inverse, which is a saddle cobordism,
by functoriality the map $\Phi_{-\xi}$ must be the tautological isomorphism.

We now consider the special cobordism $\Ws$, for which we have the following lemma.

\begin{lemma}
\label{lem:specialbirth}
The special cobordism $\Ws$ associated to a birth consists of
a 1-handle attached to $N \times I$ along $N \times \{1\}$
that connects the two components of $N \times \{1\}$.
\end{lemma}

\begin{proof}
Without loss of generality, we can assume that the birth takes place in a ball~$B$ disjoint from~$L$.
We can suppose that the sections $\Ws_t$ of the special cobordism inside $Y \times I$ appear as follows:
\begin{itemize}
\item{for $t \in [0,1/2]$, we have $\Ws = M$,}
\item{for $t \in (1/2,3/4)$, we have $\Ws = M \setminus B$, and}
\item{for $t \in [3/4,1]$, since $U_1 \subseteq B$, we can suppose that, if $t > \bar t$,
      then $\Ws_t \supseteq \Ws_{\bar t}$, and $\Ws_1 = M_1$.}
\end{itemize}

\begin{figure}
\centering
\resizebox{0.3\hsize}{!}{\input{handles2.tex}}
\caption{The figure shows a birth cobordism from $(U_1, P_{U_1})$ to $(U_2, P_{U_2})$.
The core of the 1-handle~$\handle^1$ is shown in blue.}
\label{fig:handles2}
\end{figure}
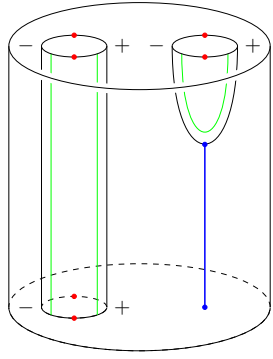

Choose a local coordinate system $\R^3 \subset Y$ containing~$B$. We define
\[
\handle^1 := \left\{\, (x,y,z,t) \in \R^3 \times I \,\colon\, (x,y,z) \in B \text{ and } t \in [0,1/2]\,\right\}.
\]
This is a 1-handle that connects the two components of~$N$.
An illustration of~$\handle^1$ in a particular link cobordism
can be found in Figure~\ref{fig:handles2}.
The 1-handle structure is clear from the diffeomorphism
\[
\handle^1 \cong [0,1/2] \times B \cong D^1 \times D^3.
\]
The height function on $\Ws\setminus\handle^1$ does not have any critical points,
and can be rescaled to be a Morse function.
It follows that $\Ws$ consists of the single 1-handle $\handle^1$.
\end{proof}

Hence, by \cite[Definition 7.5]{cobordisms}, the map $F_{\Ws}$ is
\begin{align*}
F_{\Ws} \colon \SFH(M, \gamma) &\to \SFH(M, \gamma) \otimes V.\\
x &\mapsto x \otimes T
\end{align*}
It follows that the composition $F_{\mc B} = F_{\Ws} \circ \Phi_{-\xi}$ is as claimed.
\end{proof}

\subsection{The map associated to a death cobordism}
\label{sec:morse-deaths}
\label{sec:death}

For a death cobordism $\mc D = (Y \times I, F, \sigma)$ from $(Y \# S^3, L \sqcup U_1, P \sqcup P_{U_1})$ to $(Y, L, P)$,
where $(S^3, U_1, P_{U_1})$ is the unknot with two decorations,
we have a result dual to Theorem~\ref{thm:births}.
As before, $V$ denotes a 2-dimensional vector space over~$\mb F_2$,
generated by two homogeneous elements: $T$ in Alexander grading~$0$ and Maslov grading~$1/2$,
and $B$ in Alexander grading~$0$ and Maslov grading~$-1/2$.

\begin{theorem}
\label{thm:deaths}
Let $\mc D$ be a death cobordism from
\[
(Y, L_0, P_0) = (Y \# S^3, L \sqcup U_1, P \sqcup P_{U_1})
\]
to $(Y, L, P)$. Then there is an isomorphism
\[
\HFLh(Y, L_0, P_0) \cong \HFLh(Y,L,P) \otimes V
\]
such that $F_{\mc D}(x \otimes T) = 0$ and $F_{\mc D}(x \otimes B) = x$ for every $x \in \HFLh(Y,L_0,P_0)$.
\end{theorem}

\begin{proof}
By turning the cobordism $\mc W(\mc D)$ upside down,
we obtain a cobordism $\bar{\W(\mc D)}$ complementary to a birth,
which we can break into a boundary cobordism $\bar{\Wb}$
followed by a special cobordism $\bar{\Ws}$.
By turning this decomposition upside down,
we get a decomposition $\W(\mc D) = \Wb \circ \Ws$, where
\begin{itemize}
\item $\Ws$ is the \emph{special} cobordism obtained by turning $\bar{\Ws}$ upside down;
\item $\Wb$ is a (\emph{not necessarily boundary}) cobordism, obtained by turning $\bar{\Wb}$ upside down.
\end{itemize}
Let $(M_0,\g_0) = \mc W(Y,L_0,P_0)$ and $(M,\g) = \mc W(Y,L,P)$.
Then $\Ws$ is a cobordism from $(M_0, \gamma_0) \cong (M, \gamma) \# (S^1 \times D^2, \nu)$
to $(M, \gamma) \sqcup (S^1 \times D^2, \nu)$, where $\nu$ consists of two longitudinal sutures.
The map $F_{\Ws}$ is dual to $F_{\bar{\Ws}}$ by \cite[Theorem~11.8]{cobordisms}.
More explicitly, $\Ws$ consists of a single 3-handle
dual to the 1-handle $\handle^1$ from Lemma~\ref{lem:specialbirth}.
By \cite[Definition~7.8]{cobordisms}, there is an isomorphism
\[
\SFH(M_0, \gamma_0) \cong \SFH(M, \gamma) \otimes V
\]
such that the map
\[
F_{\Ws} \colon \SFH(M, \gamma) \otimes V \to \SFH(M, \gamma) \otimes \mb F_2
\]
is given by $F_{\Ws}(x \otimes T) = 0$ and $F_{\Ws}(x \otimes B) = x \otimes 1$
for all $x \in \SFH(M, \gamma)$. Note that, as in Section~\ref{sec:morse-births},
we are identifying $\SFH((M, \gamma) \sqcup (S^1 \times D^2, \nu))$
with $\SFH(M, \gamma) \otimes \mb F_2$ via the canonical isomorphism
given by the multiplicativity of sutured Floer homology.

We now need to understand the map
\[
F_{\Wb} \colon \SFH(M, \gamma) \otimes \mb F_2 \to \SFH(M, \gamma).
\]
By multiplicativity of sutured Floer homology, this is either the tautological isomorphism or the zero map.
Since $\mc D$ has a right inverse, which is a saddle preceded by two stabilisations (see Figure~\ref{fig:saddle}),
the map $F_{\Wb}$ is the tautological isomorphism.
Hence the composition $F_{\mc D} = F_{\Wb} \circ F_{\Ws}$ is as claimed.
\end{proof}

\subsection{Duality}
\label{sec:duality}

We conclude this section with a few remarks about duality and gradings.
If $(Y \times I, F, \sigma)$ is a decorated link cobordism from $(Y, L_0, P_0)$ to
$(Y, L_1, P_1)$ (or more generally, a sutured cobordism from $(M_0, \gamma_0)$
to $(M_1, \gamma_1)$), then one can turn it upside down
to get a decorated link cobordism from $(-Y, L_1, P_1)$
to $(-Y, L_0, P_0)$ (or more generally a sutured cobordism
from $(-M_1, \gamma_1)$ to $(-M_0, \gamma_0)$).

In \cite[Theorem~11.9]{cobordisms}, the first author proved
that if you turn a \emph{special} cobordism upside down, then
the map induced on $\SFH$ is the dual map, and asked \cite[Question~11.10]{cobordisms}
whether this is true for any sutured cobordism. A weaker question
is whether this is true at least for decorated link cobordisms.
An approach to answer this (weaker) question would be to check
what happens for elementary cobordisms.

The positive (respectively negative) stabilisation and destabilisation
are obtained from each other by turning the cobordism upside down,
and we already checked that the maps that they induce are dual
to each other.
Analogously, the birth and the death cobordisms induce maps that are
dual to each other.

\begin{figure}
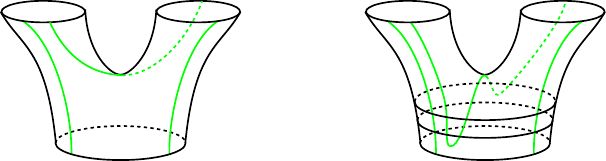
\caption{If we turn a saddle cobordism upside down, then we can write it as a composition of two
stabilizations, followed by another saddle cobordism.}
\label{fig:saddle}
\end{figure}
The situation is more complicated for the saddle maps.
First, note that if we turn a saddle cobordism upside down,
then the height function has a local minimum on~$\sigma$ as opposed to
a maximum, so this is not an elementary saddle cobordism. However,
as in Figure~\ref{fig:saddle}, we can write this as a composition of
two stabilizations and another saddle cobordism (which is a split saddle if the
original cobordism was a merge saddle, and vice versa).

Let $\mc W$ be the sutured cobordism complementary to a saddle.
As usual, we write it as $\Ws \circ \Wb$, where
\begin{itemize}
\item{$\Wb$ is a boundary cobordism from $(M_0, \gamma_0)$ to $(N, \gamma_1)$,
given by gluing along a product annulus, and}
\item{$\Ws$ is a special cobordism from $(N, \gamma_1)$ to $(M_1, \gamma_1)$
    that consists of a single 2-handle.}
\end{itemize}
By turning this decomposition upside down, we obtain a decomposition
for the upside down sutured cobordism $\bar\W$ as $\bar{\W}^b \circ \bar{\W}^s$, where
\begin{itemize}
\item{$\bar{\W}^s$ is a special cobordism from $(-M_1, \gamma_1)$ to $(-N, \gamma_1)$
    (consisting of the 2-handle dual to the one for $\Ws$), and}
\item{$\bar{\W}^b$ is a cobordism from $(-N, \gamma_1)$ to $(-M_0, \gamma_0)$
    that is \emph{not} a boundary cobordism.}
\end{itemize}
The map $F_{\Ws}$ induced by the special cobordism is dual to
the map $F_{\bar{\W}^s}$ by \cite[Theorem~11.9]{cobordisms}.

The map $F_{\bar{\W}^b}$ is not so clear.
If one could prove that this is the inverse of
the isomorphism induced by the product annulus decomposition
\[
(-N, \gamma_1) \rightsquigarrow (-M_0, \gamma_0),
\]
then the upside down cobordism would induce the dual map.
Therefore, the question about duality in the decorated link category
boils down to the following.

\begin{question}
\label{qn:duality}
Let $(N, \gamma_1) \overset{A}{\rightsquigarrow} (M_0, \gamma_0)$
be a product annulus decomposition, and let $\mc W$ be the boundary cobordism
from $(M_0, \gamma_0)$ to $(N, \gamma_1)$ given by gluing along the product annulus~$A$.
Then do we have $F_{\bar\W} = \Phi_A^{-1}$?
\end{question}

\subsection{Gradings}
\label{sec:morse-gradings}
\label{sec:gradings}

Let $(Y,L,P)$ be a null-homologous, oriented, decorated link
in a 3-manifold $Y$ and $\t \in \SpinC(Y)$ a torsion $\SpinC$ structure.
Recall that $\HFLh(Y,L,P,\t)$ can naturally be endowed with two gradings:
the Alexander grading, denoted by $\A$,
and the Maslov grading, denoted by $\agr$. They were both defined in Section~\ref{sec:origrad}.
We will also consider the $\delta$-grading, defined as $\delta = \A-\agr$,
following the convention established by Manolescu and Ozsv\'ath \cite{manolescu2007khovanov}.

\begin{table}%
\def\arraystretch{1.5}\tabcolsep=10pt
\begin{tabular}{l|cc}
 & Alexander grading & Maslov grading\\ \hline
Positive (de)stabilisation & $+\frac12$ & $+\frac12$\\
Negative (de)stabilisation & $-\frac12$ & $-\frac12$\\
Birth or death & $0$ & $+\frac{1}{2}$\\
Saddle & $0$ & $-\frac{1}{2}$\\
\end{tabular}
\vspace{5pt}
\caption[{Table \ref{tab:gradings}}]{The table shows the shift in the Alexander grading and the Maslov grading
associated to each elementary cobordism.}
\label{tab:gradings}
\end{table}

\begin{theorem}
\label{thm:grading}
Let $Y$ be a closed oriented 3-manifold and $\t \in \SpinC(Y)$ a torsion $\SpinC$ structure.
Furthermore, let $(Y,L_0,P_0)$ and $(Y,L_1,P_1)$ be oriented, null-homologous, decorated links.
If $\mc X = (Y \times I, F, \sigma)$ is a decorated link cobordism,
then the map
\[
F_{\mc X} \colon \HFLh(Y,L_0,P_0,\t) \to \HFLh(Y,L_1,P_1,\t)
\]
is homogeneous with respect to both the Alexander and the Maslov gradings.
Furthermore, if $\mc X$ is an elementary cobordism,
then the grading shifts are as given in Table~\ref{tab:gradings}.
Consequently, $F_{\mc X}$ shifts the Maslov grading by
\begin{equation} \label{eq:Mshift}
\frac{\chi(F) + \chi(R_+(\sigma)) - \chi(R_-(\sigma))}{2},
\end{equation}
and the Alexander grading by
\begin{equation} \label{eq:Ashift}
\frac{\chi(R_+(\sigma)) - \chi(R_-(\sigma))}{2}.
\end{equation}
\end{theorem}

\begin{remark}
As $\delta = \A-\agr$, an immediate consequence of Theorem~\ref{thm:grading} is that $F_{\mc X}$ is also $\delta$-homogeneous, and that the $\delta$-grading shift is $-\chi(F)/2$, and therefore does not depend on the decorations.
\end{remark}

\begin{proof}
Since every decorated cobordism in $Y \times I$ is a product of elementary cobordisms by Proposition~\ref{prop:generation},
and formulas~\eqref{eq:Mshift} and~\eqref{eq:Ashift} are additive under composition of decorated link cobordisms, it is sufficient to prove that the maps associated to elementary cobordisms are homogeneous.

By Proposition~\ref{prop:stab}, a stabilization map $F_{\mc S}$ factorizes as $I_{\mc S} \circ s_{\mc S}$,
where $I_{\mc S}$ is homogeneous and preserves the bigrading. Furthermore,
$s_{\mc S}(x) = x \otimes t$ in case of a positive stabilisation, and $s_{\mc S}(x) = x \otimes b$
in case of a negative stabilisation.
The bigrading of $t$ is $(1/2,1/2)$, and the bigrading of~$b$ is $(-1/2,-1/2)$,
which agrees with the grading shift.
By Proposition~\ref{prop:destab}, a destabilisation map $F_{\mc D}$ is of the form $d_{\mc D} \circ I_{\mc D}$,
where $I_{\mc D}$ is homogeneous and preserves the bigrading.
Furthermore, the grading shift of $d_{\mc D}$ is $(1/2,1/2)$ for a positive destabilisation,
and $(-1/2,-1/2)$ for a negative destabilisation.
Note that the grading shifts of stabilisations and destabilisations agree since they are dual to each other.

By Theorem \ref{thm:births}, the birth cobordism map is induced by the tensor product with the element~$T$
of the vector space~$V$ generated by~$T$ in bigrading $(0,1/2)$
and~$B$ in bigrading $(0,-1/2)$.
Therefore the map is homogeneous and the grading shift is $(\A, \agr) = (0, 1/2)$.
The death cobordism map is dual to the birth cobordism map (see Theorem \ref{thm:deaths}), so the grading shift is the same.

By Theorem~\ref{thm:exacttriangle}, the saddle map agrees with the map~$f$ in the skein exact
sequence in Theorem~\ref{thm:skein}, which preserves the Alexander grading and decreases the
Maslov grading by~$1/2$.

Finally, isotopies induce diffeomorphism maps, which preserve both the Alexander and Maslov gradings.
\end{proof}

%% file: square.tex
\begin{tikzpicture}[cross line/.style={preaction={draw=white, -, line width=8pt}}]

\def\u{1cm};
\def\d{\u/2};
\def\e{\u/4};
\def\h{\u*0.2};
\def\rad{\u/12};
\def\cdx{4*\u};

\begin{scope}[shift={(-\cdx,0)}, rotate =90]
\begin{scope}[rotate=90]
\draw[thick, cross line, -] (-\u,-\u) .. controls (-\d,-\d) and (-\d,\d) .. (-\u,\u);
\draw[thick, cross line, -] (\u,\u) .. controls (\d,\d) and (\d,-\d) .. (\u,-\u);
\draw[color=red, fill=red] (-\u*0.63, 0) circle (\rad);
\draw[color=red, fill=red] (\u*0.63, 0) circle (\rad);
\end{scope}
\draw (-\d,\d) node[anchor=north] {\footnotesize $-$};
\draw (-\d,-\d) node[anchor=north] {\footnotesize $-$};
\draw (\d,\d) node[anchor=south] {\footnotesize $+$};
\draw (\d,-\d) node[anchor=south] {\footnotesize $+$};
\end{scope}
\draw (-\cdx,-\u) node [anchor=north] {\small $L_0$};

\begin{scope}[shift={(0,0)}, rotate=90]
\draw[thick, cross line, -] (-\u,-\u) .. controls (-\d,-\d) and (-\d,\d) .. (-\u,\u);
\draw[thick, cross line, -] (\u,\u) .. controls (\d,\d) and (\d,-\d) .. (\u,-\u);
\draw (-\d*1.2,0) node[anchor=north] {\footnotesize $-$};
\draw (\d*1.2,0) node[anchor=south] {\footnotesize $+$};
\end{scope}
\draw (0,-\u) node [anchor=north] {\small $L_1$};

\begin{scope}[shift={(\cdx,0)}, rotate=90]

\begin{scope}[rotate=90]
\draw[thick, -] (-\u,-\u) .. controls (-\d,-\d) and (-\d,\d) .. (-\u,\u);
\draw[thick, -] (\u,\u) .. controls (\d,\d) and (\d,-\d) .. (\u,-\u);
\draw[-] (-\u,-\u) .. controls (0,0) and (0,0) .. (-\u,\u);
\draw[-] (\u,\u) .. controls (0,0) and (0,0) .. (\u,-\u);
\draw[-] (-\u,-\u) .. controls (-\e,-\e) and (-\e,\e) .. (-\u,\u);
\draw[-] (\u,\u) .. controls (\e,\e) and (\e,-\e) .. (\u,-\u);
\end{scope}

\draw[thick, -] (-\u,-\u) .. controls (-\d,-\d) and (-\d,\d) .. (-\u,\u);
\draw[thick, -] (\u,\u) .. controls (\d,\d) and (\d,-\d) .. (\u,-\u);
\draw[color=red, fill=red] (0,-\u*0.63) circle (\rad);
\draw[color=red, fill=red] (0,\u*0.63) circle (\rad);

\draw[-] (-\u,-\u) .. controls (0,0) and (0,0) .. (-\u,\u);
\draw[-] (\u,\u) .. controls (0,0) and (0,0) .. (\u,-\u);

\draw[-] (-\u,-\u) .. controls (-\e,-\e) and (-\e,\e) .. (-\u,\u);
\draw[-] (\u,\u) .. controls (\e,\e) and (\e,-\e) .. (\u,-\u);

\draw[-] (-\u,-\u) -- (\u,\u);
\draw[-] (\u,-\u) -- (-\u,\u);

\draw (0,-\u*0.75) node [anchor=west] {\footnotesize $t=0$};
\draw (\u*0.75,0) node [anchor=south] {\footnotesize $t=1$};

\draw[thick, color=green, -] (\h, \u*0.63) .. controls (\h, 0) and (-\h, 0) .. (-\h, -\u*0.63);
\draw[thick, color=blue] (-\u*0.63,0) -- (\u*0.63,0);
\end{scope}

\end{tikzpicture}

%% file: merge.pdf_tex
\begingroup%
  \makeatletter%
  \providecommand\color[2][]{%
    \errmessage{(Inkscape) Color is used for the text in Inkscape, but the package 'color.sty' is not loaded}%
    \renewcommand\color[2][]{}%
  }%
  \providecommand\transparent[1]{%
    \errmessage{(Inkscape) Transparency is used (non-zero) for the text in Inkscape, but the package 'transparent.sty' is not loaded}%
    \renewcommand\transparent[1]{}%
  }%
  \providecommand\rotatebox[2]{#2}%
  \ifx\svgwidth\undefined%
    \setlength{\unitlength}{340.00000248bp}%
    \ifx\svgscale\undefined%
      \relax%
    \else%
      \setlength{\unitlength}{\unitlength * \real{\svgscale}}%
    \fi%
  \else%
    \setlength{\unitlength}{\svgwidth}%
  \fi%
  \global\let\svgwidth\undefined%
  \global\let\svgscale\undefined%
  \makeatother%
  \begin{picture}(1,0.60235294)%
    \put(0,0){\includegraphics[width=\unitlength,page=1]{merge.pdf}}%
    \put(0.16722689,0.09697481){\color[rgb]{0,0,0}\makebox(0,0)[lb]{\smash{$a$}}}%
    \put(0.15225666,0.33050395){\color[rgb]{0,0,0}\makebox(0,0)[lb]{\smash{$+$}}}%
    \put(0.22334074,0.2398049){\color[rgb]{0,0,0}\makebox(0,0)[lb]{\smash{$-$}}}%
    \put(0.75227159,0.098982){\color[rgb]{0,0,0}\makebox(0,0)[lb]{\smash{$a$}}}%
    \put(0.74195248,0.33291256){\color[rgb]{0,0,0}\makebox(0,0)[lb]{\smash{$+$}}}%
    \put(0.74195248,0.25055976){\color[rgb]{0,0,0}\makebox(0,0)[lb]{\smash{$+$}}}%
  \end{picture}%
\endgroup%

%% file: square2.tex
\begin{tikzpicture}

\def\u{1.8cm};
\def\p{\u/5};
\definecolor{light-gray}{gray}{0.95}
\usetikzlibrary{positioning}

\draw [fill=light-gray, thick] (-\u,-\u) -- (-\u, \u) -- (\u, \u) -- (\u, -\u) -- cycle;
\draw [->] (0,-1.7*\u) -- (0,1.7*\u) node[right] {$y$};
\draw [->] (-1.7*\u,0) -- (1.7*\u,0) node[above] {$x$};
\draw[color=red, thick] (-\u,0) -- (\u,0);
\draw[color=red, thick] (-\u,\p) -- (\u,\p);
\draw[color=green, thick] (0,-\u) -- (0, \u);
\draw (\u, 0) node [above right]{$L_0$};
\draw (-\u, 0) node [above left]{$L_0$};
\draw (0,-\u) node [below left]{$L_+$};
\draw (0, \u) node [above left]{$L_+$};
\draw (\u/2,0) node [below]{$c$};
\draw (\u/2, \p) node [above]{$c'$};
\draw (0, -\u/2) node[right]{$c^*$};
\draw (-\u/2, -\u/2) node{$S$};

\end{tikzpicture}

%% file: handles1.tex
\begin{tikzpicture}[crossline/.style={preaction={draw=white, -, line width=5pt}}]
\usetikzlibrary{patterns}

\def\u{0.6cm};
\def\h{4*\u};
\def\v{4*\u};
\def\b{3*\u};
\def\c{4*\u};
\def\hshift{12*\u};
\def\d{0.23570226039*\u};
\def\smallradius{\u/15};

\begin{scope}[shift={(0,-\v)}]
\draw[thick, ->] (0,-\u) -- (0,2*\v);
\draw[thick, ->] (180:\h) -- (0:\h);
\draw[thick, ->] (225:\h/2) -- (45:\h);

\draw (0, 2*\v) node[anchor=north east]{$t$};
\draw (0:\h) node[anchor=north east]{$x$};
\draw (45:\h) node[anchor=west]{$y$};
\end{scope}

\begin{scope}[shift={(-\hshift,0)}]

\begin{scope}[shift={(0,0-\v)}]
\def\halfradius{\h};
\draw (0,0) ellipse (\halfradius*1 and \halfradius/3);
\draw[white, fill=white] (0-1.2*\halfradius,0) rectangle (1.2*\halfradius,\halfradius/2);
\draw[dashed] (0,0) ellipse (\halfradius*1 and \halfradius/3);
\end{scope}

\begin{scope}[shift={(\h/2,0-\v)}]
\def\halfradius{\u};
\draw (0,0) ellipse (\halfradius*1 and \halfradius/3);
\draw[white, fill=white] (0-1.2*\halfradius,0) rectangle (1.2*\halfradius,\halfradius/2);
\draw[dashed] (0,0) ellipse (\halfradius*1 and \halfradius/3);
\draw (0,-\halfradius/3) node[anchor=north] {$-$};
\draw (0.8*\halfradius,0) node[anchor=south west] {$+$};
\end{scope}
\begin{scope}[shift={(-\h/2,0-\v)}]
\def\halfradius{\u};
\draw (0,0) ellipse (\halfradius*1 and \halfradius/3);
\draw[white, fill=white] (0-1.2*\halfradius,0) rectangle (1.2*\halfradius,\halfradius/2);
\draw[dashed] (0,0) ellipse (\halfradius*1 and \halfradius/3);
\draw (0,-\halfradius/3) node[anchor=north] {$-$};
\draw (-0.8*\halfradius,0) node[anchor=south east] {$+$};
\end{scope}

\begin{scope}[yscale=-1]
\draw (-\h/2-\u,\v) .. controls (-\h/2-\u,\v-\b) and (-\u,-\v+\b) .. (-\u,-\v);
\draw (\h/2+\u,\v) .. controls (\h/2+\u,\v-\b) and (\u,-\v+\b) .. (\u,-\v);
\draw (-\h/2+\u,\v) .. controls (-\h/2+\u,\v-\c) and (\h/2-\u,\v-\c) .. (\h/2-\u,\v);

\draw[dashed,color=green] (\h/2+3*\d,\v-\d) .. controls (\h/2+3*\d,\v-\d-0.9*\b) and (3*\d,-\v-\d+1.2*\b) .. (3*\d,-\v-\d);
\begin{scope}[yshift=2*\d]
\draw[color=green] (-\h/2-3*\d,\v-\d) .. controls (-\h/2-3*\d,\v-\d-1.1*\b) and (-3*\d,-\v-\d+1*\b) .. (-3*\d,-\v-\d);
\end{scope}
\draw[dashed, color=green] (-\h/2,\v-\u/3) .. controls  (-\h/2,\v-3*\u/3) and (-1.5*\u,\v-0.753*\c) .. (0,\v-0.753*\c);
\draw[color=green] (\h/2,\v+\u/3) .. controls  (\h/2,\v-4*\u/3) and (1.2*\u,\v-0.753*\c) .. (0,\v-0.753*\c);
\end{scope}

\begin{scope}[shift={(0,\v)}]
\def\halfradius{\h};
\draw[crossline] (0,0) ellipse (\halfradius*1 and \halfradius/3);
\end{scope}

\begin{scope}[shift={(0,\v)}]
\def\halfradius{\u};
\draw (0,0) ellipse (\halfradius*1 and \halfradius/3);
\draw (0,-\halfradius/3) node[anchor=north] {$-$};
\draw (0,\halfradius/3) node[anchor=south] {$+$};
\end{scope}

\draw (-\h, \v) -- (-\h, -\v);
\draw (\h, \v) -- (\h, -\v);

\begin{scope}[]
\draw[thick,color=blue] (-\h/2+\u,-\v) .. controls (-\h/2+\u,-\v+\c) and (\h/2-\u,-\v+\c) .. (\h/2-\u,-\v);
\draw[thick,color=blue] (-\h/2+\u,-\v) --(\h/2-\u,-\v);
\clip  (-\h/2+\u,-\v) .. controls (-\h/2+\u,-\v+\c) and (\h/2-\u,-\v+\c) .. (\h/2-\u,-\v) -- cycle;
\draw[pattern color=blue, pattern = north east lines] (-\h/2-2*\u, -2*\v) rectangle (\h/2+2*\u, -\v+2*\c);
\end{scope}

\begin{scope}[shift={(0,-\v)}]
\draw[color=red,fill=red] (-\h/2-\u, 0) circle (\smallradius);
\draw[color=red,fill=red] (-\h/2 + \u,0) circle (\smallradius);
\draw[color=red,fill=red] (\h/2 + \u,0) circle (\smallradius);
\draw[color=red,fill=red] (\h/2 -\u,0) circle (\smallradius);
\end{scope}
\begin{scope}[shift={(0,\v)}]
\draw[color=red,fill=red] (-\u,0) circle (\smallradius);
\draw[color=red,fill=red] (\u,0) circle (\smallradius);
\end{scope}

\end{scope}

\end{tikzpicture} 

%% file: L.pdf_tex
\begingroup%
  \makeatletter%
  \providecommand\color[2][]{%
    \errmessage{(Inkscape) Color is used for the text in Inkscape, but the package 'color.sty' is not loaded}%
    \renewcommand\color[2][]{}%
  }%
  \providecommand\transparent[1]{%
    \errmessage{(Inkscape) Transparency is used (non-zero) for the text in Inkscape, but the package 'transparent.sty' is not loaded}%
    \renewcommand\transparent[1]{}%
  }%
  \providecommand\rotatebox[2]{#2}%
  \ifx\svgwidth\undefined%
    \setlength{\unitlength}{376.10243164bp}%
    \ifx\svgscale\undefined%
      \relax%
    \else%
      \setlength{\unitlength}{\unitlength * \real{\svgscale}}%
    \fi%
  \else%
    \setlength{\unitlength}{\svgwidth}%
  \fi%
  \global\let\svgwidth\undefined%
  \global\let\svgscale\undefined%
  \makeatother%
  \begin{picture}(1,0.25978699)%
    \put(0,0){\includegraphics[width=\unitlength,page=1]{L.pdf}}%
    \put(0.11542771,0.17575958){\color[rgb]{0,0,0}\makebox(0,0)[lb]{\smash{\textsl{$K$}}}}%
    \put(0,0){\includegraphics[width=\unitlength,page=2]{L.pdf}}%
    \put(0.42424802,0.0044533){\color[rgb]{0,0,0}\makebox(0,0)[lb]{\smash{\textsl{$(M_l(K),\gamma)$}}}}%
    \put(0.45493729,0.08925471){\color[rgb]{0,0,0}\makebox(0,0)[lb]{\smash{\textsl{$(m,s)$}}}}%
    \put(0,0){\includegraphics[width=\unitlength,page=3]{L.pdf}}%
    \put(0.80785401,0.0044533){\color[rgb]{0,0,0}\makebox(0,0)[lb]{\smash{\textsl{$(N,\gamma)$}}}}%
    \put(0,0){\includegraphics[width=\unitlength,page=4]{L.pdf}}%
    \put(0.37403626,0.22132717){\color[rgb]{0,0,0}\makebox(0,0)[lb]{\smash{\textsl{$L_*$}}}}%
    \put(0.75783375,0.08925471){\color[rgb]{0,0,0}\makebox(0,0)[lb]{\smash{\textsl{$(b \cup c, b' \cup c')$}}}}%
    \put(0.7311728,0.22132717){\color[rgb]{0,0,0}\makebox(0,0)[lb]{\smash{\textsl{$L_0$}}}}%
    \put(0,0){\includegraphics[width=\unitlength,page=5]{L.pdf}}%
    \put(0.08954504,0.1326112){\color[rgb]{0,0,0}\makebox(0,0)[lb]{\smash{}}}%
    \put(0.08106535,0.1431716){\color[rgb]{0,0,0}\makebox(0,0)[lb]{\smash{\textsl{$s$}}}}%
    \put(0.20472885,0.1432567){\color[rgb]{0,0,0}\makebox(0,0)[lb]{\smash{\textsl{$m$}}}}%
    \put(0.01604835,0.22132717){\color[rgb]{0,0,0}\makebox(0,0)[lb]{\smash{\textsl{$L$}}}}%
  \end{picture}%
\endgroup%

%% file: hopf.pdf_tex
\begingroup%
  \makeatletter%
  \providecommand\color[2][]{%
    \errmessage{(Inkscape) Color is used for the text in Inkscape, but the package 'color.sty' is not loaded}%
    \renewcommand\color[2][]{}%
  }%
  \providecommand\transparent[1]{%
    \errmessage{(Inkscape) Transparency is used (non-zero) for the text in Inkscape, but the package 'transparent.sty' is not loaded}%
    \renewcommand\transparent[1]{}%
  }%
  \providecommand\rotatebox[2]{#2}%
  \ifx\svgwidth\undefined%
    \setlength{\unitlength}{471.88568054bp}%
    \ifx\svgscale\undefined%
      \relax%
    \else%
      \setlength{\unitlength}{\unitlength * \real{\svgscale}}%
    \fi%
  \else%
    \setlength{\unitlength}{\svgwidth}%
  \fi%
  \global\let\svgwidth\undefined%
  \global\let\svgscale\undefined%
  \makeatother%
  \begin{picture}(1,0.2900238)%
    \put(0,0){\includegraphics[width=\unitlength,page=1]{hopf.pdf}}%
    \put(0.07347712,0.23501003){\color[rgb]{0,0,0}\makebox(0,0)[lb]{\smash{$+$}}}%
    \put(-0.00150557,0.11983524){\color[rgb]{0,0,0}\makebox(0,0)[lb]{\smash{$-$}}}%
    \put(0.07347702,0.00491596){\color[rgb]{0,0,0}\makebox(0,0)[lb]{\smash{$+$}}}%
    \put(0.33695438,0.23512445){\color[rgb]{0,0,0}\makebox(0,0)[lb]{\smash{$+$}}}%
    \put(0.33695434,0.00503039){\color[rgb]{0,0,0}\makebox(0,0)[lb]{\smash{$+$}}}%
    \put(0.41370896,0.11983526){\color[rgb]{0,0,0}\makebox(0,0)[lb]{\smash{$-$}}}%
    \put(0.5353855,0.16221842){\color[rgb]{0,0,0}\makebox(0,0)[lb]{\smash{$z_1$}}}%
    \put(0.53221517,0.08132688){\color[rgb]{0,0,0}\makebox(0,0)[lb]{\smash{$w_1$}}}%
    \put(0.81362311,0.27745975){\color[rgb]{0,0,0}\makebox(0,0)[lb]{\smash{$\mbox{ladybug}$}}}%
    \put(0,0){\includegraphics[width=\unitlength,page=2]{hopf.pdf}}%
    \put(0.89563527,0.148777){\color[rgb]{0,0,0}\makebox(0,0)[lb]{\smash{$z_2$}}}%
    \put(0.89307034,0.09694824){\color[rgb]{0,0,0}\makebox(0,0)[lb]{\smash{$w_2$}}}%
    \put(0.97987703,0.15488967){\color[rgb]{0,0,0}\makebox(0,0)[lb]{\smash{}}}%
  \end{picture}%
\endgroup%

%% file: hdannulus.pdf_tex
\begingroup%
  \makeatletter%
  \providecommand\color[2][]{%
    \errmessage{(Inkscape) Color is used for the text in Inkscape, but the package 'color.sty' is not loaded}%
    \renewcommand\color[2][]{}%
  }%
  \providecommand\transparent[1]{%
    \errmessage{(Inkscape) Transparency is used (non-zero) for the text in Inkscape, but the package 'transparent.sty' is not loaded}%
    \renewcommand\transparent[1]{}%
  }%
  \providecommand\rotatebox[2]{#2}%
  \ifx\svgwidth\undefined%
    \setlength{\unitlength}{362.244008bp}%
    \ifx\svgscale\undefined%
      \relax%
    \else%
      \setlength{\unitlength}{\unitlength * \real{\svgscale}}%
    \fi%
  \else%
    \setlength{\unitlength}{\svgwidth}%
  \fi%
  \global\let\svgwidth\undefined%
  \global\let\svgscale\undefined%
  \makeatother%
  \begin{picture}(1,0.42175037)%
    \put(0,0){\includegraphics[width=\unitlength,page=1]{hdannulus.pdf}}%
    \put(0.81980061,0.20518355){\color[rgb]{0,0,0}\makebox(0,0)[lb]{\smash{\textsl{$a$}}}}%
    \put(0.03569823,0.19403394){\color[rgb]{0,0,0}\makebox(0,0)[lb]{\smash{\textsl{$w$}}}}%
    \put(0.34515838,0.19403394){\color[rgb]{0,0,0}\makebox(0,0)[lb]{\smash{\textsl{$z$}}}}%
  \end{picture}%
\endgroup%

%% file: HDsaddle.pdf_tex
\begingroup%
  \makeatletter%
  \providecommand\color[2][]{%
    \errmessage{(Inkscape) Color is used for the text in Inkscape, but the package 'color.sty' is not loaded}%
    \renewcommand\color[2][]{}%
  }%
  \providecommand\transparent[1]{%
    \errmessage{(Inkscape) Transparency is used (non-zero) for the text in Inkscape, but the package 'transparent.sty' is not loaded}%
    \renewcommand\transparent[1]{}%
  }%
  \providecommand\rotatebox[2]{#2}%
  \ifx\svgwidth\undefined%
    \setlength{\unitlength}{157.65143193bp}%
    \ifx\svgscale\undefined%
      \relax%
    \else%
      \setlength{\unitlength}{\unitlength * \real{\svgscale}}%
    \fi%
  \else%
    \setlength{\unitlength}{\svgwidth}%
  \fi%
  \global\let\svgwidth\undefined%
  \global\let\svgscale\undefined%
  \makeatother%
  \begin{picture}(1,1.06653534)%
    \put(0.46821119,0.22225485){\color[rgb]{1,0,0}\makebox(0,0)[lb]{\smash{\footnotesize $\alpha_1$}}}%
    \put(0,0){\includegraphics[width=\unitlength,page=1]{HDsaddle.pdf}}%
    \put(0.43567125,0.70602707){\color[rgb]{0,1,0}\makebox(0,0)[lb]{\smash{\footnotesize $\delta_1$}}}%
    \put(0.48912783,0.89442497){\color[rgb]{0,0,1}\makebox(0,0)[lb]{\smash{\footnotesize $\beta_1$}}}%
    \put(0.33116656,0.89442497){\color[rgb]{1,0,0}\makebox(0,0)[lb]{\smash{\footnotesize $\alpha_2$}}}%
    \put(0,0){\includegraphics[width=\unitlength,page=2]{HDsaddle.pdf}}%
  \end{picture}%
\endgroup%

%% file: handles2.tex
\begin{tikzpicture}[crossline/.style={preaction={draw=white, -, line width=5pt}}]
\usetikzlibrary{patterns}

\def\u{0.6cm};
\def\h{4*\u};
\def\v{4*\u};
\def\b{3*\u};
\def\c{4*\u};
\def\hshift{12*\u};
\def\d{0.23570226039*\u};
\def\smallradius{\u/15};

%
%

\begin{scope}[shift={(\hshift,0)}]

\begin{scope}[shift={(0,0-\v)}]
\def\halfradius{\h};
\draw (0,0) ellipse (\halfradius*1 and \halfradius/3);
\draw[very thin, white, fill=white] (0-1.2*\halfradius,0) rectangle (1.2*\halfradius,\halfradius/2);
\draw[dashed] (0,0) ellipse (\halfradius*1 and \halfradius/3);
\end{scope}

\begin{scope}[shift={(-\h/2,0-\v)}]
\def\halfradius{\u};
\draw (0,0) ellipse (\halfradius*1 and \halfradius/3);
\draw[very thin, white, fill=white] (0-1.2*\halfradius,0) rectangle (1.2*\halfradius,\halfradius/2);
\draw[dashed] (0,0) ellipse (\halfradius*1 and \halfradius/3);
\draw (-\halfradius,0) node[anchor=east] {$-$};
\draw (\halfradius,0) node[anchor=west] {$+$};
\end{scope}

%

\begin{scope}[shift={(-\h/2,0)}]
\draw (-\u, \v) -- (-\u, -\v);
\draw (\u, \v) -- (\u, -\v);
\draw[color=green] (-3*\d, \v-\d) -- (-3*\d, -\v-\d);
\draw[color=green] (3*\d, \v-\d) -- (3*\d, -\v-\d);
\end{scope}

\draw (\h/2-\u,\v) .. controls (\h/2-\u,\v-\c) and (\h/2+\u,\v-\c) .. (\h/2+\u,\v);
\draw[color=green] (\h/2-3*\d,\v-\d) .. controls (\h/2-3*\d,\v-\d-0.8*\c) and (\h/2+3*\d,\v-\d-0.8*\c) .. (\h/2+3*\d,\v-\d);


\begin{scope}[shift={(0,\v)}]
\def\halfradius{\h};
\draw[crossline] (0,0) ellipse (\halfradius*1 and \halfradius/3);
\end{scope}

\begin{scope}[shift={(\h/2,\v)}]
\def\halfradius{\u};
\draw (0,0) ellipse (\halfradius*1 and \halfradius/3);
\draw (-\halfradius,0) node[anchor=east] {$-$};
\draw (\halfradius,0) node[anchor=west] {$+$};
\end{scope}
\begin{scope}[shift={(-\h/2,\v)}]
\def\halfradius{\u};
\draw (0,0) ellipse (\halfradius*1 and \halfradius/3);
\draw (-\halfradius,0) node[anchor=east] {$-$};
\draw (\halfradius,0) node[anchor=west] {$+$};
\end{scope}

\begin{scope}[shift={(0,\v)}]
\draw[color=red,fill=red] (-\h/2, \u/3) circle (\smallradius);
\draw[color=red,fill=red] (-\h/2, -\u/3) circle (\smallradius);
\draw[color=red,fill=red] (\h/2, \u/3) circle (\smallradius);
\draw[color=red,fill=red] (\h/2, -\u/3) circle (\smallradius);
\end{scope}
\begin{scope}[shift={(-\h/2,-\v)}]
\draw[color=red,fill=red] (0, \u/3) circle (\smallradius);
\draw[color=red,fill=red] (0, -\u/3) circle (\smallradius);
\end{scope}

\draw (-\h, \v) -- (-\h, -\v);
\draw (\h, \v) -- (\h, -\v);

\begin{scope}[shift={(\h/2,0)}]
\draw[color=blue,fill=blue] (0, -\v) circle (\smallradius);
\draw[color=blue,fill=blue] (0, \v-0.753*\c) circle (\smallradius);
\draw[thick,blue] (-0, -\v) -- (0,\v-0.753*\c);
\end{scope}

\end{scope}

\end{tikzpicture}

%% file: saddle.pdf_tex
\begingroup%
  \makeatletter%
  \providecommand\color[2][]{%
    \errmessage{(Inkscape) Color is used for the text in Inkscape, but the package 'color.sty' is not loaded}%
    \renewcommand\color[2][]{}%
  }%
  \providecommand\transparent[1]{%
    \errmessage{(Inkscape) Transparency is used (non-zero) for the text in Inkscape, but the package 'transparent.sty' is not loaded}%
    \renewcommand\transparent[1]{}%
  }%
  \providecommand\rotatebox[2]{#2}%
  \ifx\svgwidth\undefined%
    \setlength{\unitlength}{290.725bp}%
    \ifx\svgscale\undefined%
      \relax%
    \else%
      \setlength{\unitlength}{\unitlength * \real{\svgscale}}%
    \fi%
  \else%
    \setlength{\unitlength}{\svgwidth}%
  \fi%
  \global\let\svgwidth\undefined%
  \global\let\svgscale\undefined%
  \makeatother%
  \begin{picture}(1,0.2656348)%
    \put(0,0){\includegraphics[width=\unitlength]{saddle.pdf}}%
    \put(0.47947616,0.09918218){\color[rgb]{0,0,0}\makebox(0,0)[lb]{\smash{=}}}%
  \end{picture}%
\endgroup%

%% file: ComputationsSFHcobmaps.tex
\label{sec:TQFT-elc}

This section is devoted to computing the maps induced
on~$\HFLh$ by elementary decorated link cobordisms between unlinks in~$S^3$.
We denote by $(U_n, P_n)$, or, with a slight abuse of notation, just by~$U_n$,
the standard $n$-component unlink with two decorations on each component.
We often consider $U_n$ as embedded in the plane $\{z=0\} \subseteq \R^3 \subseteq S^3$,
with the centres of the components on the $x$-axis.
If $(c_k,0,0)$ is the centre of the $k$-th component $U_n^k$,
then we suppose that for every $(x,y,0) \in R_-(P_n) \cap U_n^k$, we have $x \le c_k$,
and for every $(x,y,0) \in R_+(P_n) \cap U_n^k$, we have $x \ge c_k$
(roughly speaking, the $-$ arc is on the left and the $+$ arc is on the right).
We will often omit the labels $-$ and $+$ in our figures.

The figures illustrating decorated link cobordisms will follow the conventions explained
in Figure~\ref{fig:handles}: The cylinder enclosing a cobordism
represents the section $\R^2 \times \{0\} \times I \subset S^3 \times I$,
with coordinates as shown in Figure~\ref{fig:handles}.
In this section, we focus on the following four cobordisms.

\begin{definition}
\label{def:defcobordisms}
Let $\CobV$ denote the split saddle cobordism from $(U_1, P_1)$ to $(U_2, P_2)$,
with decorations as in Figure~\ref{fig:defcobordisms}.
Note that this is \emph{not} an elementary decorated cobordism in the sense of Definition~\ref{def:mergesaddle},
because the height function restricted to the decorations has a minimum rather than a maximum.

Let $\CobLambda$ denote the merge saddle cobordism from $(U_2, P_2)$ to $(U_1, P_1)$,
with decorations as in Figure~\ref{fig:defcobordisms}.

Let $\CobGamma$ denote the birth cobordism from $(U_1, P_1)$ to $(U_2, P_2)$,
with decorations as in Figure~\ref{fig:defcobordisms}.

Let $\CobL$ denote the death cobordism from $(U_2, P_2)$ to $(U_1, P_1)$,
with decorations as in Figure~\ref{fig:defcobordisms}.
\end{definition}

We will now determine the maps associated to the cobordisms defined above.
Recall that $V$ is the vector space $\mb F_2\langle{B,T}\rangle$ generated by
$B$ in bigrading $(0,-1/2)$ and $T$ in bigrading $(0,1/2)$.

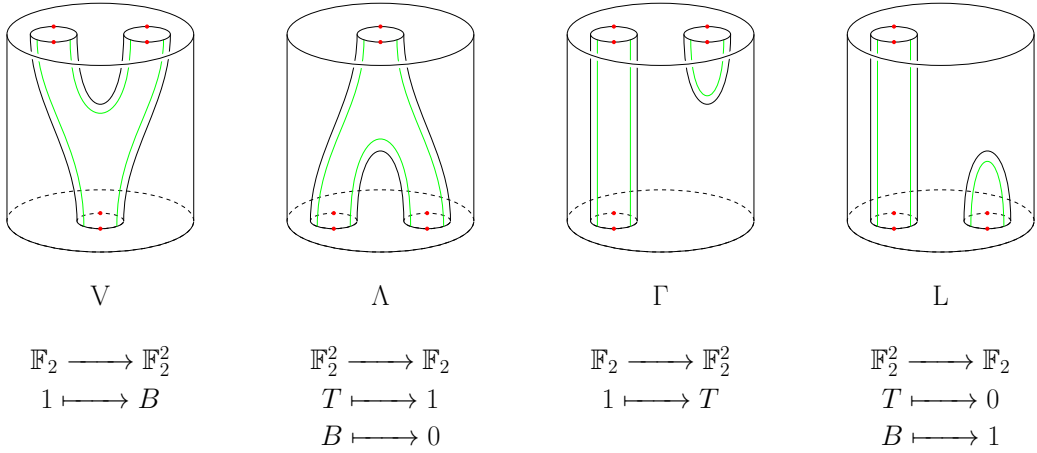
\begin{figure}
\resizebox{\textwidth}{!}{\input{defcobordisms.tex}}
\caption{The decorated link cobordisms $\CobV$, $\CobLambda$, $\CobGamma$, and $\CobL$, and the maps induced on link Floer homology.
The cylinders represent the section $\{z=0\}$ in $\R^3 \times I$.
}
\label{fig:defcobordisms}
\end{figure}

\begin{theorem}
\label{thm:unlinkcob}
Under the canonical identifications
\[
\HFLh(U_1,P_1) \cong \mb F_2 \qquad \text{and} \qquad \HFLh(U_2,P_2) \cong V,
\]
the maps associated to the decorated link cobordisms in
Definition~\ref{def:defcobordisms} are as in Figure~\ref{fig:defcobordisms}.
\end{theorem}

\begin{proof}
Consider the cobordism $\CobV$. By Theorem~\ref{thm:grading},
we know that $F_{\CobV}$ is homogeneous with respect to the Maslov grading,
and that the Maslov grading shift is $-1/2$. Therefore $F_{\CobV}(1)$ is either to $0$ or to $B$.
Since $\CobL \circ \CobV = \id_{(U_1, P_1)}$, by functoriality $F_{\CobV} \neq 0$,
hence $F_{\CobV}(1) = B$.

Similarly, $F_{\CobLambda}$ shifts the Maslov grading by $-1/2$ and is non-vanishing
since $\CobLambda \circ \CobGamma = \id_{(U_1,P_1)}$, hence $F_{\CobLambda}(T) = 1$
and $F_{\CobLambda}(B) = 0$.

By the above, $F_{\CobGamma}$ and $F_{\CobL}$ are also non-vanishing, and they shift
the Maslov grading by $1/2$ according to Theorem~\ref{thm:grading}.
If follows that $F_{\CobGamma}(1) = T$, while $F_{\CobL}(T) = 0$ and $F_{\CobL}(B) = 1$.
\end{proof}

\begin{remarkpro}
\label{rem:deco}
If the decorated cobordism $(S,\sigma)$ has the same underlying surface~$S$ as one of
$\CobV$, $\CobLambda$, $\CobGamma$, or $\CobL$,
but possibly different decorations, then $F_{(S,\sigma)}$ is also given by
the respective formula in Figure~\ref{fig:defcobordisms},
provided that the following conditions are satisfied:
\begin{itemize}
\item there are two decorations on each boundary component,
\item there is a partial (left or right) inverse in the category of decorated link cobordisms, and
\item $\chi(R_+(\sigma)) = \chi(R_-(\sigma))$.
\end{itemize}
The proof goes in the same way as for Theorem~\ref{thm:unlinkcob}.
\end{remarkpro} 

%% file: defcobordisms.tex
\begin{tikzpicture}[crossline/.style={preaction={draw=white, -, line width=5pt}}]

\def\u{0.6cm};
\def\h{4*\u};
\def\v{4*\u};
\def\b{3*\u};
\def\c{4*\u};
\def\hshift{12*\u};
\def\d{0.23570226039*\u};
\def\smallradius{\u/15};

\begin{scope}[shift={(-1.5*\hshift,0)}]

\begin{scope}[shift={(0,0-\v)}]
\def\halfradius{\h};
\draw (0,0) ellipse (\halfradius*1 and \halfradius/3);
\draw[white, fill=white] (0-1.2*\halfradius,0) rectangle (1.2*\halfradius,\halfradius/2);
\draw[dashed] (0,0) ellipse (\halfradius*1 and \halfradius/3);
\end{scope}

\begin{scope}[shift={(0,0-\v)}]
\def\halfradius{\u};
\draw (0,0) ellipse (\halfradius*1 and \halfradius/3);
\draw[white, fill=white] (0-1.2*\halfradius,0) rectangle (1.2*\halfradius,\halfradius/2);
\draw[dashed] (0,0) ellipse (\halfradius*1 and \halfradius/3);
\end{scope}

\draw (-\h/2-\u,\v) .. controls (-\h/2-\u,\v-\b) and (-\u,-\v+\b) .. (-\u,-\v);
\draw (\h/2+\u,\v) .. controls (\h/2+\u,\v-\b) and (\u,-\v+\b) .. (\u,-\v);
\draw (-\h/2+\u,\v) .. controls (-\h/2+\u,\v-\c) and (\h/2-\u,\v-\c) .. (\h/2-\u,\v);

\draw[color=green] (-\h/2-3*\d,\v-\d) .. controls (-\h/2-3*\d,\v-\d-0.9*\b) and (-3*\d,-\v-\d+1.2*\b) .. (-3*\d,-\v-\d);
\draw[color=green] (\h/2+3*\d,\v-\d) .. controls (\h/2+3*\d,\v-\d-0.9*\b) and (3*\d,-\v-\d+1.2*\b) .. (3*\d,-\v-\d);
\draw[color=green] (-\h/2+3*\d,\v-\d) .. controls (-\h/2+3*\d,\v-\d-1.05*\c) and (\h/2-3*\d,\v-\d-1.05*\c) .. (\h/2-3*\d,\v-\d);

\begin{scope}[shift={(0,\v)}]
\def\halfradius{\h};
\draw[crossline] (0,0) ellipse (\halfradius*1 and \halfradius/3);
\end{scope}

\begin{scope}[shift={(\h/2,\v)}]
\def\halfradius{\u};
\draw (0,0) ellipse (\halfradius*1 and \halfradius/3);
\end{scope}
\begin{scope}[shift={(-\h/2,\v)}]
\def\halfradius{\u};
\draw (0,0) ellipse (\halfradius*1 and \halfradius/3);
\end{scope}

\begin{scope}[shift={(0,\v)}]
\draw[color=red,fill=red] (-\h/2, \u/3) circle (\smallradius);
\draw[color=red,fill=red] (-\h/2, -\u/3) circle (\smallradius);
\draw[color=red,fill=red] (\h/2, \u/3) circle (\smallradius);
\draw[color=red,fill=red] (\h/2, -\u/3) circle (\smallradius);
\end{scope}
\begin{scope}[shift={(0,-\v)}]
\draw[color=red,fill=red] (0, \u/3) circle (\smallradius);
\draw[color=red,fill=red] (0, -\u/3) circle (\smallradius);
\end{scope}

\draw (-\h, \v) -- (-\h, -\v);
\draw (\h, \v) -- (\h, -\v);

\draw (0,-1.8*\v) node {\huge $\CobV$};

\draw (0, -2.5*\v) node {\huge $\mb{F}_2 \xrightarrow{\qquad} \mb F_2^2$};
\draw (0, -2.9*\v) node {\huge $1 \xmapsto{\qquad} B$};

\end{scope}

\begin{scope}[shift={(-0.5*\hshift,0)}]

\begin{scope}[shift={(0,0-\v)}]
\def\halfradius{\h};
\draw (0,0) ellipse (\halfradius*1 and \halfradius/3);
\draw[white, fill=white] (0-1.2*\halfradius,0) rectangle (1.2*\halfradius,\halfradius/2);
\draw[dashed] (0,0) ellipse (\halfradius*1 and \halfradius/3);
\end{scope}

\begin{scope}[shift={(\h/2,0-\v)}]
\def\halfradius{\u};
\draw (0,0) ellipse (\halfradius*1 and \halfradius/3);
\draw[white, fill=white] (0-1.2*\halfradius,0) rectangle (1.2*\halfradius,\halfradius/2);
\draw[dashed] (0,0) ellipse (\halfradius*1 and \halfradius/3);
\end{scope}
\begin{scope}[shift={(-\h/2,0-\v)}]
\def\halfradius{\u};
\draw (0,0) ellipse (\halfradius*1 and \halfradius/3);
\draw[white, fill=white] (0-1.2*\halfradius,0) rectangle (1.2*\halfradius,\halfradius/2);
\draw[dashed] (0,0) ellipse (\halfradius*1 and \halfradius/3);
\end{scope}

\begin{scope}[yscale=-1]
\draw (-\h/2-\u,\v) .. controls (-\h/2-\u,\v-\b) and (-\u,-\v+\b) .. (-\u,-\v);
\draw (\h/2+\u,\v) .. controls (\h/2+\u,\v-\b) and (\u,-\v+\b) .. (\u,-\v);
\draw (-\h/2+\u,\v) .. controls (-\h/2+\u,\v-\c) and (\h/2-\u,\v-\c) .. (\h/2-\u,\v);

\begin{scope}[yshift=2*\d]
\draw[color=green] (-\h/2-3*\d,\v-\d) .. controls (-\h/2-3*\d,\v-\d-1.1*\b) and (-3*\d,-\v-\d+1*\b) .. (-3*\d,-\v-\d);
\draw[color=green] (\h/2+3*\d,\v-\d) .. controls (\h/2+3*\d,\v-\d-1.1*\b) and (3*\d,-\v-\d+1*\b) .. (3*\d,-\v-\d);
\draw[color=green] (-\h/2+3*\d,\v-\d) .. controls (-\h/2+3*\d,\v-\d-1.25*\c) and (\h/2-3*\d,\v-\d-1.25*\c) .. (\h/2-3*\d,\v-\d);
\end{scope}
\end{scope}

\begin{scope}[shift={(0,\v)}]
\def\halfradius{\h};
\draw[crossline] (0,0) ellipse (\halfradius*1 and \halfradius/3);
\end{scope}

\begin{scope}[shift={(0,\v)}]
\def\halfradius{\u};
\draw (0,0) ellipse (\halfradius*1 and \halfradius/3);
\end{scope}

\begin{scope}[shift={(0,-\v)}]
\draw[color=red,fill=red] (-\h/2, \u/3) circle (\smallradius);
\draw[color=red,fill=red] (-\h/2, -\u/3) circle (\smallradius);
\draw[color=red,fill=red] (\h/2, \u/3) circle (\smallradius);
\draw[color=red,fill=red] (\h/2, -\u/3) circle (\smallradius);
\end{scope}
\begin{scope}[shift={(0,\v)}]
\draw[color=red,fill=red] (0, \u/3) circle (\smallradius);
\draw[color=red,fill=red] (0, -\u/3) circle (\smallradius);
\end{scope}

\draw (-\h, \v) -- (-\h, -\v);
\draw (\h, \v) -- (\h, -\v);

\draw (0,-1.8*\v) node {\huge $\Lambda$};

\draw (0, -2.5*\v) node {\huge $\mb{F}_2^2 \xrightarrow{\qquad} \mb F_2$};
\draw (0, -2.9*\v) node {\huge $T \xmapsto{\qquad} 1$};
\draw (0, -3.3*\v) node {\huge $B \xmapsto{\qquad} 0$};

\end{scope}

\begin{scope}[shift={(0.5*\hshift,0)}]

\begin{scope}[shift={(0,0-\v)}]
\def\halfradius{\h};
\draw (0,0) ellipse (\halfradius*1 and \halfradius/3);
\draw[white, fill=white] (0-1.2*\halfradius,0) rectangle (1.2*\halfradius,\halfradius/2);
\draw[dashed] (0,0) ellipse (\halfradius*1 and \halfradius/3);
\end{scope}

\begin{scope}[shift={(-\h/2,0-\v)}]
\def\halfradius{\u};
\draw (0,0) ellipse (\halfradius*1 and \halfradius/3);
\draw[white, fill=white] (0-1.2*\halfradius,0) rectangle (1.2*\halfradius,\halfradius/2);
\draw[dashed] (0,0) ellipse (\halfradius*1 and \halfradius/3);
\end{scope}

\begin{scope}[shift={(-\h/2,0)}]
\draw (-\u, \v) -- (-\u, -\v);
\draw (\u, \v) -- (\u, -\v);
\draw[color=green] (-3*\d, \v-\d) -- (-3*\d, -\v-\d);
\draw[color=green] (3*\d, \v-\d) -- (3*\d, -\v-\d);
\end{scope}

\draw (\h/2-\u,\v) .. controls (\h/2-\u,\v-\c) and (\h/2+\u,\v-\c) .. (\h/2+\u,\v);
\draw[color=green] (\h/2-3*\d,\v-\d) .. controls (\h/2-3*\d,\v-\d-0.8*\c) and (\h/2+3*\d,\v-\d-0.8*\c) .. (\h/2+3*\d,\v-\d);

\begin{scope}[shift={(0,\v)}]
\def\halfradius{\h};
\draw[crossline] (0,0) ellipse (\halfradius*1 and \halfradius/3);
\end{scope}

\begin{scope}[shift={(\h/2,\v)}]
\def\halfradius{\u};
\draw (0,0) ellipse (\halfradius*1 and \halfradius/3);
\end{scope}
\begin{scope}[shift={(-\h/2,\v)}]
\def\halfradius{\u};
\draw (0,0) ellipse (\halfradius*1 and \halfradius/3);
\end{scope}

\begin{scope}[shift={(0,\v)}]
\draw[color=red,fill=red] (-\h/2, \u/3) circle (\smallradius);
\draw[color=red,fill=red] (-\h/2, -\u/3) circle (\smallradius);
\draw[color=red,fill=red] (\h/2, \u/3) circle (\smallradius);
\draw[color=red,fill=red] (\h/2, -\u/3) circle (\smallradius);
\end{scope}
\begin{scope}[shift={(-\h/2,-\v)}]
\draw[color=red,fill=red] (0, \u/3) circle (\smallradius);
\draw[color=red,fill=red] (0, -\u/3) circle (\smallradius);
\end{scope}

\draw (-\h, \v) -- (-\h, -\v);
\draw (\h, \v) -- (\h, -\v);

\draw (0,-1.8*\v) node {\huge $\Gamma$};

\draw (0, -2.5*\v) node {\huge $\mb{F}_2 \xrightarrow{\qquad} \mb F_2^2$};
\draw (0, -2.9*\v) node {\huge $1 \xmapsto{\qquad} T$};

\end{scope}

\begin{scope}[shift={(1.5*\hshift,0)}]

\begin{scope}[shift={(0,0-\v)}]
\def\halfradius{\h};
\draw (0,0) ellipse (\halfradius*1 and \halfradius/3);
\draw[white, fill=white] (0-1.2*\halfradius,0) rectangle (1.2*\halfradius,\halfradius/2);
\draw[dashed] (0,0) ellipse (\halfradius*1 and \halfradius/3);
\end{scope}

\begin{scope}[shift={(-\h/2,0-\v)}]
\def\halfradius{\u};
\draw (0,0) ellipse (\halfradius*1 and \halfradius/3);
\draw[white, fill=white] (0-1.2*\halfradius,0) rectangle (1.2*\halfradius,\halfradius/2);
\draw[dashed] (0,0) ellipse (\halfradius*1 and \halfradius/3);
\end{scope}
\begin{scope}[shift={(\h/2,0-\v)}]
\def\halfradius{\u};
\draw (0,0) ellipse (\halfradius*1 and \halfradius/3);
\draw[white, fill=white] (0-1.2*\halfradius,0) rectangle (1.2*\halfradius,\halfradius/2);
\draw[dashed] (0,0) ellipse (\halfradius*1 and \halfradius/3);
\end{scope}

\begin{scope}[shift={(-\h/2,0)}]
\draw (-\u, \v) -- (-\u, -\v);
\draw (\u, \v) -- (\u, -\v);
\draw[color=green] (-3*\d, \v-\d) -- (-3*\d, -\v-\d);
\draw[color=green] (3*\d, \v-\d) -- (3*\d, -\v-\d);
\end{scope}

\draw (\h/2-\u,-\v) .. controls (\h/2-\u,-\v+\c) and (\h/2+\u,-\v+\c) .. (\h/2+\u,-\v);
\draw[color=green] (\h/2-3*\d,-\v-\d) .. controls (\h/2-3*\d,-\v-\d+0.93*\c) and (\h/2+3*\d,-\v-\d+0.93*\c) .. (\h/2+3*\d,-\v-\d);

\begin{scope}[shift={(0,\v)}]
\def\halfradius{\h};
\draw[crossline] (0,0) ellipse (\halfradius*1 and \halfradius/3);
\end{scope}

\begin{scope}[shift={(-\h/2,\v)}]
\def\halfradius{\u};
\draw (0,0) ellipse (\halfradius*1 and \halfradius/3);
\end{scope}

\begin{scope}[shift={(0,\v)}]
\draw[color=red,fill=red] (-\h/2, \u/3) circle (\smallradius);
\draw[color=red,fill=red] (-\h/2, -\u/3) circle (\smallradius);
\end{scope}
\begin{scope}[shift={(-\h/2,-\v)}]
\draw[color=red,fill=red] (0, \u/3) circle (\smallradius);
\draw[color=red,fill=red] (0, -\u/3) circle (\smallradius);
\end{scope}
\begin{scope}[shift={(\h/2,-\v)}]
\draw[color=red,fill=red] (0, \u/3) circle (\smallradius);
\draw[color=red,fill=red] (0, -\u/3) circle (\smallradius);
\end{scope}

\draw (-\h, \v) -- (-\h, -\v);
\draw (\h, \v) -- (\h, -\v);

\draw (0,-1.8*\v) node {\huge $\CobL$};

\draw (0, -2.5*\v) node {\huge $\mb{F}_2^2 \xrightarrow{\qquad} \mb F_2$};
\draw (0, -2.9*\v) node {\huge $T \xmapsto{\qquad} 0$};
\draw (0, -3.3*\v) node {\huge $B \xmapsto{\qquad} 1$};

\end{scope}

\end{tikzpicture}

%% file: disjointunionformula.tex
In this section, we focus on a formula for the disjoint union of link
cobordisms in~$S^3 \times I$, with the aim of putting the maps associated to the cobordisms
in Figure~\ref{fig:defcobordisms} into a TQFT.
Although we actually need a minimal version of it, we give a proof
in the general setting that might be of independent interest;
cf.~Theorem~\ref{thm:cs}.

It is well known (see for instance \cite[Proposition 9.15]{SFH})
that the link Floer homology of the ``split'' disjoint union of two
decorated links $(K,P_K)$ and $(L,P_L)$ satisfies
\begin{equation}
\label{eq:disjointunion}
\HFLh(K \sqcup L, P_K \sqcup P_L) \cong \HFLh(K, P_K) \otimes \HFLh(L, P_L) \otimes \mb F_2^2.
\end{equation}
We would like to understand if the cobordism map associated to a
disjoint union of link cobordisms splits in a similar way.

\begin{definition}
Suppose that $\mc X_K=(F, \sigma_F)$ is a decorated link
cobordism from $(K_0, P_{K_0})$ to $(K_1, P_{K_1})$,
and that $\mc X_L=(G, \sigma_G)$ is a decorated link cobordism
from $(L_0, P_{L_0})$ to $(L_1, P_{L_1})$. We define
the \emph{disjoint union cobordism} to be the cobordism in
$S^3 \times I$ obtained by taking the split disjoint union of the two
surfaces $F$ and $G$ with decoration $\sigma_F \sqcup \sigma_G$.
We denote it by $\mc X_K \cs \mc X_L$.
Such a cobordism connects the decorated link $(K_0 \sqcup L_0, P_{K_0} \sqcup P_{L_0})$ (the ``split''
disjoint union of $(K_0, P_{K_0})$ and $(L_0, P_{L_0})$)
to the decorated link $(K_1 \sqcup L_1, P_{K_1} \sqcup P_{L_1})$.
\end{definition}

One would hope that the map induced by the cobordism
$\mc X_K \cs \mc X_L$ on link Floer homology becomes
$F_{\mc X_K} \otimes F_{\mc X_L} \otimes \id_{\mb F_2^2}$,
after the applying the formula in~\eqref{eq:disjointunion}
to $\hat\HFL(K_0 \sqcup L_0, P_{K_0} \sqcup P_{L_0})$ and
$\hat\HFL(K_1 \sqcup L_1, P_{K_1} \sqcup P_{L_1})$.
We will show that this is true under some mild assumptions.

We will actually prove it in a generalised setting, by considering
sutured manifolds. Recall that the complement of a decorated link is naturally
a sutured manifold; see Definition~\ref{def:dl}. By \cite[Proposition~9.15]{SFH},
equation~\eqref{eq:disjointunion} generalises as follows:
If $(M, \gamma)$ and $(N, \ni)$ are balanced sutured manifolds, then
\begin{equation}
\label{eq:connectedsum}
\SFH\left((M,\gamma)\#(N,\ni)\right) \cong \SFH(M,\gamma) \otimes \SFH(N,\ni) \otimes \mb F_2^2.
\end{equation}
Actually, we need to make the splitting in Equation~\eqref{eq:connectedsum}
functorial, and for this purpose, we use \emph{framed pairs of points},
\emph{bouquets}, and \emph{sutured diagrams subordinate to bouquets}, as
defined in \cite[Section~7]{cobordisms}.

Given a framed pair of points $\mb P$ in the interior of $(M,\gamma) \sqcup (N, \ni)$
such that $p_- \in M$ and $p_+ \in N$, let $M^\circ = \Bl_{p_-}(M)$
and $N^\circ = \Bl_{p_+}(N)$. We write $M \#_\P N$ for the connected
sum obtained by gluing the sphere $UNp_- \subset M^\circ$ to $UNp_+ \subset N^\circ$
using the framings of $T_{p_-}M$ and $T_{p_+}N$. (In other words, we perform
surgery on $(M,\gamma) \sqcup (N, \ni)$ along the framed 0-sphere $\mb P$.)
We can view $M^\circ$ and $N^\circ$ as embedded in $M \#_\P N$,
and we denote by $S^2$ the connected sum sphere $M^\circ \cap N^\circ$.
The relative Mayer-Vietoris sequences in cohomology for
$(M , \de M ) = (M^\circ, \de M) \cup B^3$,
$(N , \de N ) = (N^\circ, \de N) \cup B^3$, and
$(M \#_\P N, \de(M \#_\P N)) = (M^\circ, \de M) \cup (N^\circ, \de N)$
imply that the map
\begin{IEEEeqnarray*}{rccc}
\label{eq:phi1}
\phi: & \,\SpinC(M, \gamma) \times \SpinC(N, \ni)\,
& \xrightarrow{\qquad} & \,\SpinC(M \#_\P N, \gamma \cup \ni) \IEEEyesnumber\\
& \,(\mf s_M,\mf s_N)\, & \xmapsto{\qquad} & \, (\s_M)|_{M^\circ} \cup (\s_N)|_{N^\circ}
\end{IEEEeqnarray*}
is well-defined and injective. We will write $\s_M \#_\P \s_N = \phi(\s_M, \s_N)$.
Furthermore, $\mf s \in \SpinC(M \#_\P N, \gamma \cup \ni)$
is in the image of $\phi$ if and only if $\langle\, c_1(\s), [S^2] \,\rangle = 0$.
In this case, we call $\mf s|_M$ and $\mf s|_N$ the unique $\SpinC$ structures
such that $\phi(\mf s|_M, \mf s|_N) = \mf s$.

The following is a $\SpinC$ refinement of \cite[Proposition 9.15]{SFH}.

\begin{lemma}
\label{lem:canonicalcs}
Let $\mb P$ be a framed pair of points in the sutured manifold
$(M, \gamma) \sqcup (N, \ni)$, such that $p_- \in M$
and $p_+ \in N$. Let $\mc H_M = (\Sigma_M, \ba_M, \bb_M)$ and
$\mc H_N = (\Sigma_N, \ba_N, \bb_N)$ be admissible balanced
diagrams of $(M,\g)$ and $(N,\nu)$, respectively,
such that $\mc H_M \sqcup \mc H_N$ is subordinate to a bouquet $B(\P)$ for~$\P$.
Let $\mc H_M \#_\P \mc H_N$ be the diagram of $M \#_\P N$
obtained by taking the connected sum of~$\S_M$ and~$\S_N$
along~$\mb P$, and adding a meridional $\a$-curve and a $\b$-curve
on the connected sum tube that intersect in two points transversely.
If $\s \in \SpinC(M \#_{\mb P} N, \gamma \sqcup \ni)$ is
such that $\langle\, c_1(\s), [S^2] \,\rangle \neq 0$, then
\[
\SFH(\mc H_M \#_\P \mc H_N, \s) = 0.
\]
If $\langle\, c_1(\s), [S^2] \,\rangle = 0$, then there is an isomorphism
\[
\varphi_{\mc H_M, \mc H_N, \P, \s} \colon
\SFH(\mc H_M \#_\P \mc H_N, \s) \to \SFH(\mc H_M, \s|_M) \otimes \SFH(\mc H_N, \s|_N) \otimes V,
\]
where $V = \mb F_2\langle T, B \rangle$ denotes a 2-dimensional
vector space generated by two homogeneous elements $T$ (top) and
$B$ (bottom).
Furthermore, there is an isomorphism
\[
\varphi_{\mc H_M, \mc H_N, \P} \colon
\SFH(\mc H_M \#_\P \mc H_N) \to \SFH(\mc H_M) \otimes \SFH(\mc H_N) \otimes V.
\]
\end{lemma}

\begin{proof}
Notice that for each generator $\x \in \CF(\mc H_M \#_\P \mc H_N)$, we have $\langle\, c_1(\s(\x)), [S^2] \,\rangle = 0$.
Therefore, $\SFH(\mc H_M \#_\P \mc H_N, \s) = 0$ unless $\langle\, c_1(\s), [S^2] \,\rangle = 0$.
There is a self-evident bijection
\[
\CF(\mc H_M \#_\P \mc H_N, \s) \to \CF(\mc H_M, \s|_M) \otimes \CF(\mc H_N, \s|_N) \otimes V.
\]
Since $\mc H_M \sqcup \mc H_N$ is subordinate to the bouquet $B(\mb P)$,
all holomorphic discs on $\Sigma_M \# \Sigma_N$ split as the disjoint union
of a holomorphic disc on $\Sigma_M$, one on $\Sigma_N$, and one on the annulus
$A$ attached to the framed pair of points $\P$. Each of these discs has
multiplicity~$0$ near the boundary.
It follows that the above bijection descends to a map on homology.
\end{proof}

The following result shows that the isomorphisms defined
in Lemma~\ref{lem:canonicalcs} are natural.

\begin{lemma}
\label{lem:naturalitycs}
Let $\mb P$ be a framed pair of points in the sutured manifold
$(M, \gamma) \sqcup (N, \ni)$, such that $p_- \in M$
and $p_+ \in N$. Let $\mc H_M$ and $\mc H'_M$ be admissible balanced diagrams
for $(M, \gamma)$, and let $\mc H_N$ and $\mc H'_N$ be admissible balanced diagrams
for $(N, \ni)$, such that $\mc H_M \sqcup \mc H_N$ and
$\mc H'_M \sqcup \mc H'_N$ are subordinate to the
bouquets $B(\mb P)$ and $B(\mb P)'$ for~$\mb P$, respectively.
If $\s \in \SpinC(M \#_{\mb P} N, \gamma \sqcup \ni)$ is such that $\langle\, c_1(\s), [S^2] \,\rangle = 0$,
then the following diagram commutes:
\begin{center}
\begin{tikzpicture}[description/.style={fill=white,inner sep=2pt}]
\matrix (m) [matrix of math nodes, row sep=4em,
column sep=7.5em, text height=1.7ex, text depth=0.3ex]
{\SFH(\mc H_M \#_\P \mc H_N, \s) & \SFH(\mc H_M, \s|_M) \otimes \SFH(\mc H_N, \s|_N) \otimes V \\
 \SFH(\mc H'_M \#_\P \mc H'_N, \s) & \SFH(\mc H'_M, \s|_M) \otimes \SFH(\mc H'_N, \s|_N) \otimes V. \\};
	\path[->,font=\scriptsize]
		(m-1-1) edge node[above]{$\varphi_{\mc H_M, \mc H_N, \P, \s}$} (m-1-2)
		(m-2-1) edge node[above]{$\varphi_{\mc H'_M, \mc H'_N, \P, \s}$} (m-2-2)
		(m-1-1) edge node[description]{$F_{\mc H_M \#_\P \mc H_N, \mc H'_M \#_\P \mc H'_N}$} (m-2-1)	
		(m-1-2) edge node[description]{$F_{\mc H_M, \mc H'_M} \otimes F_{\mc H_N, \mc H'_N} \otimes \id_V$} (m-2-2);
\end{tikzpicture}
\end{center}
Here $F_{\mc H, \mc H'}$ denotes the naturality map defined in \cite{naturality}.
An analogous statement holds for $\varphi_{\mc H_M, \mc H_N, \P}$ and
$\varphi_{\mc H'_M, \mc H'_N, \P}$.
\end{lemma}
\begin{proof}
The proof is analogous to \cite[Theorem 7.6]{cobordisms}.
\end{proof}

\begin{corollary}
\label{cor:canonicalcs}
Let $\mb P$ be a framed pair of points in the sutured manifold
$(M, \gamma) \sqcup (N, \ni)$, such that $p_- \in M$ and $p_+ \in N$.
If $\s \in \SpinC(M \#_{\mb P} N, \gamma \sqcup \ni)$ is
such that $\s|_{S^2} \neq0$, then
\[
\SFH(M \#_\P N, \s) = 0.
\]
If $\s|_{S^2} = 0$, then there is a well-defined isomorphism
\[
\varphi_{\P, \s} \colon
\SFH(M \#_\P N, \s) \to \SFH(M, \s|_M) \otimes \SFH(N, \s|_N) \otimes V.
\]
Furthermore, there is an isomorphism
\[
\varphi_{\P} \colon
\SFH(M \#_\P N) \to \SFH(M) \otimes \SFH(N) \otimes V.
\]
\end{corollary}

We now generalise the notion of framed pair of points to cobordisms,
in order to define the connected sum of sutured cobordisms in a
natural way.

\begin{definition}
Let $\mc W = (W, Z_W, [\xi_W])$ be a cobordism from $(M_0, \gamma_0)$ to $(M_1, \gamma_1)$,
and let $\mc U=(U, Z_U, [\xi_U])$ be a cobordism from $(N_0, \ni_0)$ to $(N_1, \ni_1)$.
A \emph{framed pair of arcs} $\mb A$ for $\mc W$ and $\mc U$ consists of two properly
embedded arcs $a_- \colon I \to W$ and $a_+ \colon I \to U$,
such that $a_-(i) \in M_i$ and $a_+(i) \in N_i$ for $i \in \{0,1\}$,
together with a negative frame $\langle v_1^-(t), v_2^-(t), v_3^-(t), v_4^-(t) \rangle$
of $T_{a_-(t)}W$ and a positive frame $\langle v_1^+(t), v_2^+(t), v_3^+(t), v_4^+(t) \rangle$
of $T_{a_+(t)}U$ for all $t\in I$, such that
\begin{itemize}
\item{$\mb A(i) := \{a_-(i), a_+(i)\}$ with the frames
$\langle v_1^\pm(i), v_2^\pm(i), v_3^\pm(i) \rangle$ is a framed pair
of points in $(M_i \sqcup N_i, \gamma_i \sqcup \ni_i)$ for $i \in \{0,1\}$, and}
\item{$v_4 = a'_\pm = da_\pm/dt$.}
\end{itemize}
\end{definition}

Given a framed pair of arcs, we can define the connected sum of
the two cobordisms, as explained in Definition~\ref{def:cs}, below.
We denote by $(M, \gamma) \#_\P (N, \ni)$ the connected
sum of the sutured manifolds $(M, \gamma)$ and $(N, \ni)$
along the framed pair of points $\P$, assuming that $p_- \in M$
and $p_+ \in N$.

\begin{definition}
\label{def:cs}
Let $\mc W = (W, Z_W, [\xi_W])$ be a cobordism from $(M_0, \gamma_0)$ to $(M_1, \gamma_1)$,
let $\mc U=(U, Z_U, [\xi_U])$ be a cobordism from
$(N_0, \ni_0)$ to $(N_1, \ni_1)$, and let
$\mb A$ be a framed pair of arcs for $\mc W$ and $\mc U$.

We define the \emph{connected sum} of $\mc W$ and $\mc U$ to be
the cobordism
\[
\mc W \cs_{\mb A} \mc U =(V, Z_V, [\xi_V])
\]
from
$(M_0, \gamma_0) \#_{\mb A(0)} (N_0, \ni_0)$
to $(M_1, \gamma_1) \#_{\mb A(1)} (N_1, \ni_1)$, where
\begin{itemize}
\item{$V$ is the $4$-dimensional manifold obtained by gluing
$\Bl_{a_-}W$ and $\Bl_{a_+}U$ along $UNa_-$ and $UNa_+$,
in such a way that the frame of $a_-(t)$ is identified with the
frame of $a_+(t)$,}
\item{$Z_V = Z_W \sqcup Z_U$, and}
\item{$\xi_V = \xi_W \sqcup \xi_U$.}
\end{itemize}
\end{definition}

\begin{remarkpro}
\label{rem:cs}
Let $\mc X_K = (F, \sigma_F)$ be a decorated link cobordism from $(K_0, P_{K_0})$ to $(K_1, P_{K_1})$,
and let $\mc X_L = (G, \sigma_G)$ be a decorated link cobordism from $(L_0, P_{L_0})$ to $(L_1, P_{L_1})$.
Let $(x_-, y_-, z_-) \in S^3$ be such that the arc
$a_-(t) = (x_-, y_-, z_-, t)$ for $t \in I$
is disjoint from $F$. This holds for generic $(x_-, y_-, z_-)$, after possibly isotoping~$F$.
Analogously, choose $(x_+, y_+, z_+) \in S^3$
and define $a_+$ to be disjoint from~$G$. Consider the pair of arcs
$\mb A = \{a_-, a_+\}$ with frames
$\langle \pm\de_x, \de_y, \de_z, \de_t \rangle$. Then
\[
\W(\mc X_K \cs \mc X_L) = \W(\mc X_K) \cs_{\mb A} \W(\mc X_L).
\]
\end{remarkpro}

As in the case of the connected sum of sutured manifolds,
one can define the connected sum of Spin$^c$ structures for cobordisms.
Let $\mc W =(W, Z_W, [\xi_W])$ and $\mc U=(U, Z_U, [\xi_U])$
be cobordisms of sutured manifolds, and let
$\mc W \cs_{\mb A} \mc U = (W \cs_{\mb A} U, Z, [\xi])$
denote their connected sum along some framed pair of arcs $\mb A$.
Let $W^\circ=\Bl_{a_-}W$ and $U^\circ=\Bl_{a_+}U$. Then $W^\circ$ and
$U^\circ$ are embedded in $W \cs_{\mb A} U$.
Denote by $S^2 \times I$ the thickened sphere $W^\circ \cap U^\circ$.
Then, the relative Mayer-Vietoris sequences for
$(W , Z_W ) = (W^\circ, Z_W) \cup B^3 \times I$,
$(U , Z_U ) = (U^\circ, Z_U) \cup B^3 \times I$, and
$(W \cs_{\mb A} U, Z) = (W^\circ, Z_W) \cup (U^\circ, Z_U)$ imply
that the map
\begin{IEEEeqnarray*}{rccc}
\phi: & \,\SpinC(W, Z_W) \times \SpinC(U, Z_U)\,
& \xrightarrow{\qquad} & \,\SpinC(W \cs_{\mb A} U, Z) \\
& \,(\s_W, \s_U)\, & \xmapsto{\qquad} & \, (\s_W)|_{W^\circ} \cup (\s_U)|_{U^\circ}
\end{IEEEeqnarray*}
is well-defined and injective. We will denote $\phi(\s_W, \s_U)$ by $\s_W \cs_{\mb A} \s_U$.
Furthermore, $\mf s \in \SpinC(W \cs_{\mb A} U, Z)$
is in the image of $\phi$ if and only if $\langle\, c_1(\s), [S^2 \times \{0\}] \,\rangle = 0$.
In this case, we call $\mf s|_W$ and $\mf s|_U$ the unique $\SpinC$ structures
such that $\phi(\mf s|_W, \mf s|_U) = \mf s$.

\begin{definition} \label{def:bouquet}
Let $\mb A = \left(a_\pm, \langle v_1^\pm, \dots, v_4^\pm \rangle \right)$
be a framed pair of arcs for the cobordisms $\mc W = (W, Z_W, [\xi_W])$
from $(M_0,\g_0)$ to $(M_1,\g_1)$ and $\mc U = (U, Z_U, [\xi_U])$ from $(N_0,\nu_0)$ to $(N_1,\nu_1)$.
We consider the coordinates $(s,t)$ on the square $I \times I$.
A \emph{bouquet} $B(\mb A)$ for $\mb A$ consists of a pair of embedded squares
$d_- \colon I \times I \to W$ and $d_+ \colon I \times I \to U$, and a framing
of $d_\pm$ given by a normal vector field $v_\pm$ such that
\begin{enumerate}
\item{$d_\pm(0,t) = a_\pm(t)$ and $d_\pm(1,t) \in Z_W \cup Z_U$ for all $t \in I$,}
\item{if $i \in \{0,1\}$, then the arc $d_\pm(s,i)$ with framing $v_\pm (s,i)$, which we denote by $B(\mb A)_i$,
is a bouquet for $\mb A(i)$ (in particular, $d_-(s,i) \subset M_i$ and $d_+(s,i) \subset N_i$),}
\item{$\de_s d_\pm (0, t) = v_1^\pm$ and $\de_s d_\pm (1,t)$ is transverse to $Z$,}
\item{$v_\pm (0, t) = v_2^\pm$,}
\item{$\xi = \langle \de_t d_\pm, v_\pm \rangle$ along $d_\pm(1,t)$, and} \label{it:xi}
\item{$\de_t d_\pm (s, i)$ is transverse to $M_i \sqcup N_i$ for $i \in \{0,1\}$.}
\end{enumerate}
\end{definition}

\begin{theorem}
\label{thm:cs}
Let $\mc W = (W, Z_W, [\xi_W])$ be a balanced cobordism from $(M_0, \gamma_0)$ to
$(M_1, \gamma_1)$, and let $\mc U = (U, Z_U, [\xi_U])$ be a balanced
cobordism from $(N_0, \ni_0)$ to $(N_1, \ni_1)$.
Let $\mb A$ be a framed pair of arcs for $\mc W$ and $\mc U$,
and suppose that there exists some bouquet $B(\mb A)$ for $\mb A$.
Then the following diagram commutes:
\begin{center}
\begin{tikzpicture}[description/.style={fill=white,inner sep=2pt}]
\matrix (m) [matrix of math nodes, row sep=4em,
column sep=7.5em, text height=1.7ex, text depth=0.3ex]
{\SFH(M_0 \#_{\mb A(0)} N_0, \g_0 \sqcup \nu_0) & \SFH(M_0,\g_0) \otimes \SFH(N_0,\nu_0) \otimes V \\
 \SFH(M_1 \#_{\mb A(1)} N_1, \g_1 \sqcup \nu_1) & \SFH(M_1,\g_1) \otimes \SFH(N_1,\nu_1) \otimes V \\};
	\path[->,font=\scriptsize]
		(m-1-1) edge node[above]{$\varphi_{\mb A(0)}$} (m-1-2)
		(m-2-1) edge node[above]{$\varphi_{\mb A(1)}$} (m-2-2)
		(m-1-1) edge node[description]{$F_{\mc W \cs_{\mb A} \mc U}$} (m-2-1)	
		(m-1-2) edge node[description]{$F_{\mc W} \otimes F_{\mc U} \otimes \id_V$} (m-2-2);
\end{tikzpicture}
\end{center}
\end{theorem}

\begin{example}
The existence of a bouquet $B(\mb A)$ in the statement of Theorem~\ref{thm:cs}
is not just a technical condition, as we are going to show in this example.
Consider the cobordism $\CobGamma \circ \CobL$ (see Figure~\ref{fig:defcobordisms}
for a definition of the cobordisms $\CobGamma$ and $\CobL$). By the computations
in Subsection~\ref{sec:computations} and functoriality, we know that the map
\[
F_{\CobGamma \circ \CobL}: \F_2^2 \xrightarrow{\qquad} \F_2^2
\]
induced on~$\SFH$
has rank~$1$.

This cobordism can also be seen as the disjoint union $\mc X_1 \cs \mc X_2$
of the identity cobordism $\mc X_1$ from $(U_1, P_1)$ to $(U_1, P_1)$
and a ``death and re-birth'' cobordism $\mc X_2$ again from $(U_1, P_1)$
to $(U_1, P_1)$; see Figure~\ref{fig:gammaelle}.
Choose $\mb A$ as in Remark~\ref{rem:cs}. Then
\[
F_{\CobGamma \circ \CobL} = F_{\W(\mc X_1) \cs_{\mb A} \W(\mc X_2)}.
\]
We cannot apply Theorem~\ref{thm:cs} here because there is no bouquet for $\mb A$.
If the diagram in Theorem~\ref{thm:cs} were commutative in this case,
then the map induced by $\W(\mc X_1) \cs_{\mb A} \W(\mc X_2)$ would have rank
either $0$ or $2$, contradicting our computations.
\end{example}

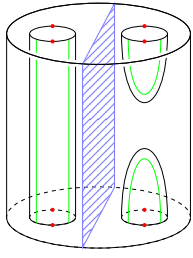
\begin{figure}
\resizebox{0.95\textwidth}{!}{\input{gammaelle.tex}}
\caption{The cobordism $\Gamma \circ \CobL$ can be seen as the disjoint
union cobordism of two cobordisms from $(U_1, P_1)$ to $(U_1, P_1)$,
namely the identity cobordism to the left of the blue rectangle, which
we denote by $\mc X_1$, and the one to its right, which we denote by~$\mc X_2$.}
\label{fig:gammaelle}
\end{figure}

\begin{proof}[{Proof of Theorem~\ref{thm:cs}}]

Split $\W$ and $\mc U$ into a boundary cobordism and a special
cobordism: $\mc W = \Ws \circ \Wb$ and $\mc U = \mc U^s \circ \mc U^b$.
Then $\mb A$ determines framed pairs of arcs $\mb A^b$ for $\mc W^b$ and
$\mc U^b$ and $\mb A^s$ for $\mc W^s$ and $\mc U^s$.
Note that
\[
\mc W \cs_{\mb A} \mc U =
\left(\mc W^s \cs_{\mb A^s} \mc U^s \right) \circ \left(\mc W^b \cs_{\mb A^b} \mc U^b \right)
\]
is a decomposition of $\mc W \cs_{\mb A} \mc U$ into a boundary cobordism
and a special cobordism.

Fix a bouquet $B(\mb A)$ for $\mb A$. Note that this induces bouquets
$B(\mb A^b)$ and $B(\mb A^s)$. Therefore, we can consider the cases of
boundary cobordism and special cobordism separately. We now split the
proof into three cases.

\paragraph{Case 1}
We first consider the case when $\mc W$ and $\mc U$ are identity
cobordisms, and $\mb A$ is any framed pair of arcs that admits a
bouquet~$B(\mb A)$. Let $B_-$ be a regular neighbourhood
of $\im(d_-)$ in $W$, and let $B_+$ be a regular neighbourhood
of $\im(d_+)$ in $U$.
For $i \in \{0,1\}$, let  $B_\pm(i) = B_\pm \cap (M_i \sqcup N_i)$.
By condition~\eqref{it:xi} of Definition~\ref{def:bouquet},
the contact structure~$\xi$ is parallel near the bouquet,
hence $B_\pm$ can be turned into a special cobordism
\[
\mc B_\pm \colon B_\pm(0) \to B_\pm(1)
\]
by extending~$\xi$ to $\de B_\pm \cap \text{Int}(W \cup U)$.
If we remove $\mc B_\pm$ from $\mc W$ and $\mc U$, we obtain
two identity cobordisms $\mc W^\circ = \id_M$ and $\mc U^\circ = \id_N$,
where $M = M_0 \sm B_-(0)$ and $N = N_0 \sm B_+(0)$.
Note that $\mb A$ is a framed pair of arcs
and $B(\mb A)$ is a bouquet in $\mc B_\pm$. The cobordism $\mc B_- \cs_{\mb A} \mc B_+$
is diffeomorphic to an identity cobordism, therefore it gives
a diffeomorphism
\[
\psi \colon B_-(0) \#_{\mb A(0)} B_+(0) \to B_-(1) \#_{\mb A(1)} B_+(1),
\]
and, without loss of generality, we can suppose that
\begin{itemize}
\item{$\psi(a_\pm(s,0)) = a_\pm(s,1)$, \text{ and}}
\item{$\mathrm{d}\psi (v_\pm(s,0)) = v_\pm(s,1)$.}
\end{itemize}

Let $D_\pm(0)$ be a $2$-disc embedded in $B_\pm(0)$ containing
$a_\pm(\cdot, 0)$ such that $v_\pm(\cdot, 0)$ is tangent to
$D_\pm(0)$, and let $D_\pm(1) = \psi(D_\pm(0))$. Then
$D_-(i) \#_{\mb A(i)} D_+(i)$, together with a meridional $\alpha$- and a $\beta$-curve
on the connecting tube that intersect in two points, is a Heegaard diagram~$\mc H_i$
for $B_-(i) \#_{\mb A(i)} B_+(i)$, and $\psi_*(\mc H_0) = \mc H_1$.
The map induced by such a diffeomorphism on sutured
Floer homology preserves the relative homological grading,
so it must be the identity map:
\begin{center}
\begin{tikzpicture}[description/.style={fill=white,inner sep=2pt}]
\matrix (m) [matrix of math nodes, row sep=4em,
column sep=4.5em, text height=1.7ex, text depth=0.3ex]
{\SFH(\mc H_0) & V \\
 \SFH(\mc H_1) & V. \\};
	\path[->,font=\scriptsize]
		(m-1-1) edge node[above]{$\simeq$} (m-1-2)
		(m-2-1) edge node[above]{$\simeq$} (m-2-2)
		(m-1-1) edge node[right]{$\psi_*$} (m-2-1)	
		(m-1-2) edge node[right]{$\id_V$} (m-2-2);
\end{tikzpicture}
\end{center}

We then choose Heegaard diagrams $\mc H_M$ and $\mc H_N$
for the sutured manifolds~$M$ and~$N$ such that
\[
\mc H_{M \#_{\mb A(0)} N} = \mc H_M \natural \mc H_0 \natural \mc H_N
\]
is a Heegaard diagram for the sutured manifold $M \#_{\mb A(0)} N$
(or, more precisely, for $M_0 \#_{\mb A(0)} N_0$).
The cobordism $\mc W \cs_{\mb A} \mc U$ induces a diffeomorphism
\[
\Psi_* \colon \mc H_{M \#_{\mb A(0)} N} \to \mc H_{M \#_{\mb A(1)} N} = \mc H_M \natural \mc H_1 \natural \mc H_N
\]
such that the following diagram of diffeomorphisms commutes:
\begin{center}
\begin{tikzpicture}[description/.style={fill=white,inner sep=2pt}]
\matrix (m) [matrix of math nodes, row sep=4em,
column sep=4.5em, text height=1.7ex, text depth=0.3ex]
{\mc H_{M \#_{\mb A(0)} N} & \mc H_M \natural \mc H_0 \natural \mc H_N \\
 \mc H_{M \#_{\mb A(1)} N} & \mc H_M \natural \mc H_1 \natural \mc H_N. \\};
	\path[->,font=\scriptsize]
		(m-1-1) edge node[above]{$\simeq$} (m-1-2)
		(m-2-1) edge node[above]{$\simeq$} (m-2-2)
		(m-1-1) edge node[right]{$\Psi_*$} (m-2-1)	
		(m-1-2) edge node[description]{$\id_{\mc H_M} \natural \psi_* \natural \id_{\mc H_N}$} (m-2-2);
\end{tikzpicture}
\end{center}
From this, it follows that the diagram below commutes:
\begin{center}
\begin{tikzpicture}[description/.style={fill=white,inner sep=2pt}]
\matrix (m) [matrix of math nodes, row sep=4em,
column sep=7.5em, text height=1.7ex, text depth=0.3ex]
{\SFH(\mc H_{M \#_{\mb A(0)} N}) & \SFH(\mc H_M) \otimes \SFH(\mc H_N) \otimes V \\
 \SFH(\mc H_{M \#_{\mb A(1)} N}) & \SFH(\mc H_M) \otimes \SFH(\mc H_N) \otimes V. \\};
	\path[->,font=\scriptsize]
		(m-1-1) edge node[above]{$\varphi_{\mc H_M, \mc H_N, \mb A(0)}$} (m-1-2)
		(m-2-1) edge node[above]{$\varphi_{\mc H_M, \mc H_N, \mb A(1)}$} (m-2-2)
		(m-1-1) edge node[description]{$F_{\mc H_{M \#_{\mb A(0)} N}, \mc H_{M \#_{\mb A(1)} N}}$} (m-2-1)	
		(m-1-2) edge node[description]{$\id_{\SFH(\mc H_M)} \otimes \id_{\SFH(\mc H_N)} \otimes \id_V$} (m-2-2);
\end{tikzpicture}
\end{center}

\paragraph{Case 2: $\mc W$ and $\mc U$ are boundary cobordisms.}
Then $M_1 \cong M_0 \cup -Z_W$ and $N_1 \cong N_0 \cup -Z_U$.
Given a bouquet $B(\mb A) = (d_\pm, v_\pm)$ for~$\mb A$, we define a bouquet
$B'(\mb A(0))$ for $\mb A(0)$ in the sutured manifold $(M_1, \gamma_1) \sqcup (N_1, \ni_1)$.
The arcs of $B'(\mb A(0))$ are $\eta_\pm = d_\pm(\cdot,0) \cup d_\pm(1,\cdot)$,
and the normal vector field is just the restriction of $v_\pm$.

Consider a sutured Heegaard diagram
$\mc H_{M_0}$ for
the sutured manifold $(M_0, \gamma_0)$ subordinate to the bouquet $B(\mb A(0))$, and an extension
of it to a sutured Heegaard diagram $\mc H_{M_1}$
for $(M_1, \gamma_1)$ subordinate to the bouquet $B'(\mb A(0))$
that are compatible with the contact structure $\xi_W$
as in the definition of the gluing map; cf.~\cite{gluingmap}.
Choose analogous sutured Heegaard diagrams $\mc H_{N_0}$ for $(N_0, \ni_0)$ and
$\mc H_{N_1}$ for $(N_1, \ni_1)$.
The gluing maps are then defined at the complex level as the tensor product
with contact classes, denoted by~$\mathbf x''_M$ and~$\mathbf x''_N$.

Thus, we can suppose that $\mc H_{M_0} \#_{\mb A(0)} \mc H_{N_0}$
is subordinate to the bouquet $B(\mb A(0))$ and $\mc H_{M_1} \#_{\mb A(0)} \mc H_{N_1}$
is subordinate to the bouquet $B'(\mb A(0))$. Furthermore, the sutured diagrams
$\mc H_{M_0} \#_{\mb A(0)} \mc H_{N_0}$ and $\mc H_{M_1} \#_{\mb A(0)} \mc H_{N_1}$
are contact-compatible, and that the contact class is just
$\mathbf x''_M \otimes \mathbf x''_N$. Therefore, the following diagram is commutative:
\begin{center}
\begin{tikzpicture}[description/.style={fill=white,inner sep=2pt}]
\matrix (m) [matrix of math nodes, row sep=4em,
column sep=7.5em, text height=1.7ex, text depth=0.3ex]
{\SFH(\mc H_{M_0} \#_{\mb A(0)} \mc H_{N_0}) & \SFH(\mc H_{M_0}) \otimes \SFH(\mc H_{N_0}) \otimes V \\
 \SFH(\mc H_{M_1} \#_{\mb A(0)} \mc H_{N_1}) & \SFH(\mc H_{M_1}) \otimes \SFH(\mc H_{N_1}) \otimes V. \\};
	\path[->,font=\scriptsize]
		(m-1-1) edge node[above]{$\varphi_{\mc H_{M_0}, \mc H_{N_0}, \mb A(0)}$} (m-1-2)
		(m-2-1) edge node[above]{$\varphi_{\mc H_{M_1}, \mc H_{N_1}, \mb A(0)}$} (m-2-2)
		(m-1-1) edge node[description]{$\Phi_{\xi_W \sqcup \xi_U}$} (m-2-1)	
		(m-1-2) edge node[description]{$\Phi_{\xi_W} \otimes \Phi_{\xi_U} \otimes \id_V$} (m-2-2);
\end{tikzpicture}
\end{center}

We are now left with the connected sum cobordism
$\id_{M_1} \cs_{\mb A} \id_{N_1}$, which still contains a bouquet,
induced by~$B(\mb A)$.
It follows from Case~1 that we have a commutative diagram as below,
which concludes the proof in the case of boundary cobordisms:
\begin{center}
\begin{tikzpicture}[description/.style={fill=white,inner sep=2pt}]
\matrix (m) [matrix of math nodes, row sep=4em,
column sep=7.5em, text height=1.7ex, text depth=0.3ex]
{\SFH(\mc H_{M_1} \#_{\mb A(0)} \mc H_{N_1}) & \SFH(\mc H_{M_1}) \otimes \SFH(\mc H_{N_1}) \otimes V \\
 \SFH(\mc H_{M_1} \#_{\mb A(1)} \mc H_{N_1}) & \SFH(\mc H_{M_1}) \otimes \SFH(\mc H_{N_1}) \otimes V. \\};
	\path[->,font=\scriptsize]
		(m-1-1) edge node[above]{$\varphi_{\mc H_{M_1}, \mc H_{N_1}, \mb A(0)}$} (m-1-2)
		(m-2-1) edge node[above]{$\varphi_{\mc H_{M_1}, \mc H_{N_1}, \mb A(1)}$} (m-2-2)
		(m-1-1) edge node[description]{$F_{\mc H_{M_1} \#_{\mb A(0)} \mc H_{N_1}, \mc H_{M_1} \#_{\mb A(1)} \mc H_{N_1}}$} (m-2-1)	
		(m-1-2) edge node[description]{$\id_{\SFH(\mc H_{M_1})} \otimes \id_{\SFH(\mc H_{N_1})} \otimes \id_V$} (m-2-2);
\end{tikzpicture}
\end{center}

\paragraph{Case 3: $\mc W$ and $\mc U$ are special cobordisms.}
Consider the function~$t$ on $B(\mb A)$, and extend it to a
Morse function on $\mc W \sqcup \mc U$.
This gives a decomposition of the cobordisms $\mc W$ and~$\mc U$
into 1-, 2-, and 3-handles that are attached
\emph{away from the bouquet}. Furthermore, $B(\mb A)$ is a product
with respect to this handle decomposition. In particular, $\mb A(0)$
coincides with $\mb A(1)$ and $B(\mb A)_0$ coincides with $B(\mb A)_1$
after the handle attachments.

For each handle attachment, one can choose adapted Heegaard diagrams
that are also subordinate to the bouquet $B(\mb A)_0$ for the framed
pair of points $\mb A(0)$.
We therefore obtain a commutative diagram
\begin{center}
\begin{tikzpicture}[description/.style={fill=white,inner sep=2pt}]
\matrix (m) [matrix of math nodes, row sep=4em,
column sep=7.5em, text height=1.7ex, text depth=0.3ex]
{\SFH(M_0 \#_{\mb A(0)} N_0) & \SFH(M_0) \otimes \SFH(N_0) \otimes V \\
 \SFH(M_1 \#_{\mb A(0)} M_1) & \SFH(M_1) \otimes \SFH(N_1) \otimes V. \\};
	\path[->,font=\scriptsize]
		(m-1-1) edge node[above]{$\varphi_{\mb A(0)}$} (m-1-2)
		(m-2-1) edge node[above]{$\varphi_{\mb A(0)}$} (m-2-2)
		(m-1-1) edge node[description]{$F^s_{\mc W \cs_{\mb A} \mc U}$} (m-2-1)	
		(m-1-2) edge node[description]{$F^s_{\mc W} \otimes F^s_{\mc U} \otimes \id_V$} (m-2-2);
\end{tikzpicture}
\end{center}
This concludes the proof of Theorem~\ref{thm:cs}.
\end{proof}

%% file: gammaelle.tex
\begin{tikzpicture}[crossline/.style={preaction={draw=white, -, line width=5pt}}]

\def\u{0.6cm};
\def\h{4*\u};
\def\v{4*\u};
\def\b{3*\u};
\def\c{4*\u};
\def\hshift{12*\u};
\def\d{0.23570226039*\u};
\def\smallradius{\u/15};
\def\seno{0.98480775301*\h/3};
\def\coseno{0.17364817766*\h};


\begin{scope}[shift={(0,0-\v)}]
\def\halfradius{\h};
\draw (0,0) ellipse (\halfradius*1 and \halfradius/3);
\draw[white, fill=white] (0-1.2*\halfradius,0) rectangle (1.2*\halfradius,\halfradius/2);
\draw[dashed] (0,0) ellipse (\halfradius*1 and \halfradius/3);
\end{scope}

\begin{scope}[shift={(-\h/2,0-\v)}]
\def\halfradius{\u};
\draw (0,0) ellipse (\halfradius*1 and \halfradius/3);
\draw[white, fill=white] (0-1.2*\halfradius,0) rectangle (1.2*\halfradius,\halfradius/2);
\draw[dashed] (0,0) ellipse (\halfradius*1 and \halfradius/3);
\end{scope}
\begin{scope}[shift={(\h/2,0-\v)}]
\def\halfradius{\u};
\draw (0,0) ellipse (\halfradius*1 and \halfradius/3);
\draw[white, fill=white] (0-1.2*\halfradius,0) rectangle (1.2*\halfradius,\halfradius/2);
\draw[dashed] (0,0) ellipse (\halfradius*1 and \halfradius/3);
\end{scope}

\begin{scope}[shift={(-\h/2,0)}]
\draw (-\u, \v) -- (-\u, -\v);
\draw (\u, \v) -- (\u, -\v);
\draw[color=green] (-3*\d, \v-\d) -- (-3*\d, -\v-\d);
\draw[color=green] (3*\d, \v-\d) -- (3*\d, -\v-\d);
\end{scope}

\draw (\h/2-\u,\v) .. controls (\h/2-\u,\v-\c) and (\h/2+\u,\v-\c) .. (\h/2+\u,\v);
\draw[color=green] (\h/2-3*\d,\v-\d) .. controls (\h/2-3*\d,\v-\d-0.8*\c) and (\h/2+3*\d,\v-\d-0.8*\c) .. (\h/2+3*\d,\v-\d);

\draw (\h/2-\u,-\v) .. controls (\h/2-\u,-\v+\c) and (\h/2+\u,-\v+\c) .. (\h/2+\u,-\v);
\draw[color=green] (\h/2-3*\d,-\v-\d) .. controls (\h/2-3*\d,-\v-\d+0.93*\c) and (\h/2+3*\d,-\v-\d+0.93*\c) .. (\h/2+3*\d,-\v-\d);

\begin{scope}[shift={(0,\v)}]
\def\halfradius{\h};
\draw[crossline] (0,0) ellipse (\halfradius*1 and \halfradius/3);
\end{scope}

\draw[thick, lightblue, pattern=north east lines, pattern color=lightblue] (-\coseno, \v-\seno) -- (\coseno, \v + \seno) -- (\coseno, -\v+\seno) -- (-\coseno, -\v-\seno) -- cycle;

\begin{scope}[shift={(0,\v)}]
\def\halfradius{\h};
\draw[] (0,0) ellipse (\halfradius*1 and \halfradius/3);
\end{scope}

\begin{scope}[shift={(\h/2,\v)}]
\def\halfradius{\u};
\draw (0,0) ellipse (\halfradius*1 and \halfradius/3);
\end{scope}
\begin{scope}[shift={(-\h/2,\v)}]
\def\halfradius{\u};
\draw (0,0) ellipse (\halfradius*1 and \halfradius/3);
\end{scope}

\begin{scope}[shift={(0,\v)}]
\draw[color=red,fill=red] (-\h/2, \u/3) circle (\smallradius);
\draw[color=red,fill=red] (-\h/2, -\u/3) circle (\smallradius);
\draw[color=red,fill=red] (\h/2, \u/3) circle (\smallradius);
\draw[color=red,fill=red] (\h/2, -\u/3) circle (\smallradius);
\end{scope}
\begin{scope}[shift={(-\h/2,-\v)}]
\draw[color=red,fill=red] (0, \u/3) circle (\smallradius);
\draw[color=red,fill=red] (0, -\u/3) circle (\smallradius);
\end{scope}
\begin{scope}[shift={(\h/2,-\v)}]
\draw[color=red,fill=red] (0, \u/3) circle (\smallradius);
\draw[color=red,fill=red] (0, -\u/3) circle (\smallradius);
\end{scope}

\draw (-\h, \v) -- (-\h, -\v);
\draw (\h, \v) -- (\h, -\v);

\draw (-1.5*\hshift-\h,0);
\draw (1.5*\hshift+\h,0);

\end{tikzpicture}

%% file: TQFT.tex
The aim of this section is to determine the $(1+1)$-dimensional
TQFT defined by the cobordism maps computed in Section~\ref{sec:computations}.

\subsection{The category of marked embedded cobordisms and the $\HFLh$ TQFT}
\label{sec:HFLTQFT}

We first define the cobordism category that we are interested in.

\begin{definition}
We define the category of \emph{marked decorated embedded cobordisms} $\EmbCob$ as follows.
The objects are the standard $n$-component unlinks $(U_n, P_n)$ for every $n \in \mb N$
with two decorations on each component. As Kronheimer and Mrowka suggest
in \cite[Section~8.2]{detection}, we can take a specific
model for it: We define $U_n$ to be the union of standard
circles $L_1, \dots, L_n$ in the $(x,y)$-plane, each of diameter $0.5$,
centered at the first $n$ integer lattice points along the
$x$-axis, and such that $R_-(P_n) \cap L_i = L_i \cap \{x \leq i\}$.
Furthermore, we suppose that the component~$L_1$ of the unlink $U_n$
is marked (in addition to being decorated),
and denote it by a dot on the component itself.

The morphisms are all cobordisms in $S^3 \times I$ generated by
\begin{itemize}
\item{$\CobV^e_n$ ($n \geq 1$), the cobordism from $(U_n, P_n)$ to $(U_{n+1}, P_{n+1})$
obtained by stacking horizontally the cobordism $\CobV$ in Figure~\ref{fig:defcobordisms}
and the identity cobordism of $(U_{n-1}, P_{n-1})$. Note that the marked component
is in the $\CobV$ part of the cobordism.}
\item{$\CobLambda^e_n$ ($n \geq 2$), the cobordism from $(U_n, P_n)$ to $(U_{n-1}, P_{n-1})$
obtained by stacking horizontally the cobordism $\CobLambda$ in Figure~\ref{fig:defcobordisms}
and the identity cobordism of $(U_{n-2}, P_{n-2})$. Note that the marked component
is in the $\CobLambda$ part of the cobordism.}
\item{$\IX^e_{i,n}$ (for $2 \leq i \leq n-1$ and $n \geq 3$), obtained by swapping the
$i$-th and the $(i+1)$-th component of $(U_n, P_n)$. Note that the marked component
is in the first component of the cobordism, and that therefore the marked component
cannot be swapped.}
\item{$\IV^e_n$ ($n \geq 2$), the cobordism from $(U_n, P_n)$ to $(U_{n+1}, P_{n+1})$
obtained by stacking horizontally an identity cobordism $\CobI$ between the marked components,
the cobordism $\CobV$ and, finally, identity cobordisms on the last $n-2$ components.}
\item{$\ILambda^e_n$ ($n \geq 3$), the cobordism from $(U_n, P_n)$ to $(U_{n-1}, P_{n-1})$
obtained by stacking horizontally an identity cobordism $\CobI$ between the marked components,
the cobordism $\CobLambda$ and, finally, identity cobordisms on the last $n-3$ components.}
\end{itemize}
The above cobordisms are shown in Figure~\ref{fig:HFLTQFT}.
\end{definition}

Our aim is to describe the TQFT
\[
\HFLh \colon \EmbCob \xrightarrow{\qquad} \Vect_{\F_2}.
\]

\begin{figure}
\resizebox{0.65\textwidth}{!}{
\input{HFLTQFT.tex}
}
\caption{The figure shows the cobordisms $\V^e_1$, $\Lambda^e_2$,
$\IX^e_{2,3}$, $\IV^e_2$ and $\ILambda^e_3$ in the category $\EmbCob$
(represented following the conventions explained in Figure~\ref{fig:handles}),
and the maps they induce via the TQFT $\HFLh$. Note that the
marked component allows us to choose a canonical identification
$\psi_n \colon \HFLh(U_n, P_n) \to V^{\otimes (n-1)}$.}
\label{fig:HFLTQFT}
\end{figure}

To do so, we need to deal with some technicalities due to the fact
that link Floer homology does not naturally arise as an invariant
of a marked link. For this reason, we now explain how marking a
component of $(U_n, P_n)$ gives an isomorphism
\begin{equation}
\label{eq:psiUn}
\psi_n \colon \HFLh(U_n, P_n)  \xrightarrow{\quad \simeq \quad} V^{\otimes (n-1)}.
\end{equation}

As usual, let $V$ be the $\F_2$ vector space of dimension~$2$, generated
by two homogeneous elements $T$ (top-graded) and $B$
(bottom-graded). We now construct a canonical
isomorphism
\[
\HFLh(U_n) \cong V^{\otimes (n-1)},
\]
where each factor is
associated to an unmarked component of $U_n$.
Note that a sutured Heegaard diagram for the unknot in $S^3$
is given by $(A, \emptyset, \emptyset)$, where $A$ is an annulus.
We can construct a (sutured) Heegaard diagram~$\mc H_n$ for $(U_n, P_n)$
by taking $n$ copies of this Heegaard diagram (each one associated
to a link component), and by connecting the first annulus (i.e., the
one that corresponds to the marked component) to all the other
annuli with connected sum tubes. Each tube contains a homotopically non-trivial
$\alpha$-curve and a homotopically non-trivial $\beta$-curve
that intersect in two points, which we call $B$ and $T$
(for bottom-graded and top-graded).
Then we have an isomorphism
\[
\psi_n \colon \HFLh(U_n, P_n) \xrightarrow{\quad \simeq \quad} V_2 \otimes \ldots \otimes V_{n},
\]
where by $V_i$ we mean the vector space $V$ associated to the tube
connecting the first component to the $i$-th component.

Having set the isomorphism in equation~\eqref{eq:psiUn}, our next aim
is to understand the cobordism maps induced by the generators of $\EmbCob$
with respect to the standard basis consisting of elements of the form
\[
v_2 \otimes \dots \otimes v_n,
\]
where $v_i \in \{B,T\}$ for every $i \in \{2,\dots,n\}$.

\subsubsection{The cobordism $\V^e_n$}
\label{sec:CobV}
This cobordism is the disjoint union of the cobordism $\CobV$ and
the identity cobordism $\CobI^{\otimes (n-1)} = \id_{U_{n-1}}$.
Consider a framed pair of arcs $\mb A$ for $\CobV$ and $\id_{U_{n-1}}$
such that they are close to the surfaces,
and their endpoints are close to the marked components of the links.
Then, by Corollary~\ref{cor:canonicalcs}, there are isomorphisms
\[
\begin{split}
\varphi_n \colon \HFLh(U_n, P_n) \to \mb F_2 \otimes V^{\otimes(n-2)} \otimes V, \text{ and}\\
\varphi_{n+1} \colon \HFLh(U_{n+1}, P_{n+1}) \to V \otimes V^{\otimes(n-2)} \otimes V
\end{split}
\]
such that the cobordism map can be expressed with respect to these isomorphisms
as in Theorem~\ref{thm:cs}. However, $\varphi_n$ and $\varphi_{n+1}$ are not the
isomorphisms~$\psi_n$ and~$\psi_{n+1}$
given in equation~\eqref{eq:psiUn} for decorated unlinks with a marked component.

We define the isomorphism $\varphi_n$ using the sutured diagram $\mc H_n' := \mc H_1 \#_{\mb A(0)} \mc H_{n-1}$,
where $\mb A(0)$ has one point in $\mc H_1$ and one point in the first
$(A,\emptyset,\emptyset)$ summand of $H_{n-1}$. The tube attached along $\mb A(0)$ corresponds to the last $V$
factor of $\mb F_2 \otimes V^{\otimes(n-2)} \otimes V$.
Similarly, we define $\varphi_{n+1}$ using the diagram $\mc H_{n+1}' := \mc H_2 \#_{\mb A(1)} \mc H_{n-1}$,
where $\mb A(1)$ has one point in the first summand of $\mc H_2$ and one point in the first summand
of $\mc H_{n-1}$. On the other hand, $\psi_n$ is defined using $\mc H_n$, and $\psi_{n+1}$
is defined using $\mc H_{n+1}$. We obtain $\mc H_n'$ from $\mc H_n$ by sliding the feet
of all tubes in the first summand over the tube attached along $\mb A(0)$. This gives rise
to the naturality map $F_{\mc H_n, \mc H_n'}$. We obtain $\mc H_{n+1}'$ from $\mc H_{n+1}$
in a similar way, except we do not slide the tube~$A_1$ connecting the first two summands of $\mc H_{n+1}$
(the tube attached along $\mb A(1)$ connects the first and third summands).
In particular, the map $F_{\mc H_{n+1},\mc H_{n+1}'}$ acts on the intersection points in $A_1$
via the identity, and on the other intersection points via $F_{\mc H_n, \mc H_n'}$.
This can be seen by observing that the triangle maps involved in $F_{\mc H_{n+1},\mc H_{n+1}'}$
have only small triangles in~$A_1$. Alternatively, one can apply Lemma~\ref{lem:naturalitycs},
by viewing $U_{n+1}$ as $U_1 \sqcup U_n$, where $U_1$ is the second component of $U_{n+1}$ and $U_n$ is the rest,
and $\mb P$ connects the first and second components:
\[
  \xymatrixcolsep{5pc}\xymatrix{\SFH(\mc H_{n+1}) \ar[r]^-{\varphi_{\mc H_1, \mc H_n, \P}} \ar[d]^{F_{\mc H_{n+1}, \mc H_{n+1}'}} &
    \SFH(\mc H_1) \otimes \SFH(\mc H_n) \otimes V \ar[d]^{\id_{\mb F_2} \otimes F_{\mc H_n, \mc H_n'} \otimes \id_V}
    \ar[r]^-{\tau \circ (\psi_n \otimes \id_V)} & V \otimes V^{\otimes (n-1)} \ar[d]^{\id_V \otimes f_n}\\
    \SFH(\mc H_{n+1}')\ar[r]^-{\varphi_{\mc H_1, \mc H_n', \P}} &
    \SFH(\mc H_1) \otimes \SFH(\mc H_n') \otimes V \ar[r]^-{\tau \circ (\varphi_n \otimes \id_V)}&
    V \otimes V^{\otimes (n-2)} \otimes V,}
\]
where $f_n = \varphi_n \circ F_{\mc H_n, \mc H_n'} \circ \psi_n^{-1}$ and $\tau(x \otimes y) = y \otimes x$
for $x \in V^{\otimes (n-1)}$ and $y \in V$.
Furthermore, $\psi_{n+1} =  \tau \circ (\psi_n \otimes \id_V) \circ \varphi_{\mc H_1, \mc H_n, \P}$ and
$\varphi_{n+1} = \tau \circ (\varphi_n \otimes \id_V) \circ \varphi_{\mc H_1, \mc H_n', \P}$,
hence we obtain that
\[
\varphi_{n+1} \circ F_{\mc H_{n+1}, \mc H_{n+1}'} \circ \psi_{n+1}^{-1} = \id_V \otimes f_n.
\]

If we consider the commutative diagram
\begin{center}
\begin{tikzpicture}[description/.style={fill=white,inner sep=2pt}]
\matrix (m) [matrix of math nodes, row sep=3em,
column sep=5em, text height=1.7ex, text depth=0.3ex]
{\HFLh(\mc H_n) & \F_2 \otimes V^{\otimes(n-1)} \\
 \HFLh(\mc H_n') & \F_2 \otimes \left( V^{\otimes(n-2)} \otimes V \right)\\
 \HFLh(\mc H_{n+1}') & V \otimes \left( V^{\otimes(n-2)} \otimes V \right)\\
 \HFLh(\mc H_{n+1}) & V \otimes V^{\otimes(n-1)}, \\};
	\path[->,font=\scriptsize]
		(m-1-1) edge node[above]{$\psi_n$} (m-1-2)
		(m-2-1) edge node[above]{$\varphi_n$} (m-2-2)
		(m-3-1) edge node[above]{$\varphi_{n+1}$} (m-3-2)
		(m-4-1) edge node[above]{$\psi_{n+1}$} (m-4-2)
		(m-1-1) edge node[right]{$F_{\mc H_n, \mc H_n'}$} (m-2-1)	
		(m-2-1) edge node[right]{$ F_{\V^e_n}$} (m-3-1)	
		(m-3-1) edge node[right]{$F_{\mc H_{n+1}', \mc H_{n+1}}$} (m-4-1)	
		(m-1-2) edge node[description]{$\id_{\mb F_2} \otimes  f_n$} (m-2-2)
		(m-2-2) edge node[description]{$ F_{\V} \otimes \id_{V^{\otimes(n-2)}} \otimes \id_V$} (m-3-2)
		(m-3-2) edge node[description]{$\id_{V} \otimes  f_n^{-1}$} (m-4-2);
\end{tikzpicture}
\end{center}
then we see that the map $ F_{\V^e_n}$ with respect to the isomorphisms
$\psi_n$ and $\psi_{n+1}$ is
\begin{IEEEeqnarray*}{rccc}
 F_{\V^e_n}: & \,\F_2 \otimes V^{\otimes(n-1)} \, & \xrightarrow{\qquad} & \, V \otimes V^{\otimes(n-1)}. \\
& \,1 \otimes *\, & \xmapsto{\qquad} & \,B \otimes *
\end{IEEEeqnarray*}

\subsubsection{The cobordism $\Lambda^e_n$}
\label{sec:CobLambda}

This cobordism is the disjoint union of the merge cobordism~$\Lambda$ as in
Figure~\ref{fig:defcobordisms}, from the first two components of~$U_n$
to the first component of~$U_{n-1}$, and the identity cobordism on the
last $n-2$ components. By arguing as in Section~\ref{sec:CobV},
we see that the map $ F_{\Lambda^e_n}$ with respect to the isomorphisms
$\psi_n$ and $\psi_{n-1}$ is
\begin{IEEEeqnarray*}{rccc}
 F_{\Lambda^e_n}: & \,V \otimes V^{\otimes(n-2)} \, & \xrightarrow{\qquad} & \, \F_2 \otimes V^{\otimes(n-2)}. \\
& \,T \otimes *\, & \xmapsto{\qquad} & \,1 \otimes * \\
& \,B \otimes *\, & \xmapsto{\qquad} & \,0 \\
\end{IEEEeqnarray*}

\subsubsection{The cobordism $\IX^e_{i,n}$}

The effect of the cobordism $\IX^e_{i,n}$ on the Heegaard diagram
that induces the isomorphism $\psi_n$ in equation~\eqref{eq:psiUn}
is a swap between the two (unmarked) components $i$ and $i+1$.
The map on $\HFLh$ is therefore
\begin{IEEEeqnarray*}{rccc}
 F_{\IX^e_{i,n}} \colon & \,V^{\otimes(i-2)} \otimes V \otimes V \otimes V^{\otimes(n-i-1)} \, & \xrightarrow{\qquad}
& \,V^{\otimes(i-2)} \otimes V \otimes V \otimes V^{\otimes(n-i-1)}. \, \\
& \,* \otimes a \otimes b \otimes *\, & \xmapsto{\qquad} & \,* \otimes b \otimes a \otimes *\, \\
\end{IEEEeqnarray*}

\subsubsection{The cobordism $\IV^e_n$}
\label{sec:CobIV}

We now turn our attention to the cobordism $\IV^e_n$.
By Theorem~\ref{thm:cs}, and arguing as in Section~\ref{sec:CobV},
we can restrict to the case $n=2$. Then we can split the
cobordism $\IV^e_2$ as the disjoint union of the identity
cobordism~$\CobI$ and the split cobordism~$\CobV$ from Figure~\ref{fig:defcobordisms}.
By Theorem~\ref{thm:cs}, we have isomorphisms
\[
\begin{split}
\varphi_2 \colon \HFLh(U_2, P_{U_2}) &\to \F_2 \otimes \F_2 \otimes V, \text{ and} \\
\varphi_3 \colon \HFLh(U_3, P_{U_3}) &\to \F_2 \otimes V \otimes V,
\end{split}
\]
such that the map $ F_{\IV^e_2}$ under these identifications is as follows:
\begin{IEEEeqnarray*}{rccc}
 F_{\IV^e_{2}} \colon & \,\F_2 \otimes \F_2 \otimes V \, & \xrightarrow{\qquad} & \,\F_2 \otimes V \otimes V. \, \\
& \,1 \otimes 1 \otimes T\, & \xmapsto{\qquad} & \,1 \otimes B \otimes T\, \\
& \,1 \otimes 1 \otimes B\, & \xmapsto{\qquad} & \,1 \otimes B \otimes B\, \\
\end{IEEEeqnarray*}

However, we are interested in writing the map with respect to the identifications
$\psi_2$ and $\psi_3$ in equation~\eqref{eq:psiUn}. Notice that the map
\[
\varphi_2 \circ \psi_2^{-1} \colon V \to \F_2 \otimes \F_2 \otimes V
\]
is clear because it preserves the relative grading. We can therefore identify
\[
\HFLh(U_2, P_{U_2}) \cong V \cong \F_2 \otimes \F_2 \otimes V.
\]
We need to understand $\psi_3 \circ  F_{\IV^e_2}(T)$ and $\psi_3 \circ  F_{\IV^e_2}(B)$.
Notice that $\varphi_3 \circ F_{\IV^e_2}(B)$ is the unique bottom-graded
element of $\F_2 \otimes V \otimes V$, therefore we deduce that
\[
\psi_3 \circ F_{\IV^e_2}(B) = B \otimes B \in V \otimes V.
\]
On the other hand, $T$ is mapped to one of the three middle-graded
homogeneous elements of $\F_2 \otimes V \otimes V$, so $\psi_3 \circ F_{\IV^e_2} (T)$
is also one of the middle-graded homogeneous elements: $T \otimes B$, $B \otimes T$,
or $T \otimes B + B \otimes T$.
If we consider the composition of cobordisms $\IX^e_{2,3} \circ \IV^e_2$,
we obtain the cobordism $\IV^e_2$ \emph{with different decorations}, which, by
Remark~\ref{rem:deco}, induces the same map as $\IV^e_2$ on $\HFLh$.
Therefore, the map $\psi_3 \circ F_{\IV^e_2}$ is invariant
under the swap. In the end, the map associated to $ F_{\IV^e_n}$ with respect
to the identifications in equation~\eqref{eq:psiUn} is
\begin{IEEEeqnarray*}{rccc}
 F_{\IV^e_{n}} \colon & \, V \otimes V^{\otimes(n-2)} \, & \xrightarrow{\qquad}
& \,V^{\otimes 2} \otimes V^{\otimes(n-2)}. \, \\
& \, T \otimes *\, & \xmapsto{\qquad} & \,(T \otimes B + B \otimes T) \otimes *\, \\
& \, B \otimes *\, & \xmapsto{\qquad} & \,(B \otimes B) \otimes *\, \\
\end{IEEEeqnarray*}

\subsubsection{The cobordism $\ILambda^e_n$}
\label{sec:CobILambda}

The last case that we have to study is the cobordism~$\ILambda^e_n$.
As in Section~\ref{sec:CobIV}, we can restrict our attention to $\ILambda^e_3$.
By Theorem~\ref{thm:cs}, we have isomorphisms $\varphi_2$ and $\varphi_3$
as in Section~\ref{sec:CobIV} such that the map $ F_{\ILambda^e_3}$ under
these identifications is as follows:
\begin{IEEEeqnarray*}{rccc}
 F_{\ILambda^e_{3}} \colon & \,\F_2 \otimes V \otimes V \, & \xrightarrow{\qquad} & \,\F_2 \otimes \F_2 \otimes V. \, \\
& \,1 \otimes T \otimes T\, & \xmapsto{\qquad} & \,1 \otimes 1 \otimes T\, \\
& \,1 \otimes T \otimes B\, & \xmapsto{\qquad} & \,1 \otimes 1 \otimes B\, \\
& \,1 \otimes B \otimes T\, & \xmapsto{\qquad} & \,0\, \\
& \,1 \otimes B \otimes B\, & \xmapsto{\qquad} & \,0\, \\
\end{IEEEeqnarray*}
As in Section~\ref{sec:CobIV}, we need to understand this map with respect
to the isomorphisms~$\psi_2$ and~$\psi_3$. By the fact that
$\ILambda^e_3 \circ \IX^e_{2,3}$ is the cobordism $\ILambda^e_3$ (with
different decorations) and by Remark~\ref{rem:deco}, we know that
the map $ F_{\ILambda^e_3}$ with respect to~$\psi_2$ and~$\psi_3$ is
invariant under swaps.
Since we know that there is a unique middle-graded homogeneous element
of $V^{\otimes 2}$ that is in the kernel of $ F_{\ILambda^e_{3}}$, and
since $\ker( F_{\ILambda^e_{3}})$ is invariant under swaps, we deduce that
this element is $T \otimes B + B \otimes T$.
Therefore, the map $ F_{\ILambda^e_n}$ with respect to the identifications
$\psi_n$ and $\psi_{n-1}$ is given by
\begin{IEEEeqnarray*}{rccc}
 F_{\ILambda^e_{n}} \colon & \,V^{\otimes 2} \otimes V^{\otimes (n-3)} \, & \xrightarrow{\qquad} & \,V \otimes V^{\otimes (n-3)}. \, \\
& \,(T \otimes T) \otimes *\, & \xmapsto{\qquad} & \,T \otimes *\, \\
& \,(T \otimes B) \otimes *\, & \xmapsto{\qquad} & \,B \otimes *\, \\
& \,(B \otimes T) \otimes *\, & \xmapsto{\qquad} & \,B \otimes *\, \\
& \,(B \otimes B) \otimes *\, & \xmapsto{\qquad} & \,0\, \\
\end{IEEEeqnarray*}

\begin{remark}
The results above completely determine the TQFT $\HFLh$. This
is expressed in a concise form in Figure~\ref{fig:HFLTQFT}.
\end{remark}

\subsection{The category of marked abstract cobordisms and the reduced Khovanov TQFT}

We now compare the TQFT $\HFLh$ from Section~\ref{sec:HFLTQFT} with another TQFT
from a category of marked cobordisms, namely the reduced Khovanov TQFT, which gives
rise to the reduced Khovanov homology of marked links.
We start with the following definition.

\begin{definition}
We define the category of \emph{marked abstract cobordisms} $\AbsCob$ as follows.
The objects are closed 1-manifolds with a marked component, which we will denote
by a dot on the component itself. Let $U_n$ be a particular closed 1-manifold
with $n$ components; we will suppose that each component is labeled by a number
between $1$ and $n$, and that the marked component is the first one. Notice that
any object of $\AbsCob$ is diffeomorphic to some $U_n$.

The morphisms are all abstract cobordisms between any two closed 1-manifolds.
These are generated by:
\begin{itemize}
\item{$\CobV^a_n$ ($n \geq 1$), the cobordism from $U_n$ to $U_{n+1}$ that consists
of a pair-of-pants between the first component of $U_n$ and the first two components of
$U_{n+1}$, and, for each $i > 1$, of a cylinder between the $i$-th component of $U_n$
and the $(i+1)$-th component of $U_{n+1}$. Notice that the marked components of $U_n$
and $U_{n+1}$ are both contained in the pair-of-pants.}
\item{$\CobLambda^a_n$ ($n \geq 2$), the cobordism from $U_{n}$ to $U_{n-1}$ that consists
of a pair-of-pants between the first two components of $U_n$ and the first component of
$U_{n-1}$, and, for each $i > 2$, of a cylinder between the $i$-th component of $U_n$
and the $(i-1)$-th component of $U_{n-1}$. Notice that the marked components of $U_n$
and $U_{n-1}$ are both contained in the pair-of-pants.}
\item{$\IX^a_{i,n}$ (for $2 \leq i \leq n-1$ and $n \geq 3$), obtained by swapping the
$i$-th and the $(i+1)$-th component of $U_n$. Notice that the marked components
are on the first $\CobI$ component of the cobordism, and that therefore the marked
component cannot be swapped.}
\item{$\IV^a_n$ ($n \geq 2$), the cobordism from $U_n$ to $U_{n+1}$ that consists
of a pair-of-pants between the second component of $U_n$ and the second and the third
components of $U_{n+1}$, and of cylinders between all the other components.
The marked components of $U_n$ and $U_{n+1}$ are connected by a cylinder.}
\item{$\ILambda^a_n$ ($n \geq 3$), the cobordism from $U_n$ to $U_{n-1}$ that consists
of a pair-of-pants between the second and the third components of $U_n$ and the second
component of $U_{n-1}$, and of cylinders between all the other components.
The marked components of $U_n$ and $U_{n-1}$ are connected by a cylinder.}
\end{itemize}
The above cobordisms are represented in Figure~\ref{fig:redKhTQFT}.
\end{definition}

\begin{figure}
\resizebox{0.7\textwidth}{!}{
\input{redKhTQFT.tex}
}
\caption{The figure shows the cobordisms $\V^a_1$, $\Lambda^a_2$,
$\IX^a_{2,3}$, $\IV^a_2$, and $\ILambda^a_3$ in the category $\AbsCob$
and the maps they induce via the reduced Khovanov TQFT $\td\Kh$.
Note that the marked component gives a canonical identification
$\td\Kh(U_n) = V^{\otimes (n-1)}$.}
\label{fig:redKhTQFT}
\end{figure}
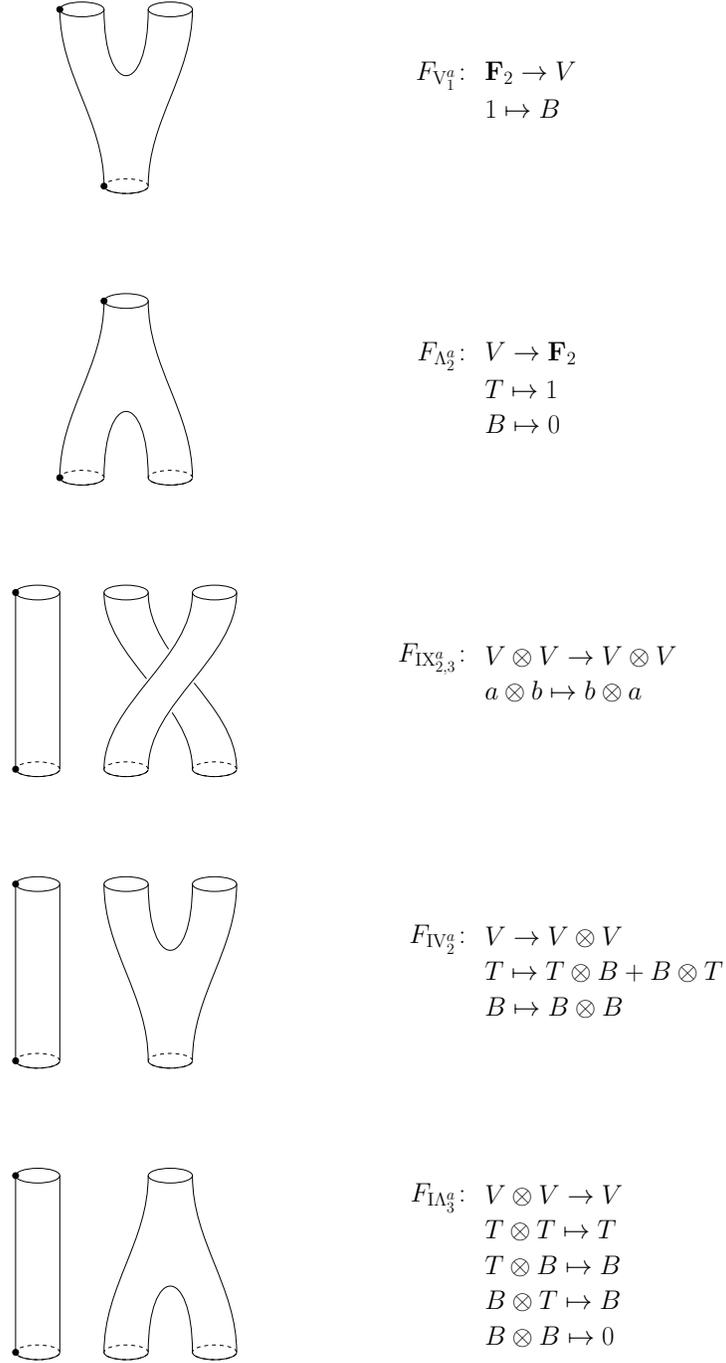

\begin{definition}
Let
\[
\Ob \colon \EmbCob \to \AbsCob,
\]
be the forgetful functor such that on the objects $\Ob(U_n, P_n) = U_n$,
and on the cobordisms it simply forgets all the decorations and the embedding.
\end{definition}

\begin{definition}
The reduced Khovanov TQFT is the functor
\[
\td\Kh \colon \AbsCob \to \Vect_{\F_2}
\]
such that on the objects $\td\Kh(U_n) = V^{\otimes (n-1)}$,
where each $V$ factor should be thought of as associated to
an unmarked component, and on the morphisms it is defined
as in Figure~\ref{fig:redKhTQFT}.
\end{definition}

The main result of this section is the following theorem.

\begin{theorem}
\label{thm:triangle}
The following triangle of functors is commutative:
\begin{center}
\begin{tikzpicture}[description/.style={fill=white,inner sep=2pt}]
\matrix (m) [matrix of math nodes, row sep=3em,
column sep=3em, text height=1.7ex, text depth=0.3ex]
{\EmbCob & & \Vect_{\F_2} \\
  & \AbsCob & \\};
	\path[->,font=\scriptsize]
		(m-1-1) edge node[above]{$\HFLh$} (m-1-3)
		(m-1-1) edge node[description]{$\Ob$} (m-2-2)
		(m-2-2) edge node[description]{$\td\Kh$} (m-1-3);
\end{tikzpicture}
\end{center}
\end{theorem}

\begin{proof}
By definition,
\[
\HFLh(U_n, P_n) = V^{\otimes (n-1)} = \td\Kh(U_n) = \td\Kh\left(\Ob\left(U_n, P_n\right)\right).
\]
On the morphisms, it suffices to check the commutativity of the triangle
for the generators. This is
achieved by comparing Figures~\ref{fig:HFLTQFT} and~\ref{fig:redKhTQFT}.
\end{proof}

As a consequence of Theorem~\ref{thm:triangle}, one can compute the reduced
Khovanov homology of~$L$ using the TQFT $\HFLh$. See
Section~\ref{sec:ss} for the definition of \emph{cube of resolutions}.

\begin{corollary}
\label{cor:TQFT}
The reduced Khovanov homology (with $\F_2$ coefficients) of a marked link
$L$ can be computed by applying the TQFT $\HFLh$ to a cube of resolutions
of $L$.
\end{corollary}

\begin{remark}
Let $\underline{\EmbCob}$ denote the category defined as $\EmbCob$,
but without any decorations on links and surfaces, and for
only for links in~$\R^3$ and surfaces in $\R^3 \times I$.
Jacobsson~\cite{Jacobsson} and Bar-Natan~\cite{bar2005khovanov} defined cobordism maps induced
on Khovanov homology by oriented link cobordisms in~$\R^3 \times I$.
When restricted to $\underline{\EmbCob}$, we recover the reduced Khovanov TQFT.
\end{remark}

Rasmussen~\cite{rasmussen2005knot} conjectured that there exists
a spectral sequence from the totally reduced Khovanov homology of a link
with one marked point on each link component
to its knot Floer homology. This motivates the following conjecture.

\begin{conjecture}
Let $(L_0,P_0)$ and $(L_1,P_1)$ be links with two decorations on each link component.
Given a (suitably) decorated link cobordism $\mc S$ from $L_0$ to $L_1$,
it functorially induces a morphism $\td\Kh(\mc S) \colon \td\Kh(L_0,P_0) \to \td\Kh(L_1,P_1)$,
where $\td\Kh(L_i,P_i)$ is the totally reduced Khovanov homology of the link~$L_i$
marked with the $z$ decorations in~$P_i$.
Furthermore, assuming Rasmussen's conjecture above, there is a morphism of
spectral sequences that induces $\td\Kh(\mc S)$ on the $E^2$ page and
$F_{\mc S} \colon \hat\HFK(L_0,P_0) \to \hat\HFK(L_1,P_1)$ on the $E^\infty$ page,
where we identify $\hat\HFK(L_i,P_i)$ with $\hat\HFL(L_i,P_i)$ via
the isomorphisms $\Phi_{L_i,P_i}$ defined in Section~\ref{sec:skein}.
\end{conjecture} 

%% file: HFLTQFT.tex
\begin{tikzpicture}[crossline/.style={preaction={draw=white, -, line width=5pt}}]

\def\u{0.6cm};
\def\h{6*\u};
\def\p{\h/3};
\def\v{4*\u};
\def\b{3*\u};
\def\c{4*\u};
\def\vshift{3.3*\v};
\def\hshift{16*\u};
\def\d{0.23570226039*\u};
\def\cclip{\d};
\def\smallradius{\u/15};
\def\markradius{2*\smallradius};
\def\w{0.4*\v};


\begin{scope}[shift={(\hshift,\u)}]
\draw (0,0) node[anchor=east] {\huge $F_{\V^e_1} \colon$};
\draw (0, \d/2) node[anchor=west] {\huge $\mathbf{F}_2 \to V$};
\draw (0, \d/2-\w) node[anchor=west] {\huge $1 \mapsto B$};
\end{scope}

\begin{scope}[shift={(\hshift,-\vshift+1.5*\u)}]
\draw (0,0) node[anchor=east] {\huge $F_{\Lambda^e_2} \colon$};
\draw (0, \d/2) node[anchor=west] {\huge $V \to \mathbf{F}_2$};
\draw (0, \d/2-\w) node[anchor=west] {\huge $T \mapsto 1$};
\draw (0, \d/2-2*\w) node[anchor=west] {\huge $B \mapsto 0$};
\end{scope}

\begin{scope}[shift={(\hshift,-2*\vshift+\u)}]
\draw (0,0) node[anchor=east] {\huge $F_{\IX^e_{2,3}} \colon$};
\draw (0, \d/2) node[anchor=west] {\huge $V \otimes V \to V \otimes V$};
\draw (0, \d/2-\w) node[anchor=west] {\huge $a \otimes b \mapsto b \otimes a$};
\end{scope}

\begin{scope}[shift={(\hshift,-3*\vshift+1.5*\u)}]
\draw (0,0) node[anchor=east] {\huge $F_{\IV^e_2} \colon$};
\draw (0, \d/2) node[anchor=west] {\huge $V \to V \otimes V$};
\draw (0, \d/2-\w) node[anchor=west] {\huge $T \mapsto T \otimes B + B \otimes T$};
\draw (0, \d/2-2*\w) node[anchor=west] {\huge $B \mapsto B \otimes B$};
\end{scope}

\begin{scope}[shift={(\hshift,-4*\vshift+3*\u)}]
\draw (0,0) node[anchor=east] {\huge $F_{\ILambda^e_3} \colon$};
\draw (0, \d/2) node[anchor=west] {\huge $V \otimes V \to V$};
\draw (0, \d/2-\w) node[anchor=west] {\huge $T \otimes T \mapsto T$};
\draw (0, \d/2-2*\w) node[anchor=west] {\huge $T \otimes B \mapsto B$};
\draw (0, \d/2-3*\w) node[anchor=west] {\huge $B \otimes T \mapsto B$};
\draw (0, \d/2-4*\w) node[anchor=west] {\huge $B \otimes B \mapsto 0$};
\end{scope}

\begin{scope}[shift={(0,0)}]

\begin{scope}[shift={(0,0-\v)}]
\def\halfradius{\h};
\draw (0,0) ellipse (\halfradius*1 and \halfradius/3);
\draw[white, fill=white] (0-1.2*\halfradius,0) rectangle (1.2*\halfradius,\halfradius/2);
\draw[dashed] (0,0) ellipse (\halfradius*1 and \halfradius/3);
\end{scope}

\begin{scope}[shift={(0,0-\v)}]
\def\halfradius{\u};
\draw (0,0) ellipse (\halfradius*1 and \halfradius/3);
\draw[white, fill=white] (0-1.2*\halfradius,0) rectangle (1.2*\halfradius,\halfradius/2);
\draw[dashed] (0,0) ellipse (\halfradius*1 and \halfradius/3);

\draw (-\halfradius,0) node[anchor=east] {$-$};
\draw (\halfradius,0) node[anchor=west] {$+$};
\end{scope}

\draw (-\p-\u,\v) .. controls (-\p-\u,\v-\b) and (-\u,-\v+\b) .. (-\u,-\v);
\draw (\p+\u,\v) .. controls (\p+\u,\v-\b) and (\u,-\v+\b) .. (\u,-\v);
\draw (-\p+\u,\v) .. controls (-\p+\u,\v-\c) and (\p-\u,\v-\c) .. (\p-\u,\v);

\draw[color=green] (-\p-3*\d,\v-\d) .. controls (-\p-3*\d,\v-\d-0.9*\b) and (-3*\d,-\v-\d+1.2*\b) .. (-3*\d,-\v-\d);
\draw[color=green] (\p+3*\d,\v-\d) .. controls (\p+3*\d,\v-\d-0.9*\b) and (3*\d,-\v-\d+1.2*\b) .. (3*\d,-\v-\d);
\draw[color=green] (-\p+3*\d,\v-\d) .. controls (-\p+3*\d,\v-\d-1.05*\c) and (\p-3*\d,\v-\d-1.05*\c) .. (\p-3*\d,\v-\d);

\begin{scope}[shift={(0,\v)}]
\def\halfradius{\h};
\draw[crossline] (0,0) ellipse (\halfradius*1 and \halfradius/3);
\end{scope}

\begin{scope}[shift={(\p,\v)}]
\def\halfradius{\u};
\draw (0,0) ellipse (\halfradius*1 and \halfradius/3);
\draw (-\halfradius,0) node[anchor=east] {$-$};
\draw (\halfradius,0) node[anchor=west] {$+$};
\end{scope}
\begin{scope}[shift={(-\p,\v)}]
\def\halfradius{\u};
\draw (0,0) ellipse (\halfradius*1 and \halfradius/3);
\draw (-\halfradius,0) node[anchor=east] {$-$};
\draw (\halfradius,0) node[anchor=west] {$+$};
\end{scope}

\begin{scope}[shift={(0,\v)}]
\draw[color=red,fill=red] (-\p, \u/3) circle (\smallradius);
\draw[color=red,fill=red] (-\p, -\u/3) circle (\smallradius);
\draw[color=black,fill=black] (-\p-\u,0) circle (\markradius);
\draw[color=red,fill=red] (\p, \u/3) circle (\smallradius);
\draw[color=red,fill=red] (\p, -\u/3) circle (\smallradius);
\end{scope}
\begin{scope}[shift={(0,-\v)}]
\draw[color=red,fill=red] (0, \u/3) circle (\smallradius);
\draw[color=red,fill=red] (0, -\u/3) circle (\smallradius);
\draw[color=black,fill=black] (-\u,0) circle (\markradius);
\end{scope}

\draw (-\h, \v) -- (-\h, -\v);
\draw (\h, \v) -- (\h, -\v);

\end{scope}

\begin{scope}[shift={(0,-\vshift)}]

\begin{scope}[shift={(0,0-\v)}]
\def\halfradius{\h};
\draw (0,0) ellipse (\halfradius*1 and \halfradius/3);
\draw[white, fill=white] (0-1.2*\halfradius,0) rectangle (1.2*\halfradius,\halfradius/2);
\draw[dashed] (0,0) ellipse (\halfradius*1 and \halfradius/3);
\end{scope}

\begin{scope}[shift={(\p,0-\v)}]
\def\halfradius{\u};
\draw (0,0) ellipse (\halfradius*1 and \halfradius/3);
\draw[white, fill=white] (0-1.2*\halfradius,0) rectangle (1.2*\halfradius,\halfradius/2);
\draw[dashed] (0,0) ellipse (\halfradius*1 and \halfradius/3);
\draw (-\halfradius,0) node[anchor=east] {$-$};
\draw (\halfradius,0) node[anchor=west] {$+$};
\end{scope}
\begin{scope}[shift={(-\p,0-\v)}]
\def\halfradius{\u};
\draw (0,0) ellipse (\halfradius*1 and \halfradius/3);
\draw[white, fill=white] (0-1.2*\halfradius,0) rectangle (1.2*\halfradius,\halfradius/2);
\draw[dashed] (0,0) ellipse (\halfradius*1 and \halfradius/3);
\draw[color=black,fill=black] (-\u,0) circle (\markradius);
\draw (-\halfradius,0) node[anchor=east] {$-$};
\draw (\halfradius,0) node[anchor=west] {$+$};
\end{scope}

\begin{scope}[yscale=-1]
\draw (-\p-\u,\v) .. controls (-\p-\u,\v-\b) and (-\u,-\v+\b) .. (-\u,-\v);
\draw (\p+\u,\v) .. controls (\p+\u,\v-\b) and (\u,-\v+\b) .. (\u,-\v);
\draw (-\p+\u,\v) .. controls (-\p+\u,\v-\c) and (\p-\u,\v-\c) .. (\p-\u,\v);

\begin{scope}[yshift=2*\d]
\draw[color=green] (-\p-3*\d,\v-\d) .. controls (-\p-3*\d,\v-\d-1.1*\b) and (-3*\d,-\v-\d+1*\b) .. (-3*\d,-\v-\d);
\draw[color=green] (\p+3*\d,\v-\d) .. controls (\p+3*\d,\v-\d-1.1*\b) and (3*\d,-\v-\d+1*\b) .. (3*\d,-\v-\d);
\draw[color=green] (-\p+3*\d,\v-\d) .. controls (-\p+3*\d,\v-\d-1.25*\c) and (\p-3*\d,\v-\d-1.25*\c) .. (\p-3*\d,\v-\d);
\end{scope}
\end{scope}

\begin{scope}[shift={(0,\v)}]
\def\halfradius{\h};
\draw[crossline] (0,0) ellipse (\halfradius*1 and \halfradius/3);
\end{scope}

\begin{scope}[shift={(0,\v)}]
\def\halfradius{\u};
\draw (0,0) ellipse (\halfradius*1 and \halfradius/3);
\draw[color=black,fill=black] (-\u,0) circle (\markradius);
\draw (-\halfradius,0) node[anchor=east] {$-$};
\draw (\halfradius,0) node[anchor=west] {$+$};
\end{scope}

\begin{scope}[shift={(0,-\v)}]
\draw[color=red,fill=red] (-\p, \u/3) circle (\smallradius);
\draw[color=red,fill=red] (-\p, -\u/3) circle (\smallradius);
\draw[color=red,fill=red] (\p, \u/3) circle (\smallradius);
\draw[color=red,fill=red] (\p, -\u/3) circle (\smallradius);
\end{scope}
\begin{scope}[shift={(0,\v)}]
\draw[color=red,fill=red] (0, \u/3) circle (\smallradius);
\draw[color=red,fill=red] (0, -\u/3) circle (\smallradius);
\end{scope}

\draw (-\h, \v) -- (-\h, -\v);
\draw (\h, \v) -- (\h, -\v);

\end{scope}

\begin{scope}[shift={(0,-2*\vshift)}]

\begin{scope}[shift={(0,0-\v)}]
\def\halfradius{\h};
\draw (0,0) ellipse (\halfradius*1 and \halfradius/3);
\draw[white, fill=white] (0-1.2*\halfradius,0) rectangle (1.2*\halfradius,\halfradius/2);
\draw[dashed] (0,0) ellipse (\halfradius*1 and \halfradius/3);
\end{scope}

\begin{scope}[shift={(-2*\p,0-\v)}]
\def\halfradius{\u};
\draw (0,0) ellipse (\halfradius*1 and \halfradius/3);
\draw[white, fill=white] (0-1.2*\halfradius,0) rectangle (1.2*\halfradius,\halfradius/2);
\draw[dashed] (0,0) ellipse (\halfradius*1 and \halfradius/3);
\draw[color=black,fill=black] (-\u,0) circle (\markradius);
\draw (-\halfradius,0) node[anchor=east] {$-$};
\draw (\halfradius,0) node[anchor=west] {$+$};
\draw[color=red, fill=red] (0,-\halfradius/3) circle (\smallradius);
\draw[color=red, fill=red] (0,\halfradius/3) circle (\smallradius);
\end{scope}
\begin{scope}[shift={(0,0-\v)}]
\def\halfradius{\u};
\draw (0,0) ellipse (\halfradius*1 and \halfradius/3);
\draw[white, fill=white] (0-1.2*\halfradius,0) rectangle (1.2*\halfradius,\halfradius/2);
\draw[dashed] (0,0) ellipse (\halfradius*1 and \halfradius/3);
\draw (-\halfradius,0) node[anchor=east] {$-$};
\draw (\halfradius,0) node[anchor=west] {$+$};
\draw[color=red, fill=red] (0,-\halfradius/3) circle (\smallradius);
\draw[color=red, fill=red] (0,\halfradius/3) circle (\smallradius);
\end{scope}
\begin{scope}[shift={(2*\p,0-\v)}]
\def\halfradius{\u};
\draw (0,0) ellipse (\halfradius*1 and \halfradius/3);
\draw[white, fill=white] (0-1.2*\halfradius,0) rectangle (1.2*\halfradius,\halfradius/2);
\draw[dashed] (0,0) ellipse (\halfradius*1 and \halfradius/3);
\draw (-\halfradius,0) node[anchor=east] {$-$};
\draw (\halfradius,0) node[anchor=west] {$+$};
\draw[color=red, fill=red] (0,-\halfradius/3) circle (\smallradius);
\draw[color=red, fill=red] (0,\halfradius/3) circle (\smallradius);
\end{scope}

\draw (-2*\p-\u,\v) -- (-2*\p-\u,-\v);
\draw (-2*\p+\u,\v) -- (-2*\p+\u,-\v);
\draw[color=green] (-2*\p-3*\d, \v-\d) -- (-2*\p-3*\d,-\v-\d);
\draw[color=green] (-2*\p+3*\d, \v-\d) -- (-2*\p+3*\d,-\v-\d);

\begin{scope}[]
\clip (-\u-\cclip, \v+\cclip) -- (-\u-4*\cclip, -\v+\cclip) -- (2*\p-0*\cclip,\v+\cclip) --cycle;
\begin{scope}[xscale=-1, xshift=-2*\p]
\draw (2*\p-\u,\v) .. controls (2*\p-\u,\v-\b) and (-\u,-\v+\b) .. (-\u,-\v);
\draw (2*\p+\u,\v) .. controls (2*\p+\u,\v-\b) and (\u,-\v+\b) .. (\u,-\v);
\draw[color=green] (2*\p-3*\d,\v-\d) .. controls (2*\p-3*\d,\v-\d-0.9*\b) and (-3*\d,-\v-\d+1*\b) .. (-3*\d,-\v-\d);
\draw[color=green] (2*\p+3*\d,\v-\d) .. controls (2*\p+3*\d,\v-\d-0.9*\b) and (3*\d,-\v-\d+1.2*\b) .. (3*\d,-\v-\d);
\end{scope}
\end{scope}

\begin{scope}
\clip (2*\p+\u+\cclip, -\v-\cclip) -- (2*\p+\u+4*\cclip, \v-\cclip) -- (0,-\v-\cclip) --cycle;
\begin{scope}[xscale=-1, xshift=-2*\p]
\draw (2*\p-\u,\v) .. controls (2*\p-\u,\v-\b) and (-\u,-\v+\b) .. (-\u,-\v);
\draw (2*\p+\u,\v) .. controls (2*\p+\u,\v-\b) and (\u,-\v+\b) .. (\u,-\v);
\draw[color=green] (2*\p-3*\d,\v-\d) .. controls (2*\p-3*\d,\v-\d-0.9*\b) and (-3*\d,-\v-\d+1*\b) .. (-3*\d,-\v-\d);
\draw[color=green] (2*\p+3*\d,\v-\d) .. controls (2*\p+3*\d,\v-\d-0.9*\b) and (3*\d,-\v-\d+1.2*\b) .. (3*\d,-\v-\d);
\end{scope}
\end{scope}

\draw (2*\p-\u,\v) .. controls (2*\p-\u,\v-\b) and (-\u,-\v+\b) .. (-\u,-\v);
\draw (2*\p+\u,\v) .. controls (2*\p+\u,\v-\b) and (\u,-\v+\b) .. (\u,-\v);
\draw[color=green] (2*\p-3*\d,\v-\d) .. controls (2*\p-3*\d,\v-\d-0.9*\b) and (-3*\d,-\v-\d+1*\b) .. (-3*\d,-\v-\d);
\draw[color=green] (2*\p+3*\d,\v-\d) .. controls (2*\p+3*\d,\v-\d-0.9*\b) and (3*\d,-\v-\d+1.2*\b) .. (3*\d,-\v-\d);

\begin{scope}[shift={(0,\v)}]
\def\halfradius{\h};
\draw[crossline] (0,0) ellipse (\halfradius*1 and \halfradius/3);
\end{scope}

\begin{scope}[shift={(-2*\p,\v)}]
\def\halfradius{\u};
\draw (0,0) ellipse (\halfradius*1 and \halfradius/3);
\draw[color=black,fill=black] (-\u,0) circle (\markradius);
\draw (-\halfradius,0) node[anchor=east] {$-$};
\draw (\halfradius,0) node[anchor=west] {$+$};
\draw[color=red, fill=red] (0,-\halfradius/3) circle (\smallradius);
\draw[color=red, fill=red] (0,\halfradius/3) circle (\smallradius);
\end{scope}
\begin{scope}[shift={(0,\v)}]
\def\halfradius{\u};
\draw (0,0) ellipse (\halfradius*1 and \halfradius/3);
\draw (-\halfradius,0) node[anchor=east] {$-$};
\draw (\halfradius,0) node[anchor=west] {$+$};
\draw[color=red, fill=red] (0,-\halfradius/3) circle (\smallradius);
\draw[color=red, fill=red] (0,\halfradius/3) circle (\smallradius);
\end{scope}
\begin{scope}[shift={(2*\p,\v)}]
\def\halfradius{\u};
\draw (0,0) ellipse (\halfradius*1 and \halfradius/3);
\draw (-\halfradius,0) node[anchor=east] {$-$};
\draw (\halfradius,0) node[anchor=west] {$+$};
\draw[color=red, fill=red] (0,-\halfradius/3) circle (\smallradius);
\draw[color=red, fill=red] (0,\halfradius/3) circle (\smallradius);
\end{scope}

\draw (-\h, \v) -- (-\h, -\v);
\draw (\h, \v) -- (\h, -\v);

\end{scope}

\begin{scope}[shift={(0,-3*\vshift)}]

\begin{scope}[shift={(0,0-\v)}]
\def\halfradius{\h};
\draw (0,0) ellipse (\halfradius*1 and \halfradius/3);
\draw[white, fill=white] (0-1.2*\halfradius,0) rectangle (1.2*\halfradius,\halfradius/2);
\draw[dashed] (0,0) ellipse (\halfradius*1 and \halfradius/3);
\end{scope}

\begin{scope}[xshift=\p]

\begin{scope}[shift={(0,0-\v)}]
\def\halfradius{\u};
\draw (0,0) ellipse (\halfradius*1 and \halfradius/3);
\draw[white, fill=white] (0-1.2*\halfradius,0) rectangle (1.2*\halfradius,\halfradius/2);
\draw[dashed] (0,0) ellipse (\halfradius*1 and \halfradius/3);

\draw (-\halfradius,0) node[anchor=east] {$-$};
\draw (\halfradius,0) node[anchor=west] {$+$};
\end{scope}

\draw (-\p-\u,\v) .. controls (-\p-\u,\v-\b) and (-\u,-\v+\b) .. (-\u,-\v);
\draw (\p+\u,\v) .. controls (\p+\u,\v-\b) and (\u,-\v+\b) .. (\u,-\v);
\draw (-\p+\u,\v) .. controls (-\p+\u,\v-\c) and (\p-\u,\v-\c) .. (\p-\u,\v);

\draw[color=green] (-\p-3*\d,\v-\d) .. controls (-\p-3*\d,\v-\d-0.9*\b) and (-3*\d,-\v-\d+1.2*\b) .. (-3*\d,-\v-\d);
\draw[color=green] (\p+3*\d,\v-\d) .. controls (\p+3*\d,\v-\d-0.9*\b) and (3*\d,-\v-\d+1.2*\b) .. (3*\d,-\v-\d);
\draw[color=green] (-\p+3*\d,\v-\d) .. controls (-\p+3*\d,\v-\d-1.05*\c) and (\p-3*\d,\v-\d-1.05*\c) .. (\p-3*\d,\v-\d);

\begin{scope}[shift={(\p,\v)}]
\def\halfradius{\u};
\draw (0,0) ellipse (\halfradius*1 and \halfradius/3);
\draw (-\halfradius,0) node[anchor=east] {$-$};
\draw (\halfradius,0) node[anchor=west] {$+$};
\end{scope}
\begin{scope}[shift={(-\p,\v)}]
\def\halfradius{\u};
\draw (0,0) ellipse (\halfradius*1 and \halfradius/3);
\draw (-\halfradius,0) node[anchor=east] {$-$};
\draw (\halfradius,0) node[anchor=west] {$+$};
\end{scope}

\begin{scope}[shift={(0,\v)}]
\draw[color=red,fill=red] (-\p, \u/3) circle (\smallradius);
\draw[color=red,fill=red] (-\p, -\u/3) circle (\smallradius);
\draw[color=red,fill=red] (\p, \u/3) circle (\smallradius);
\draw[color=red,fill=red] (\p, -\u/3) circle (\smallradius);
\end{scope}
\begin{scope}[shift={(0,-\v)}]
\draw[color=red,fill=red] (0, \u/3) circle (\smallradius);
\draw[color=red,fill=red] (0, -\u/3) circle (\smallradius);
\end{scope}

\end{scope}

\begin{scope}[shift={(-2*\p,\v)}]
\def\halfradius{\u};
\draw (0,0) ellipse (\halfradius*1 and \halfradius/3);
\draw[color=black,fill=black] (-\u,0) circle (\markradius);
\draw (-\halfradius,0) node[anchor=east] {$-$};
\draw (\halfradius,0) node[anchor=west] {$+$};
\draw[color=red, fill=red] (0,-\halfradius/3) circle (\smallradius);
\draw[color=red, fill=red] (0,\halfradius/3) circle (\smallradius);
\end{scope}
\begin{scope}[shift={(-2*\p,0-\v)}]
\def\halfradius{\u};
\draw (0,0) ellipse (\halfradius*1 and \halfradius/3);
\draw[white, fill=white] (0-1.2*\halfradius,0) rectangle (1.2*\halfradius,\halfradius/2);
\draw[dashed] (0,0) ellipse (\halfradius*1 and \halfradius/3);
\draw[color=black,fill=black] (-\u,0) circle (\markradius);
\draw (-\halfradius,0) node[anchor=east] {$-$};
\draw (\halfradius,0) node[anchor=west] {$+$};
\draw[color=red, fill=red] (0,-\halfradius/3) circle (\smallradius);
\draw[color=red, fill=red] (0,\halfradius/3) circle (\smallradius);
\end{scope}
\draw (-2*\p-\u,\v) -- (-2*\p-\u,-\v);
\draw (-2*\p+\u,\v) -- (-2*\p+\u,-\v);
\draw[color=green] (-2*\p-3*\d, \v-\d) -- (-2*\p-3*\d,-\v-\d);
\draw[color=green] (-2*\p+3*\d, \v-\d) -- (-2*\p+3*\d,-\v-\d);

\begin{scope}[shift={(0,\v)}]
\def\halfradius{\h};
\draw[crossline] (0,0) ellipse (\halfradius*1 and \halfradius/3);
\end{scope}

\draw (-\h, \v) -- (-\h, -\v);
\draw (\h, \v) -- (\h, -\v);

\end{scope}

\begin{scope}[shift={(0,-4*\vshift)}]

\begin{scope}[shift={(0,0-\v)}]
\def\halfradius{\h};
\draw (0,0) ellipse (\halfradius*1 and \halfradius/3);
\draw[white, fill=white] (0-1.2*\halfradius,0) rectangle (1.2*\halfradius,\halfradius/2);
\draw[dashed] (0,0) ellipse (\halfradius*1 and \halfradius/3);
\end{scope}

\begin{scope}[xshift=\p]

\begin{scope}[shift={(\p,0-\v)}]
\def\halfradius{\u};
\draw (0,0) ellipse (\halfradius*1 and \halfradius/3);
\draw[white, fill=white] (0-1.2*\halfradius,0) rectangle (1.2*\halfradius,\halfradius/2);
\draw[dashed] (0,0) ellipse (\halfradius*1 and \halfradius/3);
\draw (-\halfradius,0) node[anchor=east] {$-$};
\draw (\halfradius,0) node[anchor=west] {$+$};
\end{scope}
\begin{scope}[shift={(-\p,0-\v)}]
\def\halfradius{\u};
\draw (0,0) ellipse (\halfradius*1 and \halfradius/3);
\draw[white, fill=white] (0-1.2*\halfradius,0) rectangle (1.2*\halfradius,\halfradius/2);
\draw[dashed] (0,0) ellipse (\halfradius*1 and \halfradius/3);
\draw (-\halfradius,0) node[anchor=east] {$-$};
\draw (\halfradius,0) node[anchor=west] {$+$};
\end{scope}

\begin{scope}[yscale=-1]
\draw (-\p-\u,\v) .. controls (-\p-\u,\v-\b) and (-\u,-\v+\b) .. (-\u,-\v);
\draw (\p+\u,\v) .. controls (\p+\u,\v-\b) and (\u,-\v+\b) .. (\u,-\v);
\draw (-\p+\u,\v) .. controls (-\p+\u,\v-\c) and (\p-\u,\v-\c) .. (\p-\u,\v);

\begin{scope}[yshift=2*\d]
\draw[color=green] (-\p-3*\d,\v-\d) .. controls (-\p-3*\d,\v-\d-1.1*\b) and (-3*\d,-\v-\d+1*\b) .. (-3*\d,-\v-\d);
\draw[color=green] (\p+3*\d,\v-\d) .. controls (\p+3*\d,\v-\d-1.1*\b) and (3*\d,-\v-\d+1*\b) .. (3*\d,-\v-\d);
\draw[color=green] (-\p+3*\d,\v-\d) .. controls (-\p+3*\d,\v-\d-1.25*\c) and (\p-3*\d,\v-\d-1.25*\c) .. (\p-3*\d,\v-\d);
\end{scope}
\end{scope}

\begin{scope}[shift={(0,\v)}]
\def\halfradius{\u};
\draw (0,0) ellipse (\halfradius*1 and \halfradius/3);
\draw (-\halfradius,0) node[anchor=east] {$-$};
\draw (\halfradius,0) node[anchor=west] {$+$};
\end{scope}

\begin{scope}[shift={(0,-\v)}]
\draw[color=red,fill=red] (-\p, \u/3) circle (\smallradius);
\draw[color=red,fill=red] (-\p, -\u/3) circle (\smallradius);
\draw[color=red,fill=red] (\p, \u/3) circle (\smallradius);
\draw[color=red,fill=red] (\p, -\u/3) circle (\smallradius);
\end{scope}
\begin{scope}[shift={(0,\v)}]
\draw[color=red,fill=red] (0, \u/3) circle (\smallradius);
\draw[color=red,fill=red] (0, -\u/3) circle (\smallradius);
\end{scope}

\end{scope}

\begin{scope}[shift={(-2*\p,\v)}]
\def\halfradius{\u};
\draw (0,0) ellipse (\halfradius*1 and \halfradius/3);
\draw[color=black,fill=black] (-\u,0) circle (\markradius);
\draw (-\halfradius,0) node[anchor=east] {$-$};
\draw (\halfradius,0) node[anchor=west] {$+$};
\draw[color=red, fill=red] (0,-\halfradius/3) circle (\smallradius);
\draw[color=red, fill=red] (0,\halfradius/3) circle (\smallradius);
\end{scope}
\begin{scope}[shift={(-2*\p,0-\v)}]
\def\halfradius{\u};
\draw (0,0) ellipse (\halfradius*1 and \halfradius/3);
\draw[white, fill=white] (0-1.2*\halfradius,0) rectangle (1.2*\halfradius,\halfradius/2);
\draw[dashed] (0,0) ellipse (\halfradius*1 and \halfradius/3);
\draw[color=black,fill=black] (-\u,0) circle (\markradius);
\draw (-\halfradius,0) node[anchor=east] {$-$};
\draw (\halfradius,0) node[anchor=west] {$+$};
\draw[color=red, fill=red] (0,-\halfradius/3) circle (\smallradius);
\draw[color=red, fill=red] (0,\halfradius/3) circle (\smallradius);
\end{scope}
\draw (-2*\p-\u,\v) -- (-2*\p-\u,-\v);
\draw (-2*\p+\u,\v) -- (-2*\p+\u,-\v);
\draw[color=green] (-2*\p-3*\d, \v-\d) -- (-2*\p-3*\d,-\v-\d);
\draw[color=green] (-2*\p+3*\d, \v-\d) -- (-2*\p+3*\d,-\v-\d);

\begin{scope}[shift={(0,\v)}]
\def\halfradius{\h};
\draw[crossline] (0,0) ellipse (\halfradius*1 and \halfradius/3);
\end{scope}

\draw (-\h, \v) -- (-\h, -\v);
\draw (\h, \v) -- (\h, -\v);

\end{scope}

\end{tikzpicture}

%% file: redKhTQFT.tex
\begin{tikzpicture}[crossline/.style={preaction={draw=white, -, line width=5pt}}]

\def\u{0.6cm};
\def\h{6*\u};
\def\p{\h/3};
\def\v{4*\u};
\def\b{3*\u};
\def\c{4*\u};
\def\vshift{3.3*\v};
\def\hshift{16*\u};
\def\d{0.23570226039*\u};
\def\cclip{\d};
\def\smallradius{\u/15};
\def\markradius{2*\smallradius};
\def\w{0.4*\v};


\begin{scope}[shift={(\hshift,\u)}]
\draw (0,0) node[anchor=east] {\huge $F_{\V^a_1} \colon$};
\draw (0, \d/2) node[anchor=west] {\huge $\mathbf{F}_2 \to V$};
\draw (0, \d/2-\w) node[anchor=west] {\huge $1 \mapsto B$};
\end{scope}

\begin{scope}[shift={(\hshift,-\vshift+1.5*\u)}]
\draw (0,0) node[anchor=east] {\huge $F_{\Lambda^a_2} \colon$};
\draw (0, \d/2) node[anchor=west] {\huge $V \to \mathbf{F}_2$};
\draw (0, \d/2-\w) node[anchor=west] {\huge $T \mapsto 1$};
\draw (0, \d/2-2*\w) node[anchor=west] {\huge $B \mapsto 0$};
\end{scope}

\begin{scope}[shift={(\hshift,-2*\vshift+\u)}]
\draw (0,0) node[anchor=east] {\huge $F_{\IX^a_{2,3}} \colon$};
\draw (0, \d/2) node[anchor=west] {\huge $V \otimes V \to V \otimes V$};
\draw (0, \d/2-\w) node[anchor=west] {\huge $a \otimes b \mapsto b \otimes a$};
\end{scope}

\begin{scope}[shift={(\hshift,-3*\vshift+1.5*\u)}]
\draw (0,0) node[anchor=east] {\huge $F_{\IV^a_2} \colon$};
\draw (0, \d/2) node[anchor=west] {\huge $V \to V \otimes V$};
\draw (0, \d/2-\w) node[anchor=west] {\huge $T \mapsto T \otimes B + B \otimes T$};
\draw (0, \d/2-2*\w) node[anchor=west] {\huge $B \mapsto B \otimes B$};
\end{scope}

\begin{scope}[shift={(\hshift,-4*\vshift+3*\u)}]
\draw (0,0) node[anchor=east] {\huge $F_{\ILambda^a_3} \colon$};
\draw (0, \d/2) node[anchor=west] {\huge $V \otimes V \to V$};
\draw (0, \d/2-\w) node[anchor=west] {\huge $T \otimes T \mapsto T$};
\draw (0, \d/2-2*\w) node[anchor=west] {\huge $T \otimes B \mapsto B$};
\draw (0, \d/2-3*\w) node[anchor=west] {\huge $B \otimes T \mapsto B$};
\draw (0, \d/2-4*\w) node[anchor=west] {\huge $B \otimes B \mapsto 0$};
\end{scope}

\begin{scope}[shift={(0,0)}]

\draw (-\h,0);

\begin{scope}[shift={(0,0-\v)}]
\def\halfradius{\u};
\draw (0,0) ellipse (\halfradius*1 and \halfradius/3);
\draw[white, fill=white] (0-1.2*\halfradius,0) rectangle (1.2*\halfradius,\halfradius/2);
\draw[dashed] (0,0) ellipse (\halfradius*1 and \halfradius/3);
\end{scope}

\draw (-\p-\u,\v) .. controls (-\p-\u,\v-\b) and (-\u,-\v+\b) .. (-\u,-\v);
\draw (\p+\u,\v) .. controls (\p+\u,\v-\b) and (\u,-\v+\b) .. (\u,-\v);
\draw (-\p+\u,\v) .. controls (-\p+\u,\v-\c) and (\p-\u,\v-\c) .. (\p-\u,\v);

\begin{scope}[shift={(\p,\v)}]
\def\halfradius{\u};
\draw (0,0) ellipse (\halfradius*1 and \halfradius/3);
\end{scope}
\begin{scope}[shift={(-\p,\v)}]
\def\halfradius{\u};
\draw (0,0) ellipse (\halfradius*1 and \halfradius/3);
\end{scope}

\begin{scope}[shift={(0,\v)}]
\draw[color=black,fill=black] (-\p-\u,0) circle (\markradius);
\end{scope}
\begin{scope}[shift={(0,-\v)}]
\draw[color=black,fill=black] (-\u,0) circle (\markradius);
\end{scope}

\end{scope}

\begin{scope}[shift={(0,-\vshift)}]

\begin{scope}[shift={(\p,0-\v)}]
\def\halfradius{\u};
\draw (0,0) ellipse (\halfradius*1 and \halfradius/3);
\draw[white, fill=white] (0-1.2*\halfradius,0) rectangle (1.2*\halfradius,\halfradius/2);
\draw[dashed] (0,0) ellipse (\halfradius*1 and \halfradius/3);
\end{scope}
\begin{scope}[shift={(-\p,0-\v)}]
\def\halfradius{\u};
\draw (0,0) ellipse (\halfradius*1 and \halfradius/3);
\draw[white, fill=white] (0-1.2*\halfradius,0) rectangle (1.2*\halfradius,\halfradius/2);
\draw[dashed] (0,0) ellipse (\halfradius*1 and \halfradius/3);
\draw[color=black,fill=black] (-\u,0) circle (\markradius);
\end{scope}

\begin{scope}[yscale=-1]
\draw (-\p-\u,\v) .. controls (-\p-\u,\v-\b) and (-\u,-\v+\b) .. (-\u,-\v);
\draw (\p+\u,\v) .. controls (\p+\u,\v-\b) and (\u,-\v+\b) .. (\u,-\v);
\draw (-\p+\u,\v) .. controls (-\p+\u,\v-\c) and (\p-\u,\v-\c) .. (\p-\u,\v);
\end{scope}

\begin{scope}[shift={(0,\v)}]
\def\halfradius{\u};
\draw (0,0) ellipse (\halfradius*1 and \halfradius/3);
\draw[color=black,fill=black] (-\u,0) circle (\markradius);
\end{scope}

\end{scope}

\begin{scope}[shift={(0,-2*\vshift)}]

\begin{scope}[shift={(-2*\p,0-\v)}]
\def\halfradius{\u};
\draw (0,0) ellipse (\halfradius*1 and \halfradius/3);
\draw[white, fill=white] (0-1.2*\halfradius,0) rectangle (1.2*\halfradius,\halfradius/2);
\draw[dashed] (0,0) ellipse (\halfradius*1 and \halfradius/3);
\draw[color=black,fill=black] (-\u,0) circle (\markradius);
\end{scope}
\begin{scope}[shift={(0,0-\v)}]
\def\halfradius{\u};
\draw (0,0) ellipse (\halfradius*1 and \halfradius/3);
\draw[white, fill=white] (0-1.2*\halfradius,0) rectangle (1.2*\halfradius,\halfradius/2);
\draw[dashed] (0,0) ellipse (\halfradius*1 and \halfradius/3);
\end{scope}
\begin{scope}[shift={(2*\p,0-\v)}]
\def\halfradius{\u};
\draw (0,0) ellipse (\halfradius*1 and \halfradius/3);
\draw[white, fill=white] (0-1.2*\halfradius,0) rectangle (1.2*\halfradius,\halfradius/2);
\draw[dashed] (0,0) ellipse (\halfradius*1 and \halfradius/3);
\end{scope}

\draw (-2*\p-\u,\v) -- (-2*\p-\u,-\v);
\draw (-2*\p+\u,\v) -- (-2*\p+\u,-\v);

\begin{scope}[]
\clip (-\u-\cclip, \v+\cclip) -- (-\u-4*\cclip, -\v+\cclip) -- (2*\p-0*\cclip,\v+\cclip) --cycle;
\begin{scope}[xscale=-1, xshift=-2*\p]
\draw (2*\p-\u,\v) .. controls (2*\p-\u,\v-\b) and (-\u,-\v+\b) .. (-\u,-\v);
\draw (2*\p+\u,\v) .. controls (2*\p+\u,\v-\b) and (\u,-\v+\b) .. (\u,-\v);
\end{scope}
\end{scope}

\begin{scope}
\clip (2*\p+\u+\cclip, -\v-\cclip) -- (2*\p+\u+4*\cclip, \v-\cclip) -- (0,-\v-\cclip) --cycle;
\begin{scope}[xscale=-1, xshift=-2*\p]
\draw (2*\p-\u,\v) .. controls (2*\p-\u,\v-\b) and (-\u,-\v+\b) .. (-\u,-\v);
\draw (2*\p+\u,\v) .. controls (2*\p+\u,\v-\b) and (\u,-\v+\b) .. (\u,-\v);
\end{scope}
\end{scope}

\draw (2*\p-\u,\v) .. controls (2*\p-\u,\v-\b) and (-\u,-\v+\b) .. (-\u,-\v);
\draw (2*\p+\u,\v) .. controls (2*\p+\u,\v-\b) and (\u,-\v+\b) .. (\u,-\v);

\begin{scope}[shift={(-2*\p,\v)}]
\def\halfradius{\u};
\draw (0,0) ellipse (\halfradius*1 and \halfradius/3);
\draw[color=black,fill=black] (-\u,0) circle (\markradius);
\end{scope}
\begin{scope}[shift={(0,\v)}]
\def\halfradius{\u};
\draw (0,0) ellipse (\halfradius*1 and \halfradius/3);
\end{scope}
\begin{scope}[shift={(2*\p,\v)}]
\def\halfradius{\u};
\draw (0,0) ellipse (\halfradius*1 and \halfradius/3);
\end{scope}

\end{scope}

\begin{scope}[shift={(0,-3*\vshift)}]

\begin{scope}[xshift=\p]

\begin{scope}[shift={(0,0-\v)}]
\def\halfradius{\u};
\draw (0,0) ellipse (\halfradius*1 and \halfradius/3);
\draw[white, fill=white] (0-1.2*\halfradius,0) rectangle (1.2*\halfradius,\halfradius/2);
\draw[dashed] (0,0) ellipse (\halfradius*1 and \halfradius/3);
\end{scope}

\draw (-\p-\u,\v) .. controls (-\p-\u,\v-\b) and (-\u,-\v+\b) .. (-\u,-\v);
\draw (\p+\u,\v) .. controls (\p+\u,\v-\b) and (\u,-\v+\b) .. (\u,-\v);
\draw (-\p+\u,\v) .. controls (-\p+\u,\v-\c) and (\p-\u,\v-\c) .. (\p-\u,\v);

\begin{scope}[shift={(\p,\v)}]
\def\halfradius{\u};
\draw (0,0) ellipse (\halfradius*1 and \halfradius/3);
\end{scope}
\begin{scope}[shift={(-\p,\v)}]
\def\halfradius{\u};
\draw (0,0) ellipse (\halfradius*1 and \halfradius/3);
\end{scope}

\end{scope}

\begin{scope}[shift={(-2*\p,\v)}]
\def\halfradius{\u};
\draw (0,0) ellipse (\halfradius*1 and \halfradius/3);
\draw[color=black,fill=black] (-\u,0) circle (\markradius);
\end{scope}
\begin{scope}[shift={(-2*\p,0-\v)}]
\def\halfradius{\u};
\draw (0,0) ellipse (\halfradius*1 and \halfradius/3);
\draw[white, fill=white] (0-1.2*\halfradius,0) rectangle (1.2*\halfradius,\halfradius/2);
\draw[dashed] (0,0) ellipse (\halfradius*1 and \halfradius/3);
\draw[color=black,fill=black] (-\u,0) circle (\markradius);
\end{scope}
\draw (-2*\p-\u,\v) -- (-2*\p-\u,-\v);
\draw (-2*\p+\u,\v) -- (-2*\p+\u,-\v);

\end{scope}

\begin{scope}[shift={(0,-4*\vshift)}]

\begin{scope}[xshift=\p]

\begin{scope}[shift={(\p,0-\v)}]
\def\halfradius{\u};
\draw (0,0) ellipse (\halfradius*1 and \halfradius/3);
\draw[white, fill=white] (0-1.2*\halfradius,0) rectangle (1.2*\halfradius,\halfradius/2);
\draw[dashed] (0,0) ellipse (\halfradius*1 and \halfradius/3);
\end{scope}
\begin{scope}[shift={(-\p,0-\v)}]
\def\halfradius{\u};
\draw (0,0) ellipse (\halfradius*1 and \halfradius/3);
\draw[white, fill=white] (0-1.2*\halfradius,0) rectangle (1.2*\halfradius,\halfradius/2);
\draw[dashed] (0,0) ellipse (\halfradius*1 and \halfradius/3);
\end{scope}

\begin{scope}[yscale=-1]
\draw (-\p-\u,\v) .. controls (-\p-\u,\v-\b) and (-\u,-\v+\b) .. (-\u,-\v);
\draw (\p+\u,\v) .. controls (\p+\u,\v-\b) and (\u,-\v+\b) .. (\u,-\v);
\draw (-\p+\u,\v) .. controls (-\p+\u,\v-\c) and (\p-\u,\v-\c) .. (\p-\u,\v);
\end{scope}

\begin{scope}[shift={(0,\v)}]
\def\halfradius{\u};
\draw (0,0) ellipse (\halfradius*1 and \halfradius/3);
\end{scope}

\end{scope}

\begin{scope}[shift={(-2*\p,\v)}]
\def\halfradius{\u};
\draw (0,0) ellipse (\halfradius*1 and \halfradius/3);
\draw[color=black,fill=black] (-\u,0) circle (\markradius);
\end{scope}
\begin{scope}[shift={(-2*\p,0-\v)}]
\def\halfradius{\u};
\draw (0,0) ellipse (\halfradius*1 and \halfradius/3);
\draw[white, fill=white] (0-1.2*\halfradius,0) rectangle (1.2*\halfradius,\halfradius/2);
\draw[dashed] (0,0) ellipse (\halfradius*1 and \halfradius/3);
\draw[color=black,fill=black] (-\u,0) circle (\markradius);
\end{scope}
\draw (-2*\p-\u,\v) -- (-2*\p-\u,-\v);
\draw (-2*\p+\u,\v) -- (-2*\p+\u,-\v);

\end{scope}

\end{tikzpicture}

%% file: ss.tex
In this last section, we note that one can define a spectral sequence
from Khovanov homology (or reduced Khovanov homology) by using the
cobordism maps as higher differentials. We prove that
the spectral sequence is invariant under Reidemeister moves and that
it is therefore a link invariant up to isomorphism.

This spectral sequence has also been studied independently by Baldwin, Hedden, and
Lobb~\cite{BHL}. They use the machinery of ``Khovanov-Floer theories''
to prove invariance. Their work also implies that the spectral sequence is functorial
with respect to link cobordisms in $\R^3 \times I$.

We first recall the definition of a cube of resolutions for a link diagram.

\begin{definition}
\label{def:resolution}
Let $L$ be a marked link in $\R^3$, and let~$D$ be a marked diagram of~$L$.
A \emph{resolution} of~$D$ is a map~$u$ from the set of crossings of~$D$
to $\left\{0,1\right\}$.
If $u$ is a resolution, then we denote by~$L_u$ the marked unlink
obtained by smoothing each crossing of~$D$ according to the conventions
in Figure~\ref{fig:smoothings}, which we also call \emph{resolution}.
If $u$ and $v$ are two resolutions of~$D$, then we say that $u < v$ if $u \neq v$ and
$u(c) \le v(c)$ for every crossing~$c$ of~$D$.
\end{definition}

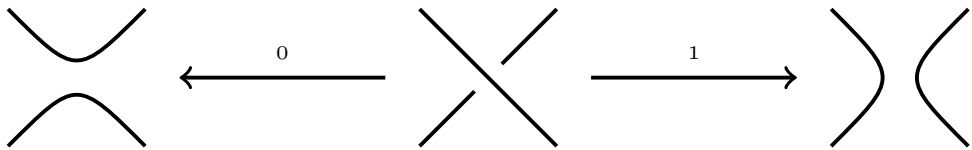
\begin{figure}
\resizebox{0.95\textwidth}{!}{\input{smoothings.tex}}
\caption{The figure shows the $0$-smoothing and the $1$-smoothing
for a crossing of an unoriented link.}
\label{fig:smoothings}
\end{figure}

\begin{definition}
\label{def:cube}
Let $u$ and $v$ be resolutions of the marked link diagram~$D$.
If $u<v$ and $u$ and $v$ only differ at a single crossing, then
$L_v$ is obtained from $L_u$ by a pair-of-pants cobordism,
which we denote by $\mc G_{u,v}$.

If $u<v$ and $u$ and $v$ differ at $k$ crossings, then choose a
sequence $u=w_0<\ldots<w_k=v$, where $w_i$ and $w_{i+1}$ differ
only at one crossing. Let $\mc G_{u,v}$ denote the composition
$\mc G_{w_{k-1},w_k} \circ \ldots \circ \mc G_{w_{0},w_1}$.
This is independent of the choice of intermediate
resolutions, up to equivalence.

We call the set of all resolutions $L_u$ of $D$,
together with the cobordisms $\mc G_{u,v}$
such that $u < v$ and $u$ and $v$ only differ at a single crossing,
the \emph{cube of resolutions} of the marked diagram $D$.

We call the set of all resolutions $L_u$ of $D$,
together with the cobordisms $\mc G_{u,v}$ for $u < v$,
the \emph{full cube of resolutions} of the marked diagram $D$.
\end{definition}

We now apply the TQFT $\HFLh$ defined in Section~\ref{sec:TQFT}
to the full cube of resolutions of a diagram of a link~$L$ in~$\R^3$.

\begin{definition}
\label{def:rKhcomplex}
Let $L$ be a marked link in $\R^3$, together with a diagram $D$. We define
a complex associated to its full cube of resolutions as follows.
Given a resolution~$u$, let $P_u$ be a decoration of~$L_u$ that
consists of two points on each component of~$L_u$.
Since $(L_u, P_u)$ and $(U_{|L_u|}, P_{|L_u|})$ are diffeomorphic
as marked decorated links, we have an isomorphism
\[
\HFLh(L_u, P_u) \cong \HFLh(U_{|L_u|}, P_{|L_u|}).
\]
To every resolution $L_u$, we associate the vector space
\[
\td C_u = \HFLh(L_u, P_u). 
\]
Let $\td C$ be the $\F_2$ vector space
\[
\td C = \bigoplus_{\text{ all resolutions }u} \td C_u.
\]
If $u < v$, let $\td\de_{u,v} \colon \td C_u \to \td C_v$ denote the map
induced by the cobordism $\mc G_{u,v}$ on $\HFLh$. Let the map
$\td\de$ be
\[
\td\de = \sum_{u<v} \td\de_{u,v}.
\]
\end{definition}

\begin{lemma}
The map $\td\de$ is a differential; i.e., $\td\de^2 = 0$.
\end{lemma}

\begin{proof}
The proof is immediate as we are working with $\F_2$ coefficients.
Notice that
\[
\td\de^2 = \sum_{w<z} \td\de_{w,z} \circ \sum_{u<v} \td\de_{u,v} =
\sum_{u<v=w<z} \td\de_{w,z} \circ \td\de_{u,v} =
\sum_{u<v=w<z} \td\de_{u,z} =
\sum_{u<z} \sum_{\substack{v\\u<v<z}} \td\de_{u,z}.
\]
The set $\left\{\,v\,\colon\, u<v<z\,\right\}$ has cardinality
$2^k-2$, where $k$ is the number of crossings
where $u$ is different from $z$. Since we are working over $\F_2$,
we immediately obtain that $\td\de^2=0$.
\end{proof}

The complex $(\td C,\td\de)$ is also filtered, as we explain in the next definition.

\begin{definition}
\label{def:rKhfilteredcomplex}
Let $L$ be a marked link in $\R^3$, together with a marked diagram $D$.
Consider the complex $(\td C, \td\de)$ from Definition~\ref{def:rKhcomplex}.
For $i \in \Z$, let
\begin{equation*}
  \td C_i =
  \bigoplus_{\substack{\mbox{\scriptsize{resolutions $u$ with}}\\
                     \mbox{\scriptsize{$i$ $0$-smoothings}}
                     }}
        \td C_u,
\end{equation*}
and define the filtration
\begin{equation}
\label{eq:filtration}
\mc F_p (\td C) = \bigoplus_{i \le p} \td C_i.
\end{equation}
Then $\left(\bigoplus_{i \in \Z} \td C_i, \td\de\right)$
is a graded filtered complex, which we call the \emph{reduced Khovanov filtered
complex} associated to the marked diagram $D$.
\end{definition}

Strictly speaking, the complex $(\td C, \td\de)$ is not graded,
because the differential does not respect the grading. However,
we use the term \emph{graded filtered} to underline the fact that
the filtration comes from a grading on $\td C$; cf.~equation~\eqref{eq:filtration}.

As explained for instance by McCleary~\cite{mccleary}, a filtered complex
yields a spectral sequence whose $E^2$ page is the homology of
the associated graded complex. Hence, by Corollary~\ref{cor:TQFT},
in our case the $E^2$ page is the reduced Khovanov homology of~$L$.

\begin{remark}
The definition of the reduced Khovanov filtered complex actually does not require
any Floer theory. It can also be defined using the formalism of
Bar-Natan~\cite{bar2002khovanov} or Khovanov~\cite{kh}.
The vector space $\td C_u$ can be defined as
\[
\td C_u = V^{\otimes |L_u|} /\left(\langle v_- \rangle \otimes V^{\otimes (|L_u| - 1)}\right),
\]
where $v_-$ is associated to the marked component of $L_u$.
The cobordism maps are given by the reduced Khovanov TQFT, which
we can denote either by $\rKh$ or $\HFLh$, and whose definition does not
require any Floer theory.
\end{remark}

In a similar way as above, we can also define an unreduced \emph{Khovanov
filtered complex} for a link diagram by applying the usual Khovanov
TQFT~\cite{bar2002khovanov} to the full cube of resolutions in
Definition~\ref{def:cube}.

\begin{definition}
\label{def:Khcomplex}
Let $L$ be a link in $\R^3$, together with a diagram $D$. We define
a complex associated to its full cube of resolutions as follows.
To every resolution $L_u$, associate the vector space $C_u \cong V^{\otimes{|L_u|}}$,
where $|L_u|$ denotes the number of components of $L_u$.
Let $C$ be the $\F_2$ vector space
\[
C = \bigoplus_{\text{all resolutions }u} C_u.
\]
If $u < v$, let $\de_{u,v}:C_u \to C_v$ denote the map
induced by applying the Khovanov TQFT to the cobordism $\mc G_{u,v}$.
The map
\[
\de= \sum_{u<v}\de_{u,v}
\]
is a differential on $C$; i.e., $\de^2 = 0$.
\end{definition}

Again, the complex $(C,\de)$ comes with a filtration.

\begin{definition}
\label{def:Khfilteredcomplex}
Let $L$ be a link in $\R^3$, together with a diagram $D$.
Consider the complex $(C, \de)$ from Definition~\ref{def:Khcomplex}.
For $i \in \Z$, let
\begin{equation*}
  C_i =
  \bigoplus_{\substack{\mbox{\scriptsize{resolutions $u$ with}}\\
                     \mbox{\scriptsize{$i$ $0$-smoothings}}
                     }}
        C_u,
\end{equation*}
and define the filtration
\begin{equation*}
\mc F_p (C) = \bigoplus_{i \le p} C_i.
\end{equation*}
Then $\left(\bigoplus_{i \in \Z} C_i, \de\right)$
is a graded filtered complex, which we call the \emph{Khovanov filtered
complex} associated to the diagram $D$.
\end{definition}

The $E^2$ page of the spectral sequence arising from the Khovanov
filtered complex is by definition the Khovanov homology of $L$.

It follows from the work of Baldwin, Hedden, and Lobb~\cite{BHL} that the spectral
sequences above are (marked) link invariants. We give here an
elementary proof of the fact that the spectral sequences
are (marked) link invariants \emph{up to isomorphism}, based on the
proof of the invariance of Khovanov homology under Reidemeister moves~\cite{bar2002khovanov}.

\begin{theorem} \label{thm:ss}
The spectral sequence defined by the Khovanov filtered complex
is an invariant of the link up to isomorphism.
\end{theorem}

\begin{proof}
What we need to check is that for each Reidemeister move (R1), (R2), and (R3),
as defined in~\cite{bar2002khovanov}, there is a morphism of filtered
chain complexes that induces an isomorphism on all pages.

Let $(C, \de)$ denote the Khovanov filtered
complex. The grading allows us to split the differential as
\[
\de = \de^h_1 + \de^h_2 + \ldots,
\]
where $\de^h_i$ denotes the component of the differential $\de$ which is
homogeneous of degree~$i$.

It is known that one can cancel a horizontal arrow in a graded filtered complex
without changing the filtered chain homotopy type of the complex;
see for example Krcatovich~\cite{Krcatovich} or Hedden and Ni \cite[Lemma~4.2]{hedden2013khovanov}.
The following lemma can be seen as a generalisation of this cancellation result.

\begin{lemma}
\label{lem:cancellation}
Let $(C = \bigoplus_{i \in \Z} C_i, \de)$ be a graded filtered complex.
Suppose that $b$ and $c$ are homogeneous elements of $C$, and that $C'' \subset C$
is a vector subspace spanned by a homogeneous basis, such that
\begin{itemize}
\item as vector spaces, $C = C' \oplus C''$, where $C' = \langle b,c \rangle$,
\item $\mc F(b) - \mc F(c) = \mc F(b) - \mc F(\de b)$, and denote this difference by~$k$,
\item if $\langle\de a, c \rangle =1$ for $a \in C''$, then $\mc F(a) -\mc F(c) \geq k$, and
\item $\textnormal{pr}_{C'}(\de b) = c$.
\end{itemize}
Let $\pi \colon C \to \quotient{C}{C'}$ and $i \colon \quotient{C}{C'} \to C''\subseteq C$
be the projection and inclusion maps. Let $h \colon C \to C$ denote the map such that $h|_{C''} = 0$,
$h(b)=0$, and $h(c)=b$, and let $\bar \de = \pi \circ (\de + \de h \de) \circ i$.
Then $(\quotient{C}{C'}, \bar\de)$ is graded filtered,
and the map
\[
f:= \pi \circ (\id + \de h) \colon C \to \quotient{C}{C'}
\]
induces isomorphisms $E^r_{p,q}(C) \cong E^r_{p,q}(\quotient{C}{C'})$ for all $r > k$.
\end{lemma}

By $\langle x, b\rangle$ and $\langle x, c\rangle$ we mean the unique elements of $\mb F_2$ such that
\[
\textnormal{pr}_{C'}(x) = \langle x, b\rangle b + \langle x, c\rangle c.
\]

\begin{proof}
The fact that $(\quotient{C}{C'}, \de)$ is graded filtered follows
from \cite[Lemmas~4.1 and~4.2]{hedden2013khovanov}.
The grading on $\quotient{C}{C'}$ is defined by
\[
\left(\quotient{C}{C'}\right)_i = \quotient{C_i}{C'_i}.
\]

Under the hypotheses of the lemma, the maps $f$ and
\[
g := (\id + h \de) \circ i \colon \quotient{C}{C'} \to C'' \subseteq C
\]
are filtered maps of degree 0, and they commute with the differential.
Therefore, they induce morphisms of spectral sequences.

Since $h^2=0$, $h \circ i=0$, and $\pi \circ h=0$,
it is straightforward to check that $f \circ g = \id_{\quotient{C}{C'}}$.
Therefore, we focus on the map $g \circ f$,
and prove that it induces the identity map on all pages $E^r$ with $r > k$.
Recall that the pages of the spectral sequence are defined as follows:
\[
E^r_{p,q} = \frac{Z^r_{p,q}}{Z^{r-1}_{p-1, q+1} + B^{r-1}_{p, q}}, \text{ where}
\]
\begin{align*}
Z^r_{p,q} &= \mc F_p(C_{(p+q)}) \cap \de^{-1} (\mc F_{p-r}(C_{(p+q-1)})),\\
B^r_{p,q} &= \mc F_p(C_{(p+q)}) \cap \de(\mc F_{p+r}(C_{(p+q+1)})).
\end{align*}
The grading on $C$ within brackets is the homological grading. We use this notation to distinguish it from the grading that yields the filtration. In the rest of the proof we will drop the homological grading to keep the notation simple.

Let $x = a + \lambda b + \mi c \in C$, with $a \in C''$ and $\lambda$, $\mi \in \mb F_{2}$.
Suppose that $x \in Z^r_p$. We claim that $g \circ f (x) + x \in Z^{r-1}_{p-1} + B^{r-1}_{p}$.
A computation shows that
\[
g \circ f (x) + x = \mi \de b + \langle \de x, c \rangle b.
\]
If $\mi \neq 0$, then $p \geq \mc F(c)$, and therefore $b \in \mc F_{p+k}(C)$.
Since $r > k$, we have $\de b \in B^{r-1}_p$.
For the term $\langle \de x, c \rangle b$, suppose that $\langle \de x, c \rangle \neq 0$.
Then $p-r \geq \mc F(c)$, since $x \in Z^r_p$.
Thus $p > \mc F(c) + k = \mc F(b)$, so $b \in \mc F_{p-1}(C)$,
and therefore $\langle \de x, c \rangle b \in Z^{r-1}_{p-1}$.
\end{proof}

The following is an analogue of \cite[Lemma~3.7]{bar2002khovanov}.
We say that a subcomplex $C' \subseteq C$ of a graded filtered complex is
a \emph{graded filtered subcomplex} if $C' = \bigoplus_{i \in \Z} C'_i$,
where $C'_i \subseteq C_i$, and $\de C' \subseteq C'$.

\begin{lemma}
\label{lem:bn}
Let $C' \subseteq C$ be a graded filtered subcomplex of a graded filtered complex $(C, \de)$,
and suppose that $\de^h_0 =0$ and $\H^*(C', \de^h_1) = 0$ (we say that $C'$ is \emph{$\de^h_1$-acyclic}).
Then $\quotient{C}{C'}$ is also a graded filtered complex, and the quotient map
$C \to \quotient{C}{C'}$ induces an isomorphism on all pages $E^r$ with $r \geq 2$.
\end{lemma}

\begin{proof}
We argue by induction on $\dim (C')$. If this is $0$, then there is nothing to prove. Suppose that $\dim (C') = 2m >0$.
Let
\[
s := \min\left\{\,p\,\colon\,C_p' \neq 0 \,\right\},
\]
and choose $c \in C_s'$. By $\de^h_1$-acyclicness, there exists a $b \in C_{s+1}'$ such that $\de^h_1(b) = c$.
Since $C'$ is a subcomplex, this also implies that $\de b = c$.
We now apply Lemma~\ref{lem:cancellation} with $k=1$.
Since $\de b = c$, we have that $\pi \circ (\id+ \de h)=\pi$,
and therefore the differential $\bar \de$ is induced by $\de$ on the quotient.
Thus, the quotient map $\pi$ induces an isomorphism on all pages $E^r$ with $r > 1$.
Furthermore, $\pi(C')$ is still $\de^h_1$-acyclic,
because $\pi$ is a $\de^h_1$ chain homotopy equivalence
by applying Lemma~\ref{lem:cancellation} to $(C', \de^h_1)$ with $r=2$.

To conclude the proof, note that $\dim(\pi(C')) = 2(m-1)$,
and that the composition of all the quotient maps gives the quotient map $C \to \quotient{C}{C'}$.
\end{proof}

We now prove invariance under Reidemeister moves.
\renewcommand{\qedsymbol}{}
\end{proof}

\subsubsection*{Invariance under (R1)}
Consider Bar-Natan's proof in \cite[Section~3.5.1]{bar2002khovanov}.
The subcomplex $\mc C'$ defined by Bar-Natan is actually a graded filtered
subcomplex of $\mc C$. The fact that $\mc C'$ is $\de^h_1$-acyclic was
already observed there by Bar-Natan.

By Lemma~\ref{lem:bn}, we can consider the graded filtered complex
$\quotient{\mc C}{\mc C'}$, which is isomorphic (up to a grading shift)
to the graded filtered complex $\llbracket D \rrbracket$, where $D$ is the
diagram of the link without the curl arising from the Reidemeister move (R1).
As in the proof by Bar-Natan, the grading shift is adjusted by defining the
complex $\mc C(L)$ from $\llbracket L \rrbracket$.

\subsubsection*{Invariance under (R2)}
We follow the second proof of the invariance under (R2), from
\cite[Sections~3.5.2 and~3.5.4]{bar2002khovanov}.
The complex $\mc C'$ is again a graded filtered subcomplex. Bar-Natan
proves that it is $\de^h_1$-acyclic, so, by Lemma~\ref{lem:bn}, we
can focus on $\quotient{\mc C}{\mc C'}$.

What we need to check next is that the subcomplex $\mc C''' \subseteq \quotient{\mc C}{\mc C'}$,
as defined by Bar-Natan, is a $\de^h_1$-acyclic graded filtered complex.
The map $d_{\star0} \circ \Delta^{-1}$ is homogeneous, so $C'''$ is
a graded \emph{subspace} of $\quotient{C}{C'}$.
Bar-Natan already proved that $C'''$ is $\de^h_1$-acyclic, so
we only need to prove that it is closed under the differential
$\de$. We need to check this on the elements of the form $\alpha$
or $(\beta, \tau(\beta))$, using the notation of \cite[Section~3.5.4]{bar2002khovanov}.
\begin{itemize}
\item{Consider an element of the form $\alpha$ in $\mc C'''$.
Any component $\mc G$ of the differential~$\de$ that does not contain
the maps $\Delta$ and $d_{\star0}$ sends $\alpha$ to some $\alpha' \in \mc C'''$.
On the other hand, the other components of $\de$ are sums of
maps that can be written as $(\Delta + d_{\star0}) \circ \mc G$,
where $\mc G$ is the composition of some other edge maps. As
already seen, $\mc G (\alpha) = \alpha' \in \mc C'''$, and the
fact that $(\Delta + d_{\star0})(\alpha') \in \mc C'''$ was
already observed by Bar-Natan.}
\item{We now show that $\de(\beta,\tau(\beta)) \in \mc C'''$.
We prove that the subspace $\left\{(\beta, \tau(\beta))\right\}$
is invariant under any edge map. If $\mc G$ is an edge map, we need
to check that $\mc G \circ \tau = \tau \circ \mc G$.
Since $\Delta$ is an isomorphism, this is equivalent to proving
that $\mc G \circ \tau \circ \Delta = \tau \circ \mc G \circ \Delta$.
The fact that $\Kh$ is a TQFT implies that the edge maps commute,
so we have
\begin{align*}
\mc G \circ \tau \circ \Delta &= \mc G \circ (d_{\star0} \circ \Delta^{-1}) \circ \Delta \\
&= \mc G \circ d_{\star0} \\
&= d_{\star0} \circ (\Delta^{-1} \circ \Delta) \circ \mc G \\
&= (d_{\star0} \circ \Delta^{-1}) \circ \mc G \circ \Delta \\
&= \tau \circ \mc G \circ \Delta.
\end{align*}
}
\end{itemize}

\subsubsection*{Invariance under (R3)}
In the same way as we did for (R2), one can prove that the
complexes $\mc C'$ and $\mc C'''$ are also $\de^h_1$-acyclic
graded filtered subcomplexes in this case.
To conclude the proof, it is sufficient to note that, since the
map $\Upsilon$ defined by Bar-Natan commutes with all edge maps,
it also commutes with $\de$, so it is an isomorphism of graded
filtered chain complexes. \qed

\subsection{Final remarks}

The spectral sequence arising from the Khovanov filtered complex is an instance of
a \emph{Khovanov-Floer theory}, as defined by Baldwin, Hedden, and Lobb~\cite{BHL}.
Their theory not only implies that the spectral sequence is a link invariant up to
isomorphism, but also achieves naturality and functoriality under link cobordisms.

A result analogous to Theorem~\ref{thm:ss} also holds for marked links and the reduced
Khovanov filtered complex.
The proof is obtained by adapting the proof of Theorem~\ref{thm:ss}. First, notice that
in order to go from a pointed diagram $D$ to another pointed diagram $D'$ of a marked
link $L$, it is sufficient to use Reidemeister moves that do not cross the basepoint:
As explained by Hedden and Ni~\cite[page 3032]{hedden2013khovanov}, when one encounters a Reidemeister
move that crosses the basepoint, one can trade it for a sequence of Reidemeister moves
that do not cross it by pulling the string in the other direction, and letting it pass
over the point at infinity.
Recall that the reduced Khovanov complex is defined by quotienting the Khovanov complex
by $\langle v_-\rangle$, where the $v_-$ is associated to the marked circle in each resolution.
In order to check that the Reidemeister moves that do not cross the basepoint induce
filtered chain isotopies between the reduced Khovanov complexes, we need to check the
following two things:
\begin{itemize}
\item{the maps of filtered complexes defined in the sections on the invariance under
Reidemeister moves send $\langle v_- \rangle$ to $\langle v_- \rangle$;}
\item{if $\td{\mc C'}$ and $\td{\mc C'''}$ denote the quotients of the complexes $\mc C'$
and $\mc{C'''}$ by $\langle v_-\rangle$, we need to check that $\td{\mc C'}$ and
$\td{\mc C'''}$ are $\td\partial^h|_1$-acyclic.}
\end{itemize}
Both points can be checked to be true for any Reidemeister move that happens away from
the basepoint.

The spectral sequences above split along the quantum grading,
because all the differentials that we consider are $q$-homogeneous.
So, actually, one gets a spectral sequence for each $q$-grading,
and this is a (marked) link invariant.

It is straighforward from the definition of the (reduced) Khovanov
filtered complex that the spectral sequences above are trivial on
$\Kh$-thin knots. The page $E^2$ is indeed supported along two
diagonals in the $(p,q)$ grading, so there is no room for non-trivial
higher differentials.

With $\F_2$ coefficients, the spectral sequence associated to the
reduced Khovanov complex does not depend on the basepoint (the same
happens for the usual reduced Khovanov homology).
If $u$ is a resolution of the diagram $D$, let $V(u)$ be the free
vector space generated by the circles in the resolution. Call the
obvious generators $X_1, \ldots, X_n$. (Here, it is more
convenient to use Khovanov's notation~\cite{kh}. The vector $X_i$ can be
translated to Bar-Natan's notation~\cite{bar2002khovanov} as the $v_-$ corresponding to
the $i$-th circle of $u$.) At the $u$-vertex of the Khovanov complex,
one finds the vector space $\bigwedge V(u)$. (In Bar-Natan's notation,
$X_1 \wedge X_2 \wedge \cdots \wedge X_j$ corresponds to
$v_-^{(1)} \otimes \cdots \otimes v_-^{(j)} \otimes v_+^{(j+1)} \otimes \cdots \otimes v_+^{(n)}$.)
If $W$ is the subspace of $V$ spanned by the differences (or sums) of the
generators, then the subcomplex $\bigwedge W \subseteq \bigwedge V(u)$
is isomorphic to the usual reduced Khovanov complex $\bigwedge V(u)/\langle X_n \rangle$
(here, we assume that the marked component is the $n$-th one.)
The isomorphism $\bigwedge V(u)/\langle X_n \rangle \rightarrow \bigwedge W$
is given by
\[
X_{i_1} \wedge \cdots \wedge X_{i_k} \mapsto (X_{i_1} - X_n) \wedge \cdots \wedge (X_{i_k} - X_n).
\]
This isomorphism preserves the filtration and commutes with the
edge maps, so it induces an isomorphism of spectral sequences.
This point of view has been communicated to us by Andrew Lobb, and
was firstly observed by Ozsv\'ath and Szab\'o~\cite{bdcs}.

The fact that the spectral sequence is an invariant implies
that, in particular, each page of the spectral sequence is a
bigraded vector space $\Kh_{p,q}^r$ that is a link invariant
for every $r \in \{\,2,3,\ldots,\infty\,\}$.
The abutment of this spectral sequence is unknown. A trivial
lower bound to the dimension of $\Kh_{*,*}^\infty$ is the number
of non-zero coefficient of the Jones polynomial, but this is
a rough bound.
We are aware that Baldwin and Lobb ran a simulation to explore
the limit of this spectral sequence, but could not find
any higher differentials. We therefore conclude with their
conjecture, based on their data.

{\renewcommand{\theproposition}{\ref{conj}}
\begin{conjecture}[{Baldwin-Lobb}]
The limit $\Kh_{p,q}^\infty$ of the spectral sequence is
Khovanov homology $\Kh_{p,q}=\Kh_{p,q}^2$.
\end{conjecture}
\addtocounter{proposition}{-1}
}

The referee pointed out that the second page and third pages
of the spectral sequence are isomorphic.
Unfortunately, we were not able to generalise their proof to show the above conjecture.
 We now outline the referee's proof.

\begin{lemma}
\label{lem:E2}
Let $(C = \bigoplus_{p \in \Z} C_p, \de)$ be a graded filtered complex such that $\de^h_0 = 0$. Then
\[
E^3_{p,q} = \H_q(E^2_{p, *}, [\de^h_2]).
\]
\end{lemma}

\begin{proof}
We drop the homological grading $q$ to simplify the notation.
By \cite[Theorem~2.6]{mccleary}, we have $E^3_p = \H_*(E^2_p, [\de])$, where the map
\[
[\de] \colon E^2_p = \frac{Z^2_p}{Z^1_{p-1} + B^1_p} \to \frac{Z^2_{p-2}}{Z^1_{p-3} + B^1_{p-2}} = E^2_{p-2}
\]
is induced by the differential $\de$. Since $\de^h_0 = 0$, we have $Z^1_{p-1} = \mc F_{p-1}(C)$,
therefore any element $[x] \in E^2_p$ can be represented by some homogeneous element
$x \in C_p \cap Z^2_p$. As $x \in Z^2_p$, we have that
\[
[\de] \left([x]\right) = [\de^h_2(x) + \de^h_3(x) + \ldots] = [\de^h_2(x)],
\]
where the last equality follows from the fact that $Z^1_{p-3} = \mc F_{p-3}(C)$. Thus
\[
E^3_p = \H_*(E^2_p, [\de]) = \H_*(E^2_p, [\de^h_2]). \qedhere
\]
\end{proof}

\begin{proposition}
For every knot $K$, we have $\Kh^2_{p,q}(K) \cong \Kh^3_{p,q}(K)$.
\end{proposition}

\begin{proof}
Let $C = \bigoplus_{r \in \{0,1\}^n} C_r$. Given a resolution $r$ and an integer $i \in [1,n]$,
denote by $r+e_i$ the function where $e_i(j) = \delta_{ij}$. Define
\[
\xi_{r,i} =
\begin{cases}
0 & \text{if $r(1) + \dots + r(i-1)$ is even, and}\\
1 & \text{otherwise,}
\end{cases}
\]
and let $H \colon C \to C$ be the map defined by the formula
\[
H(a_r) = \sum_{i=1}^n \de_{r, r + e_i}(a_r) \cdot \xi_{r,i}
\]
for $a_r \in C_r$, where we set $\de_{r, r + e_i} = 0$ if $r + e_i$
is not a resolution; i.e., if $(r + e_i)(i)=2$.

Then $\de^h_2$ is a chain map from the complex $(C, \de^h_1)$ to itself,
and it is null-homotopic, because a computation shows that
\[
\de^h_2 = H \de^h_1 + \de^h_1 H.
\]
In view of Lemma~\ref{lem:E2}, we have $\Kh^2_{p,q} \cong \Kh^3_{p,q}$.
\end{proof} 

%% file: smoothings.tex
\begin{tikzpicture}[cross line/.style={preaction={draw=white, -, line width=8pt}}]

\def\u{0.5cm};
\def\cdx{6*\u};

\begin{scope}[shift={(0,0)}]
\draw[thick] ( -\u,-\u) -- (\u,\u);
\draw[cross line, thick] (\u,-\u) -- (-\u,\u);
\end{scope}

\begin{scope}[shift={(\cdx,0)}]
\draw[thick, cross line, -] (-\u,-\u) .. controls (0,0) and (0,0) .. (-\u,\u);
\draw[thick, cross line, -] (\u,\u) .. controls (0,0) and (0,0) .. (\u,-\u);
\end{scope}
\draw[->, thick] (1.5*\u,0) -- (4.5*\u,0);
\draw (3*\u,0) node [anchor=south] {\fontsize{4pt}{4pt}\selectfont $1$};

\begin{scope}[shift={(-\cdx,0)}]
\draw[thick, cross line, -] (-\u,-\u) .. controls (0,0) and (0,0) .. (\u,-\u);
\draw[thick, cross line, -] (\u,\u) .. controls (0,0) and (0,0) .. (-\u,\u);
\end{scope}
\draw[->, thick] (-1.5*\u,0) -- (-4.5*\u,0);
\draw (-3*\u,0) node [anchor=south] {\fontsize{4pt}{4pt}\selectfont $0$};

\end{tikzpicture}